%% file: schrodinger_arxiv18102022.tex
\documentclass[a4paper,reqno]{amsart}

\input{commands.tex}      

\usepackage{mathtools}
\mathtoolsset{showonlyrefs}

\numberwithin{equation}{section}
\numberwithin{figure}{section}

\begin{document}
\title[Resolvent estimates for Schr\"{o}dinger operators]{Resolvent estimates for one-dimensional Schr\"{o}dinger operators with complex potentials}

\author{Antonio Arnal}

\address[Antonio Arnal]{Mathematical Sciences Research Centre, Queen's University Belfast, University Road, Belfast BT7 1NN, UK}

\email{aarnalperez01@qub.ac.uk}

\author{Petr Siegl}

\address[Petr Siegl]{Institute of Applied Mathematics, Graz University of Technology, Steyrergasse 30, 8010 Graz, Austria}

\email{siegl@tugraz.at}

\thanks{The authors would like to express their gratitude for helpful comments from an anonymous referee.}



\date{\today}

\begin{abstract}
	We study one-dimensional Schr\"odinger operators $\opH = \Dt + V$ with unbounded complex potentials $V$ and derive asymptotic estimates for the norm of the resolvent, $\Psi(\la) := \| (\opH - \la)^{-1} \|$, as $|\la| \to +\infty$, separately considering $\la \in \Ran V$ and $\la \in \Rplus$. In each case, our analysis yields an exact leading order term and an explicit remainder for $\Psi(\la)$ and we show these estimates to be optimal. We also discuss several extensions of the main results, their interrelation with some aspects of semigroup theory and illustrate them with examples.
\end{abstract}

\maketitle

\section{Introduction}
\label{sec:intro}

The structure of the pseudospectrum of non-self-adjoint operators can be very non-trivial and in general unrelated to the location of the spectrum. This fact is well-known to be responsible for typical non-self-adjoint effects such as spectral instabilities or long-time semigroup bounds unrelated to the spectrum, see~\eg~  \cite{Trefethen-2005,Davies-2000-43,Davies-2007,Helffer-2013-book} for details.

For Schr\"odinger operators $H=-\Delta +V$ with complex potentials $V$, the pseudospectral analysis was initiated in the seminal paper of E.~B.~Davies, \cf~\cite{Davies-1999-200}, where lower estimates for the resolvent norm inside the numerical range of $H$, $\Num(H)$, were obtained by a semi-classical pseudomode construction. The latter was subsequently generalised: in the semi-classical case in particular in \cite{Zworski-2001-129,Dencker-2004-57} and in the non-semi-classical one in \cite{Krejcirik-2019-276,Arifoski-2020-52,KREJCIRIK2022109440,Duc-2022}.

The upper estimates of the resolvent norm at the boundary of $\Num(H)$ were first obtained by L.~Boulton in \cite{Boulton-2002-47} for the quadratic potential. This work was followed up with several semi-classical generalisations in particular in \cite{Pravda-Starov-2006-73,Dencker-2004-57,BordeauxMontrieux-2013,Sjoestrand-2009,Bellis-2018-9,Bellis-2019-277,Henry-2014,Almog-2016-48} and also in \cite{Dondl-2016} based on semigroup compactness or known behaviour of spectral projections.

In this paper, we study the behaviour of the resolvent norm at the boundary of $\Num(H)$ for \emph{non-semi-classical} one-dimensional Schr\"{o}dinger operators acting in $L^2(\Rplus)$ or in $L^2(\R)$ for a wide class of unbounded complex potentials $V$ ranging from iterated log functions to super-exponential ones (which are not accessible by previously used methods). 

Our assumptions on $V$ are compatible with those in \cite{Krejcirik-2019-276} where lower resolvent norm estimates inside $\Num(H)$ were obtained. More precisely, restricting ourselves in this section to purely imaginary $V$, we assume that $\Im V$ is eventually increasing, unbounded at infinity and that the conditions (reflecting the growth of $\Im V$)
\begin{equation}\label{eq:nu.intro}
	\Im V'(x) = \BigO(\Im V(x) x^\nu), \quad \Im V''(x) = \BigO(\Im V'(x) x^\nu),  \qquad x \to + \infty,
\end{equation}
with some $\nu \geq -1$, are satisfied, see Assumption~\ref{asm:iR} for details. Moreover, the condition 
\begin{equation}\label{eq:upsilon.intro}
	\Upsilon(x) := \frac{x^{\nu}}{\Im V'(x)^{\frac13}} = o(1), \qquad x \to + \infty,
\end{equation}
is related to the separation property of the domain of $H$, see Sub-section~\ref{sssec:rem.asm.iR}, and the quantity $\Upsilon$ naturally enters the remainders in the derived asymptotic formulas (similarly to what happens \eg~for diverging eigenvalues in domain truncations in \cite{Semoradova-toappear} or for asymptotics of eigenfunctions in \cite{MiSiVi-2020}). 

It was established in \cite{Krejcirik-2019-276} that $\|(H-\la)^{-1}\|$ diverges as the spectral parameter $\la = a+ i b$ goes to infinity along a set of admissible curves determined by the potential. In particular, for operators in $L^2(\Rplus)$ the restriction on admissible curves is given by (with $a, b \in \Rplus$)
\begin{equation}\label{eq:curves.intro}
	b^{\frac23} x_b^{\frac23 \nu} \ls a \ls b^2 x_b^{-4 \nu - 4 \eps - 2}
\end{equation}
where $x_b>0$ is the turning point of $\Im V$, determined by $\Im V(x_b)=b$, $\nu$ is as defined in \eqref{eq:nu.intro} and $\eps > 0$ is arbitrarily small. Except for the case of monomial potentials, where scaling can be used to rewrite $H$ in semi-classical form, it was left as an open question whether the restrictions~\eqref{eq:curves.intro} are optimal. Our main results allow us in particular to answer this question in the affirmative (with additional assumptions on $V$ for the second restriction in \eqref{eq:curves.intro}, see Subsection~\ref{ssec:optimality.pseudomodes}).

Our first result (Theorem~\ref{thm:iR}), specialised for purely imaginary potentials here, provides a two-sided estimate for the norm of the resolvent along the imaginary axis for operators on the half-line and it includes an exact leading order term and an explicit remainder estimate. Namely, 
\begin{equation}\label{eq:iR.intro}
	\|(\opH - i b)^{-1} \| = \| \opA^{-1} \| \left( \Im V'(x_b) \right)^{-\frac{2}{3}} \left( 1 + \BigO \left(\Upsilon(x_b) \right) \right), \quad b \to +\infty,
\end{equation}
where $\opA = -\partial_x^2 + i x$ is the complex Airy operator in $L^2(\R)$ (see Sub-section~\ref{ssec:Airy.prelim}). In Section~\ref{sec:cor}, we further explain how these results extend to operators in $L^2(\R)$ as well as to multi-dimensional operators with radial potentials (see Sub-sections~\ref{ssec:iR.R} and \ref{ssec:rad}). Moreover, in Sub-section~\ref{ssec:s-c} we indicate how our strategy can be used in a semi-classical case where the problem substantially simplifies as only local properties of $V$ are needed (similarly to the pseudomode construction in \cite{Krejcirik-2019-276}). In Sub-section~\ref{ssec:inside.Num}, we extend Theorem~\ref{thm:iR} (with $\Re V = 0$) to describe the behaviour of the norm of the resolvent along general curves $\la_b = a(b) + i b$ inside the numerical range
\begin{equation*}
	\|(\opH - \la_b)^{-1} \| = \| (\opA - \mu_b)^{-1} \| \left( \Im V'(x_b) \right)^{-\frac{2}{3}} (1 + o(1)), \quad b \to +\infty,
\end{equation*}
with $\mu_b = a (\Im V'(x_b))^{-\frac{2}{3}}$. Precise resolvent estimates for semi-classical operators were found in \cite{BordeauxMontrieux-2013}; in the special cases of the Davies operator and the imaginary cubic oscillator our construction allows us to recover those same curves (see the discussion for power-like potentials in Sub-section~\ref{ssec:example.power.V}).

An analogous result (Theorem~\ref{thm:R}) is derived for operators in $L^2(\R)$ when $\la = a \in \Rplus$ for a smaller class of smooth, even, purely imaginary potentials $V$ satisfying $\Im V \ge 0$ and
\begin{enumerate}[\upshape (i)]
	\item $\Im V$ is eventually increasing:
	\begin{equation*}
		\exists x_0 > 0, \quad \forall x > x_0, \quad \Im V'(x) > 0;
	\end{equation*}
	\item $\Im V$ is regularly varying:
	\begin{equation*}
		\exists \beta > 0, \quad \forall x > 0, \quad 	\lim_{t \to + \infty} W_t(x) = \omega_{\beta}(x),
	\end{equation*}
	where
	\begin{equation*}
		W_t(x):= \frac{\Im V(tx)}{\Im V(t)}, \quad  \omega_{\beta}(x):=|x|^{\beta}, \quad x \in \R;
	\end{equation*}
	\item $\Im V$ has controlled derivatives:
	\begin{equation*}
		\forall n \in \N, \quad \exists C_n>0, \quad |\Im V^{(n)}(x)| \le C_n \, \left(1 + \Im V(x)\right) \, \langle x \rangle^{- n}, \quad x \in \R.
	\end{equation*}
\end{enumerate}
Under these conditions the resolvent norm of the operator
\begin{equation*}
	\opH = \Dt + V,
	\quad \Dom(H) = W^{2,2}(\R) \cap \Dom(V),
\end{equation*}
satisfies
\begin{equation*}
	\|(\opH - a)^{-1} \| = \| \opA_{\beta}^{-1} \| \, (\Im V(t_a))^{-1} \left( 1 + \BigO \left( \iota(t_a) + (a^{\frac 12} t_a)^{-l_{\beta, \eps}} \right) \right), \quad a \rightarrow +\infty,
\end{equation*}
where $\opA_{\beta}$ is the generalised Airy operator
\begin{equation}
	\opA_\beta = \Nt + |x|^\beta, \quad 
	\Dom(\opA_{\beta}) = W^{1,2}(\R) \cap \Dom(|x|^\beta),
\end{equation}
$t_a$ solves the equation $t_a \Im V(t_a) = 2 \sqrt{a}$ and $\iota(t)$ and $l_{\beta,\eps}$ are determined by
\begin{equation*}
	\iota(t) = \| (1 + W_t)^{-1} - (1 + \omega_\beta)^{-1} \|_{\infty}, \quad 
	l_{\beta, \eps} = 
	\begin{cases}
		1 - \eps, &  \beta > 1/2, \\[1mm] 
		1/2 + \beta - \eps, &  \beta \in (0,1/2].
	\end{cases}
\end{equation*}
The additional smoothness and growth restrictions on $V$ for this result stem from employing pseudo-differential operator techniques. The regular variation assumption arises naturally due to scaling (similarly to the analysis of the eigenfunctions' concentration in \cite{MiSiVi-2020}). Basic properties of generalised Airy operators are described in Appendix~\ref{sec:genAiry} and a detailed spectral analysis can be found in \cite{ArSi-generalised-2022}, including precise resolvent norm estimates.

The result \eqref{eq:iR.intro} in particular relates the behaviour of $V$ at infinity to the decay/growth of the resolvent along the imaginary axis, with the linear potential (\ie~the Airy operator) being the transition between the two cases. For sub-linear potentials, the resolvent norm diverges on the imaginary axis and the rate of divergence becomes very fast for slowly growing (\eg~iterated log) potentials (see Section~\ref{sec:examples} with several examples). The interest in such operators has been highlighted in recent research on one-parameter semigroups, \eg~\cite[Thm.~1.5]{Batty-2016-270} relates the decay of solutions of the Cauchy problem to the growth of the resolvent norm along the imaginary axis. More precisely, if $\opA$ is the generator of the bounded $C_0$-semigroup $(\opT(t))_{t \ge 0}$ and $\sigma(\opA) \cap i \R = \emptyset$, then for fixed $\alpha > 0$ we have
\begin{equation*}
	\| (\opA - is)^{-1} \| = \BigO(|s|^{\alpha}), \; |s| \to \infty \iff \| \opT(t) \opA^{-1} \| = \BigO(t^{-\frac1 \alpha}), \quad t \to \infty.
\end{equation*}
For more general rates, see \cite{stahn2018decay}. Inspired by the open problem presented by C.~Batty~\cite{CB-CIRM}, we note that Theorem~\ref{thm:iR} enables us to characterise the class of rates (\eg~$|s|^{\alpha}$) for which we can construct potentials $V$ such that the resolvent norm of the corresponding Schr\"{o}dinger operator equals that given rate (see Section~\ref{sec:bbt.iR} for details).

The proof of Theorem~\ref{thm:iR}, originally inspired by \cite[Prop.~14.13]{Helffer-2013-book}, revolves around a separate analysis of $\| (\opH - i b) u\|$ depending on whether or not $\supp u$ is contained in a neighbourhood of the turning point $x_b$ designed so that $\Im V$ is approximately constant inside. More specifically, the proof consists of the following steps (several technical extensions are additionally needed for the case of potentials with non-zero real part). 
\begin{enumerate}[\upshape (1),wide]
	\item In Proposition~\ref{prop:away.iR}, with $\Omega_b'$ representing a neighbourhood of $x_b$ chosen so that $\Im V(x) \approx \Im V(x_b)$ for $x \in \Omega_b'$ (see \eqref{eq:deltabdef}), we use direct quadratic form estimates to find that
	\begin{equation*}
		\begin{aligned}
			\frac{(\Im V'(x_b) )^\frac{2}{3}} {\Upsilon(x_b)} &= \frac{\Im V'(x_b)}{x_b^{\nu}}\\
			&\ls \inf \left\{ \frac{\left\| (\opH-ib) u \right\|}{\|u\|}: \; 0 \neq u \in \Dom(\opH), \; \supp u \cap \Omega'_b = \emptyset\right\},
		\end{aligned}
	\end{equation*}
	asymptotically as $b \to +\infty$, with $\Upsilon$ as in \eqref{eq:upsilon.intro}.
	\item \label{itm:step2.intro} In Proposition~\ref{prop:local.iR}, in a neighbourhood $\Omega_b$ of $x_b$ (see \eqref{eq:Omega.def}), appropriately  shifted and scaled, we Taylor-approximate $\opH - i b$ with the complex Airy operator $\opA$ to yield
	\begin{equation*}
		\begin{aligned}
			&\| \opA^{-1} \|^{-1} \left( \Im V'(x_b) \right)^{\frac{2}{3}} \left( 1 - \BigO \left(\Upsilon(x_b)\right) \right)\\
			& \qquad  \qquad \le \inf \left\{ \frac{\left\| (\opH-ib) u \right\|}{\|u\|}: \; 0 \neq u \in \Dom(\opH), \; \supp u \subset \Omega_b \right\},
		\end{aligned}	
	\end{equation*}
	as $b \to +\infty$. The norm resolvent convergence of (a localised realisation of) $\opH - i b$ to the complex Airy operator $\opA$ follows from the second resolvent identity and it makes use of certain graph-norm estimates introduced in Subsection~\ref{ssec:Airy.prelim}.
	\item In Proposition~\ref{prop:lbound.iR}, we show that our estimate for the norm of the resolvent of $\opH$ cannot be improved by finding functions $u_b \in \Dom(\opH)$ such that as $b \to + \infty$
	\begin{equation*}
		\| (\opH-ib) u_b \| = \| \opA^{-1} \|^{-1} \left( \Im V'(x_b) \right)^{\frac{2}{3}} 
		\left( 1 + 
		\BigO \big( \Upsilon(x_b) \big) 
		\right) \|u_b\|.
	\end{equation*}
	The proof relies on exploiting the localisation technique used in step \ref{itm:step2.intro} and the fact that the operators involved have compact resolvent. Thus the norms of those resolvents can be obtained from the appropriate singular values and the corresponding eigenfunctions are used to find the $u_b$ family.
	\item We combine the results from the previous steps with the aid of certain commutator estimates and a suitably constructed partition of unity.
\end{enumerate}

The proof of Theorem~\ref{thm:R}, which describes the asymptotic behaviour of the resolvent norm along the real axis, follows the template outlined above but on the Fourier side and with substantial modifications at several stages. In particular, the commutator estimates in Step 4 are obtained using pseudo-differential operator techniques (see Lemma~\ref{lem:pdo.comp}) resulting in additional smoothness and regularity assumptions.

The remainder of our paper is structured as follows. Section~\ref{sec:prelim} introduces our notation and recalls some fundamental facts for the various tools used throughout (Fourier transform, pseudo-differential operators, Schr\"{o}dinger operators with complex potentials, Airy operators and functions of regular variation). In Section~\ref{sec:iR} we formulate and prove Theorem~\ref{thm:iR} for the resolvent norm in $\Ran V$. Section~\ref{sec:R} is devoted to the proof of Theorem~\ref{thm:R} for the resolvent norm in the real line. Section~\ref{sec:cor} includes further extensions of the main theorems, in particular the resolvent estimates on more general curves in the numerical range. In Section~\ref{sec:bbt.iR} we deal with the inverse problem mentioned above and Section~\ref{sec:examples} illustrates our results on some concrete potentials. Finally, in Appendix~\ref{sec:genAiry} we show the key properties of the first order generalised Airy operators used in the proof of Theorem~\ref{thm:R}.

\section{Notation and preliminaries}
\label{sec:prelim}
We write $\N_0:=\N \cup \{0\}$, $\R_{+} := (0, +\infty)$, $\R_{-} := (-\infty, 0)$, $\C_+ := \{\la \in \C: \Re \la > 0\}$ and $\C_- := \{\la \in \C: \Re \la < 0\}$. The characteristic function of a set $E$ is denoted by $\chi_E$, the $L^2$-norm by $\| \cdot \|$, the other $L^p$ norms by $\| \cdot \|_p$, the space of smooth functions of compact support by $\CcR$ and the Schwartz space of smooth rapidly decreasing functions by $\SchwR$. The commutator of two operators $A$, $B$ is denoted by $[A,B]:=AB - BA$. For a multi-index $\alpha = (\alpha_1, \alpha_2, \ldots, \alpha_n) \in \N_0^n$, we write $|\alpha| = \alpha_1 + \alpha_2 + \ldots + \alpha_n$.

In the one-dimensional setting, we will refer  to the first and second order differential operators with $\Ntp$ and $\Dtp$, respectively, reserving the symbols $\nabla$ and $\Delta$ for statements in higher dimensions.
	
If $\cH$ denotes a Hilbert space, we shall use $\langle \cdot, \cdot \rangle_{\cH}$ and $\| \cdot \|_{\cH}$ to represent the inner product and norm on that space. The $L^2$ inner product shall be denoted by $\langle \cdot, \cdot \rangle_2$, or just by $\langle \cdot, \cdot \rangle$ if there is no ambiguity, and the $L^2$ norm by $\| \cdot \|_2$ or just by $\| \cdot \|$. The other $L^p$ norms will be represented by $\| \cdot \|_p$ with $L^{\infty}$ denoting the space of essentially bounded functions endowed with the essential $\sup$ norm $\| \cdot \|_{\infty}$.
	
Let $\emptyset \neq \Omega \subset \Rd$ be open, $k \in \N$ and $p \in [1, +\infty]$. We will denote the Sobolev spaces by $W^{k,p}(\Omega)$ and $W^{k,p}_0(\Omega)$ (the latter representing as usual the closure of $\CcOm$ in $W^{k,p}(\Omega)$, see \eg~\cite[Sub-sec.~V.3]{EE} for definitions). We shall generally be concerned with the particular cases where $\Omega = \R \text{ or } \Rplus$, $k = 1 \text{ or } 2$ and $p = 2$.

If $\opT$ is a bounded operator on a Banach space $\cX$, we will denote by $\rad(\opT)$ its spectral radius, i.e. $\rad(\opT) := \sup\{|z|: z \in \sigma(\opT)\}$ with $\sigma(\opT)$ denoting the spectrum of $\opT$. As usual, $\sigma_{p}(\opT)$ will denote the set of eigenvalues of $\opT$ and $\rho(\opT)$ its resolvent set.

If $\opT$, $\opT_b$, $b \in \Rplus$, are closed linear operators on the Banach space $\cB$, we say that $\opT_b$ converges to $\opT$ in the \textit{norm resolvent (or generalised) sense}, and we write $\opT_b \overset{nrc}{\longrightarrow} \opT$, if there exists $\la \in \C$ such that $\la \in (\underset{b \ge b_0}{\cap} \rho(\opT_b)) \cap \rho(\opT)$, for some $b_0 > 0$, and $\| (\opT_b - \la)^{-1} - (\opT - \la)^{-1} \| \to 0$ as $b \to +\infty$ (refer to \cite[Sub-sec~IV.2.6]{Kato-1966} for a detailed exposition of the concept).

If $\opH$ and $\opH_1$ are two linear operators acting in the Hilbert space $\cH$, we say that $\opH_1$ is an \textit{extension} of $\opH$, and write $\opH_1 \supset \opH$, if $\Dom(\opH_1) \supset \Dom(\opH)$ and $\opH_1 u = \opH u$ for all $u \in \Dom(\opH)$. Note that our notation covers the case $\Dom(\opH_1) = \Dom(\opH)$, \ie~the extension does \textit{not} have to be proper.

To avoid introducing multiple constants whose exact value is inessential for our purposes, we write $a \lesssim b$ to indicate that, given $a,b \ge 0$, there exists a constant $C>0$, independent of any relevant variable or parameter, such that $a \le Cb$. The relation $a \gtrsim b$ is defined analogously whereas $a \approx b$ means that $a \lesssim b$ \textit{and} $a \gtrsim b$.

\subsection{Fourier transform and pseudo-differential operators}
\label{ssec:fourier.pdo.prelim}
For $u \in \SchwR$, the Fourier and inverse Fourier transforms read (with $x,\xi \in \R$)
\begin{align*}
\sF u(\xi) := \intR e^{-i \xi x} u(x) \odd x, \qquad 
\sF^{-1} u(x) := \intR e^{i x \xi} u(\xi) \odd \xi, 
\quad \odd \cdot : = \frac{\dd \cdot}{\sqrt{2\pi}};
\end{align*}
we also use $\hat{u} := \sF u \text{ and } \check{u} := \sF^{-1} u$, and retain the same notations to refer to the corresponding isometric extensions to $\Lt(\R)$.

We recall that the Schwartz space, $\SchwR$, is endowed with the family of semi-norms
\begin{equation*}
	|f|_{k,\sS} := \underset{\alpha + \beta \le k}{\max} \, \underset{x \in \R}{\sup} \, \langle x \rangle^{\alpha} \, |\partial^{\beta}_{x} f(x)|, \quad k \in \N_0.
\end{equation*}

When introducing pseudo-differential operators in Section~\ref{sec:R}, we follow \cite[Part~I]{Abels-2011}.  Given $m \in \R$, the symbol class $\cS^m_{1,0}(\R \times \R)$ is the vector space of smooth functions $p: \R \times \R \rightarrow \C$ such that for any $\alpha, \; \beta \in \N_0$ there exists $C_{\alpha,\beta} > 0$ satisfying
\begin{equation}
\label{eq:symbolclass}
|\partial^{\alpha}_{\xi}\partial^{\beta}_{x} p(\xi, x)| \le C_{\alpha,\beta} \, \langle x \rangle^{m - \beta}, \quad (\xi,x) \in \R \times \R.
\end{equation}
This space is endowed with a natural family of semi-norms defined by
\begin{equation}
\label{eq.symbolclass.seminorm}
|p|_k^{(m)} := \underset{\alpha, \beta \le k}{\max} \, \underset{\xi, x \in \R}{\sup} \, \langle x \rangle^{-m+\beta} \, |\partial^{\alpha}_{\xi}\partial^{\beta}_{x} p(\xi, x)|, \quad k \in \N_0.
\end{equation}
Furthermore, for $m, \tau \in \R$, the space of amplitudes $\cA^m_{\tau}(\R \times \R)$ consists of the smooth functions $a: \R \times \R \rightarrow \C$ such that for any $\alpha, \; \beta \in \N_0$ there exists $C_{\alpha,\beta} > 0$ satisfying
\begin{equation}
\label{eq:amplitudeclass}
|\partial^{\alpha}_{\eta}\partial^{\beta}_{y} a(\eta,y)| \le C_{\alpha,\beta} \, \langle \eta \rangle^{\tau} \, \langle y \rangle^{m}, \quad (\eta,y) \in \R \times \R.
\end{equation}
This space is endowed with the family of semi-norms
\begin{equation}
\label{eq:amplitudes.seminorm}
|a|_{\cA^m_{\tau},k} := \underset{\alpha + \beta \le k}{\max} \, \underset{\eta, y \in \R}{\sup} \, \langle \eta \rangle^{-\tau} \, \langle y \rangle^{-m} \, |\partial^{\alpha}_{\eta}\partial^{\beta}_{y} a(\eta,y)|, \quad k \in \N_0.
\end{equation}
We associate a pseudo-differential operator with the symbol $p \in \cS^m_{1,0}(\R \times \R)$ via
\begin{equation*}
	\opOP(p) u(\xi) := \intR e^{-i \xi x} p(\xi,x) \check{u}(x) \odd x,  \quad \xi \in \R, \quad u \in \SchwR,
\end{equation*}
and it can be shown that this is a bounded mapping on $\SchwR$ (see \cite[Thm.~3.6]{Abels-2011}).

\subsection{Schr\"odinger operators with complex potentials}
\label{ssec:Schr.prelim}

Let $\emptyset \neq \Omega \subset \Rd$ be open. For a measurable function $m: \Omega \to \C$, we denote the maximal domain of the multiplication operator determined by the function $m$ as
\begin{equation}\label{eq:multiplication.domain}
\Dom(m) = \{ u \in L^2(\Omega) \, : \, mu \in L^2(\Omega)\};
\end{equation}	
the Dirichlet Laplacian in $L^2(\Omega)$ is denoted by $-\Delta_D$ and 
\begin{equation}
\Dom(\Delta_D) = \{ u \in W_0^{1,2}(\Omega) \, : \, \Delta u \in L^2(\Omega)\}.		
\end{equation}

Suppose that the complex potential $V:\Omega \rightarrow \C$, $V = V_{u} + V_{b}$, satisfies $\Re V \geq 0$ a.e.~in $\Omega$, $V_{u} \in C^1\left(\overline{\Omega}\right)$, $V_{b} \in \LiOm$ and, 
with $\eps_{\rm crit} = 2-\sqrt{2}$,
\begin{equation}\label{eq:vgrowth}
\exists\eps_\nabla \in [0,\eps_{\rm crit}), \quad  \exists M_{\nabla} \geq 0, \quad  
|\nabla V_{u}| \leq \eps_\nabla |V_{u}|^\frac 32 + M_{\nabla} \quad \text{a.e.~in~} \Omega.
\end{equation}

Under these assumptions on $V$ one can find the (Dirichlet) m-accretive realization $\opH = -\DD +V$ by appealing to a generalised Lax-Milgram theorem \cite[Thm.~2.2]{Almog-2015-40}. It is also known that the domain and the graph norm of $\opH$ separate, \ie~$\Dom(H) = \Dom(\Delta_D) \cap \Dom(V)$ and
\begin{equation}
\label{eq:Hcorelowerbound}
\| \opH u \|^2 + \| u \|^2 \gs  \| \DD u \|^2 + \| V u \|^2 + \|u\|^2, \quad u \in \Dom(\opH).
\end{equation}
Furthermore, 
\begin{equation}
\label{eq:Habstractcore}
\Core := \{ u \in \Dom(\opH): \; \supp u \text{ is bounded}\}
\end{equation}
is a core of $\opH$. 
For details see \cite{Almog-2015-40, Krejcirik-2017-221,Semoradova-toappear} and \cite{Brezis-1979-58,Kato-1978-5}, \cite[Chap.~VI.2]{EE} for cases with a minimal regularity of $V$.
	
\subsection{Airy operators}
\label{ssec:Airy.prelim}
An important class of objects in our analysis are complex Airy operators; details on the claims summarised here can be found in \cite[Ch.~14]{Helffer-2013-book} and in Section~\ref{sec:genAiry} of this paper for the more general case.

The rotated Airy operator in $L^2(\R)$ with $r > 0$ and $\theta \in (-\pi,\pi)$ is denoted by
\begin{equation}\label{eq:airyde}
\begin{aligned}
\opA_{r,\theta} = \Dt + r e^{i \theta} x, \quad  \Dom(\opA_{r,\theta}) = \WttR \cap \Dom(x).
\end{aligned}
\end{equation}
It is well-known that $\opA_{r,\theta}$ has compact resolvent, its spectrum is empty, its adjoint satisfies $\opA_{r,\theta}^*=\opA_{r,-\theta}$ and 
\begin{equation}
\label{eq:airyinequality}
\| \opA_{r,\theta} u \|^2 + \| u \|^2 \gs  \| u'' \|^2 + \| x u \|^2 + \| u \|^2, \quad u \in \Dom(\opA_{r,\theta}).
\end{equation}
Moreover, since $\| u' \|^2 \le \| u'' \| \| u \| \le (1/2) (\| u'' \|^2 + \| u \|^2)$, we also have
\begin{equation}
	\label{eq:airyinequality.firstderiv}
	\| \opA_{r,\theta} u \|^2 + \| u \|^2 \gs  \| u' \|^2, \quad u \in \Dom(\opA_{r,\theta}).
\end{equation}

In Section~\ref{sec:R}, we use operators in $L^2(\R)$ of type (with $\beta >0$)
\begin{equation}\label{eq:A.beta.def}
\opA_\beta = \Nt + |x|^\beta, \quad 
\Dom(\opA_{\beta}) = W^{1,2}(\R) \cap \Dom(|x|^\beta),
\end{equation}
which we refer to as generalised Airy operators (on Fourier space). The motivation for this choice of terminology is as follows. By transforming the complex Airy operator, $\opA_{1,\pi/2} = \Dt + i x$, to Fourier space (via $\sF \opA_{1,\pi/2} \sF^{-1}$, where $\sF$ denotes the Fourier transform, see Sub-section~\ref{ssec:fourier.pdo.prelim}), one obtains $\opA_2 = \Nt + x^2$. Operators of type $\opA_{\beta}$ extend the simple structure of $\opA_2$. Many properties of the usual complex Airy operators are preserved for $\opA_\beta$. Namely, $\opA_\beta$ has compact resolvent, empty spectrum, 
\begin{equation}
\opA_\beta^* = \Ntp + |x|^\beta, \quad 
\Dom(\opA_{\beta}^*) = W^{1,2}(\R) \cap \Dom(|x|^\beta)
\end{equation}
and 
\begin{equation}\label{eq:Abeta.graphnorm.realline.def}
\begin{aligned}
\| \opA_{\beta} \, u \|^2 + \| u \|^2 &\gs \|  u' \|^2 + \left\| |x|^\beta \, u \right\|^2 + \| u \|^2, \quad u \in \Dom(\opA_{\beta}),
\\
\|  \opA_{\beta}^* \, u \|^2 + \| u \|^2 &\gs \| u' \|^2 + \left\| |x|^\beta \, u \right\|^2 + \| u \|^2, \quad u \in \Dom(\opA_{\beta}^*).
\end{aligned}
\end{equation}
It is in fact possible to carry out this extension further to operators $\opA = \Nt + W $ with much more general $W$. See Appendix~\ref{sec:genAiry} for details and \cite{ArSi-generalised-2022} for resolvent estimates.

\subsection{Regular variation}
\label{ssec:reg.var}
A continuous function $V :\Rplus \to \Rplus$ satisfying 
\begin{equation}
\exists \beta \in \R, \quad \forall x>0, \quad \lim_{t \to + \infty} \frac{V(tx)}{V(t)}= x^{\beta},
\end{equation}
is called regularly varying (at infinity) and $\beta$ is called the index of regular variation. We can rewrite $V$ as
\begin{equation}
\label{eq:V.regvarying.decomposition}
V(x) = x^\beta L(x), \quad x > 0,
\end{equation}
where $L$ is a slowly varying function, \ie
\begin{equation}
\label{eq:slowvarying.def}
\underset{t \rightarrow +\infty}{\lim} \frac{L(t \, x)}{L(t)} = 1, \quad x > 0.
\end{equation}
It is known (see \cite[Sec.~1.5]{Seneta-1976}) that, if $L$ is slowly varying, then 
\begin{equation}
\label{eq:slowvarying.estimates}
\forall \gamma >0, \quad x^{-\gamma} < L(x) < x^{\gamma}, \quad x \rightarrow +\infty,
\end{equation}
and that the convergence in \eqref{eq:slowvarying.def} is locally uniform in $\Rplus$ (see \cite[Thm.~1.1]{Seneta-1976}). Moreover, a representation theorem (see \cite[Thm.~1.2]{Seneta-1976}) states that
\begin{equation}
\label{eq:slowvarying.reptheorem}
L(x) = a(x) \, \exp \left( \int_{1}^{x} \frac{\epsilon(y)}{y} \, \dd y \right), \quad x \ge 1,
\end{equation}
where $a$ is positive and measurable, $\eps$ is continuous and
\begin{equation}\label{eq:a.eps.regv}
\lim_{x \to + \infty} a(x)= c \in (0,\infty), \quad \lim_{x \to + \infty} \epsilon(x) =0.
\end{equation}
In this paper, we shall be chiefly concerned with functions with index $\beta > 0$.

\section{The norm of the resolvent in the range of $V$}
\label{sec:iR}
\subsection{Assumptions and statement of the result}
We begin by describing the class of potentials encompassed by our estimate for the norm of the resolvent.
\begin{asm-sec}
\label{asm:iR}
Suppose that $V \in \LilocRplusclosed \cap C^2((x_0,\infty))$ for some $x_0 \geq 0$. With $V_1 := \Re V$ and $V_2 := \Im V$, assume further that $V_1 \geq 0$ a.e.~in $\Rplus$ and that the following conditions are satisfied:
\begin{enumerate} [\upshape (i)]
\item \label{itm:incr.unbd.iR} $V_2$ is unbounded and eventually increasing:
\begin{equation}\label{eq:V.unbd.incr.iR}
\lim_{x \to +\infty} V_2(x) = +\infty, \qquad V_2'(x) > 0, \quad  x >x_0;
\end{equation}
\item \label{itm:nuasm} $V$ has controlled derivatives: there exists $\nu \in [-1,+\infty)$ such that
\begin{equation}\label{eq:V'.nu.iR}
V_2'(x) \lesssim V_2(x) \: x^{\nu}, \quad |V''(x)| \lesssim V_2'(x) \: x^{\nu}, \quad x>x_0;
\end{equation}
\item \label{itm:imvgrowth} we have:
\begin{equation}\label{eq:upsilon.def.iR}
\Upsilon(x) := x^{\nu} \left(V_2'(x)\right)^{-\frac13} = o(1), \quad x \to +\infty;
\end{equation}
\item \label{itm:realpart} $V_1'$ is sufficiently small w.r.t. $V_2'$: 
\begin{equation}\label{eq:l.def}
\lim_{x \to + \infty} \frac{V_1'(x)}{V_2'(x)} = l \in [0, +\infty).
\end{equation}
\end{enumerate}
\end{asm-sec}

For a potential $V$ satisfying Assumption~\ref{asm:iR}, the Schr\"odinger operator in $L^2(\R_+)$	
\begin{equation}\label{eq:H.def}
H = \Dt + V,
\quad \Dom(H) = W^{2,2}(\R_+) \cap W_0^{1,2}(\R_+) \cap \Dom(V)
\end{equation}
is specified as in Sub-section~\ref{ssec:Schr.prelim}; see also our comments in Sub-section~\ref{sssec:rem.asm.iR} below.

To state our result, we introduce
\begin{equation}\label{eq:r.theta.def}
r := \sqrt{l^2 + 1}, \quad \theta := \arg(l + i) \in (0, \pi/2],
\end{equation}
with $l$ as in \eqref{eq:l.def}. Assuming that $b>0$ is sufficiently large, we denote by $x_b \in \Rplus$ the unique solution (see \eqref{eq:V.unbd.incr.iR}) to the equation
\begin{equation}\label{eq:xb.def}
	V_2(x_b)=b
\end{equation}
(sometimes called a turning point of $V_2$) and define
\begin{equation}\label{eq:a.la.def}
\begin{aligned}
a &:= V_1(x_b) \ge 0, & \lambda &:= a + ib = V(x_b) \in \Ran V,\\
r_b &:= \sqrt{\left( \frac{V_1'(x_b)}{V_2'(x_b)} \right)^2 + 1}, & \theta_b &:= \arg\left( \frac{V_1'(x_b)}{V_2'(x_b)} + i \right).
\end{aligned}
\end{equation}
Furthermore, noting that by Assumption~\ref{itm:incr.unbd.iR} and \eqref{eq:xb.def} we have $x_b \to +\infty$ as $b \to +\infty$, then from Assumption~\ref{itm:realpart} we deduce that 
\begin{equation}\label{eq:kappa.def}
\kappa_b:=| r e^{i \theta} - r_b e^{i \theta_b} | = o(1), \quad b \to +\infty.
\end{equation}

\begin{theorem}
\label{thm:iR}
Let $V = V_1 + i V_2$ satisfy Assumption~\ref{asm:iR}, let $\opH$ be the Schr\"odinger operator \eqref{eq:H.def} in $\Lt(\Rplus)$ and let $A_{r,\theta}$ be the Airy operator \eqref{eq:airyde} with $r$ and $\theta$ as in \eqref{eq:r.theta.def}. Let $b$, $x_b$, $\la$ and $\kappa_b$ be as in \eqref{eq:xb.def}, \eqref{eq:a.la.def} and \eqref{eq:kappa.def}, respectively. Then as $b \rightarrow +\infty$
\begin{equation}
	\label{eq:resnorm.iR}
	\|(\opH - \lambda)^{-1} \| = \| \opA_{r,\theta}^{-1} \| \left( V_2'(x_b) \right)^{-\frac{2}{3}} \left( 1 + \BigO \left( \kappa_b + \Upsilon(x_b) \right) \right).
\end{equation}
\end{theorem}

\subsubsection{Remarks on the assumptions}
\label{sssec:rem.asm.iR}
Firstly, potentials $V$ satisfying Assumption~\ref{asm:iR} obey the separation condition \eqref{eq:vgrowth}. To see this, consider a cut-off function $\phi \in C_{c}^\infty((-2 x_0,2 x_0))$ with $0 \le \phi \le 1$ and such that
$\phi = 1$ on $\left[0, x_0\right]$. We decompose $V$ as $V = V_u + V_b := \left(1 - \phi\right) V + \phi V$, where $V_b \in \LiRplus$, $V_u \in C^2(\overline{\Rplus})$ and $\supp V_u \subset \left[x_0, +\infty\right)$. Thus it suffices to verify that \eqref{eq:vgrowth} holds for large $x$. By Assumptions~\ref{asm:iR}~\ref{itm:realpart}, \ref{itm:nuasm} and \ref{itm:imvgrowth}, we get for $x \rightarrow +\infty$
\begin{equation}\label{V.sep.nu}
\frac{|V_u'(x)|}{|V_u(x)|^{\frac{3}{2}}} \le \frac{|V_1'(x)|+|V_2'(x)|}{\left(V_2(x)\right)^{\frac{3}{2}}} \lesssim \frac{|V_2'(x)|}{\left(V_2'(x) x^{-\nu}\right)^{\frac{3}{2}}} = \Upsilon^{\frac{3}{2}}(x) = o(1).
\end{equation}

Secondly, we note that, for $\nu > 0$, Assumption~\ref{asm:iR}~\ref{itm:imvgrowth} stipulates that $V_2'$ (and hence $V_2$) must grow sufficiently fast as $x \to +\infty$. For $\nu = 0$, it simply requires that $V_2'$ be unbounded at infinity. Finally, when $\nu < 0$, potentials where $V_2'$ decays to zero as $x \to +\infty$ are supported provided that such decay is slower than that of $x^{3 \nu}$.

Our final observation is that Assumption~\ref{asm:iR}~\ref{itm:nuasm} implies that, for any $0 < \eps < 1$, all sufficiently large $x$ and $|\delta| \le \eps x^{-\nu}$, we have
\begin{equation}
\label{eq:imvapproxeq}
\frac{V_2^{(j)}(x + \delta)}{V_2^{(j)}(x)} \approx 1, \hspace{2em} j \in \{0,1\},
\end{equation}
(see e.g. \cite[Lem.~4.1]{MiSiVi-2020}). We can therefore control the variation of $V_2$ and that of $V_2'$ in intervals whose length is of order $x^{-\nu}$.

\subsection{Proof of Theorem~\ref{thm:iR}}
\label{ssec:proof.iR}
With $\la$ as in \eqref{eq:a.la.def}, let
\begin{equation}\label{eq:Hb.def}
\opH_b := \opH - \lambda.
\end{equation}
The proof is structured in four steps. Firstly, we prove the claim "away" from the zero $x_b$ of $V_2 - b$. Then we study the behaviour of the norm of the resolvent locally (i.e. near $x_b$). Next we establish a lower bound for the norm. Our final step, the theorem proof proper, combines the previously derived estimates. Throughout the proof we are chiefly concerned with the behaviour as $b \rightarrow +\infty$ and will therefore assume $b$ to be as large as needed for our assumptions to hold without further comment.

Let
\begin{equation}
\label{eq:deltabdef}
\Omega'_b := \left( x_b - \delta_b, x_b + \delta_b \right), \quad \delta_b := \delta x_b^{-\nu}, \quad 0 < \delta < \frac 14,
\end{equation}
where $\delta$ will be specified in Proposition~\ref{prop:local.iR} and $\nu \ge -1$ (see Assumption~\ref{asm:iR}~\ref{itm:nuasm}). By remarks~in Sub-section~\ref{sssec:rem.asm.iR}, the above choice for the width of $\Omega_b'$ implies that $V_2(x)$ is approximately equal to $V_2(x_b)$ inside that interval (see \eqref{eq:imvapproxeq}) and this fact will be used in the proofs below.

From \eqref{eq:deltabdef} and the already noted fact that $x_b \to +\infty$ as $b \to +\infty$, we deduce 
\begin{equation}\label{eq:xb.db}
x_b - 2 \delta_b = x_b \left( 1 - 2 \delta x_b^{-1-\nu} \right) \gtrsim x_b, \qquad b \rightarrow +\infty.	
\end{equation}
In what follows, we shall assume $b$ to be large enough so that $x_b - 2 \delta_b > \max \{1,x_0\}$ and $V_2(x_b - 2 \delta_b) > 0$. This ensures that $V_2(x) > 0 \text{ for all } x > x_b - 2 \delta_b$.

\subsubsection{Step 1: estimate outside the neighbourhood of $x_b$}
\label{sssec:step.1.iR}
\begin{proposition}	\label{prop:away.iR}
Let $\Omega'_b$ be defined by~\eqref{eq:deltabdef}, let the assumptions of Theorem~\ref{thm:iR} hold and let $\opH_b$ be as in \eqref{eq:Hb.def}. Then we have as $b \to +\infty$
\begin{equation*}
\delta \left( V_2'(x_b) \right)^\frac{2}{3} (\Upsilon(x_b))^{-1} \lesssim \inf \left\{ \frac{\left\| \opH_b u \right\|}{\|u\|}: \; 0 \neq u \in \Dom(\opH), \; \supp u \cap \Omega'_b = \emptyset\right\}.
\end{equation*}
\end{proposition}
\begin{proof}
Define $\chi_b(x) := \sgn(V_2(x) - b), \; x \in \Rplus,$ and note that $\| \chi_b \|_{\infty} \le 1$ and $\chi'_b(x) = 0, \; x \in \Rplus \setminus \Omega'_b$. Let $u \in \Dom(\opH)$ such that $\supp u \cap \Omega'_b = \emptyset$, then
\begin{equation*}
\langle \chi_b \opH_b u, u \rangle = \langle \opH_b u, \chi_b u \rangle = \langle u', \chi_b u' \rangle + \langle \left( V_1 - a \right) u, \chi_b u \rangle + i \langle \left( V_2 - b \right) u, \chi_b u \rangle.
\end{equation*}
Therefore
\begin{equation}
\label{eq:lboundstep1}
\langle |V_2 - b| u, u \rangle = \Im \langle \chi_b \opH_b u, u \rangle \le \| \opH_b u \| \| u \|.
\end{equation}
Next we find a lower bound for $|V_2(x) - V_2(x_b)|$ in $\Rplus \setminus \Omega'_b$. By Assumption~\ref{asm:iR}~\ref{itm:incr.unbd.iR}, $V_2$ is unbounded and increasing in $(x_0, +\infty)$ and, since it is also bounded on $[0, x_0]$, we have for large enough $b$ and $x \in \Rplus \setminus \Omega'_b$
\begin{align*}
|V_2(x) - V_2(x_b)| & \ge \min \left\{ V_2(x_b + \delta_b) - V_2(x_b), 
V_2(x_b) - V_2(x_b - \delta_b) \right\}.
\end{align*}
Applying the mean-value theorem for the first term inside the $\min$ with $\xi_b \in (x_b, x_b + \delta_b)$ and noting secondly that $|\xi_b - x_b| < x_b^{-\nu}/4$ by \eqref{eq:deltabdef} and therefore $V_2'(\xi_b) \approx V_2'(x_b)$ by \eqref{eq:imvapproxeq}, we deduce that for $b \to +\infty$ 
\begin{equation*}
| V_2(x_b + \delta_b) - V_2(x_b) | = V_2'(\xi_b) \delta_b \approx V_2'(x_b) \delta_b = \delta \left( V_2'(x_b) \right)^{\frac{2}{3}} (\Upsilon(x_b))^{-1}.
\end{equation*}
A similar result can be found for $| V_2(x_b - \delta_b) - V_2(x_b |$. Therefore
\begin{equation}
\label{eq:ImVlbound}
|V_2(x) - b| \gtrsim \delta \left( V_2'(x_b) \right)^{\frac{2}{3}} (\Upsilon(x_b))^{-1}, \quad x \in \Rplus \setminus \Omega'_b, \quad b \to +\infty.
\end{equation}
Hence by combining \eqref{eq:ImVlbound} and \eqref{eq:lboundstep1} we conclude that for all $u \in \Dom(\opH)$ with $\supp u \cap \Omega'_b = \emptyset$
\begin{equation*}
\delta \left( V_2'(x_b) \right)^{\frac{2}{3}} (\Upsilon(x_b))^{-1} \| u \| \lesssim \| \opH_b u \|, \quad b \to +\infty,
\end{equation*}
as required.
\end{proof}

\subsubsection{Step 2: estimate near $x_b$}
\label{sssec:step.2.iR}

\begin{proposition}	\label{prop:local.iR}
Let the assumptions of Theorem~\ref{thm:iR} hold, let $\opH_b$ be as in \eqref{eq:Hb.def} and define 
\begin{equation}
\label{eq:Omega.def}
\Omega_b := \left( x_b - 2 \delta_b, x_b + 2 \delta_b \right).
\end{equation}
Then as $b \to + \infty$
\begin{equation*}
\begin{aligned}
&\| \opA_{r,\theta}^{-1} \|^{-1} \left( V_2'(x_b) \right)^{\frac{2}{3}} \left( 1 - \BigO \left( \kappa_b + \Upsilon(x_b)\right) \right)\\
& \qquad  \qquad \le \inf \left\{ \frac{\left\| \opH_b u \right\|}{\|u\|}: \; 0 \neq u \in \Dom(\opH), \; \supp u \subset \Omega_b \right\}.
\end{aligned}	
\end{equation*}
\end{proposition}
\begin{proof}
If $x \in \Omega_b$, the Taylor expansion of $V$ around $x_b$ yields
\begin{equation*}
V(x) - V(x_b) = V'(x_b)\left( x - x_b \right) + \frac{1}{2} V''(x_b + s (x - x_b)) \left( x - x_b\right)^2,
\end{equation*}
where $s = s(x,b) \text{ and } 0 < s < 1$. Let
\begin{equation}\label{eq:V.hat.def}
\widetilde V_b (x) := V'(x_b)\left( x - x_b \right) + \frac{1}{2} V''(x_b + s (x - x_b)) \left( x - x_b\right)^2 \chi_{\Omega_b}(x), \quad x \in \R,
\end{equation}
and consider the operator in $\Lt(\R)$
\begin{equation*}
\widetilde \opH_b= \Dt + \widetilde V_b(x), \quad \Dom(\widetilde \opH_b) = W^{2,2}(\R) \cap \Dom(x).
\end{equation*}
Given $\rho > 0$, we define a unitary operator on $L^2(\R)$ by
$
(\opU_{b,\rho} \, u)(x) := \rho^{\frac{1}{2}} u(\rho x + x_b)$, $x \in \R$.
Then for any $u \in \opU_{b,\rho} ( \Dom(\widetilde \opH_b) )$
\begin{equation*}
( \opU_{b,\rho} \widetilde \opH_b \opU_{b,\rho}^{-1} u )(x) = -\frac{1}{\rho^2} u''(x) + \widetilde V_b(\rho x + x_b) u(x), \quad x \in \R.
\end{equation*}
If $\Omega_{b,\rho} := ( - 2 \delta_b \rho^{-1}, 2 \delta_b \rho^{-1} )$ and $x \in \R$, then
\begin{align*}
V_b(x) &:= \rho^2 \widetilde V_b(\rho x + x_b)\\
&= V'(x_b) \rho^3 x + \frac{1}{2} V''(\tilde{s} \rho x + x_b) \rho^4 x^2 \chi_{\Omega_{b,\rho}}(x)\\
&= \left( \frac{V_1'(x_b)}{V_2'(x_b)} + i + \frac{1}{2} \frac{V''(\tilde s \rho x + x_b)}{V_2'(x_b)} \rho x \chi_{\Omega_{b,\rho}}(x) \right) V_2'(x_b) \rho^3 x \\
&= \left( r_b e^{i \theta_b} + \frac{1}{2} \frac{V''(\tilde{s} \rho x + x_b)}{V_2'(x_b)} \rho x \chi_{\Omega_{b,\rho}}(x) \right) V_2'(x_b) \rho^3 x,
\end{align*}		
where $0 < \tilde s < 1$ and $r_b$, $\theta_b$ are as defined in \eqref{eq:a.la.def}. We are now in a position to define the value of $\rho$ for the remainder of the proof
\begin{equation}\label{eq:rho.def}
\rho := \left( V_2'(x_b) \right)^{-\frac{1}{3}}.
\end{equation}
Let us denote
\begin{equation}\label{eq:Rb.def}
R_b(x) := \frac{1}{2} \frac{V''(\tilde{s} \rho x + x_b)}{V_2'(x_b)} \rho x^2 \chi_{\Omega_{b,\rho}}(x), \quad x \in \R,
\end{equation}
then
\begin{equation}\label{eq:Vb.def}
V_b(x) = r_b e^{i \theta_b} x + R_b(x), \quad x \in \R.
\end{equation}
By Assumption~\ref{asm:iR}~\ref{itm:nuasm}, for $b \to +\infty$
\begin{equation*}
\left| \frac{V''(\tilde{s} \rho x + x_b)}{V_2'(x_b)} \right| \lesssim \frac{V_2'(\tilde{s} \rho x + x_b)}{V_2'(x_b)} (\tilde{s} \rho x + x_b)^{\nu}.
\end{equation*}
For any $x \in \Omega_{b,\rho}, \; |\tilde{s} \rho x| \le \frac 12 x_b^{-\nu}$ by \eqref{eq:deltabdef} and hence $\left( x_b^{-1}\tilde{s}\rho x + 1 \right)^{\nu} \approx 1$, \ie~$(\tilde{s} \rho x + x_b)^{\nu} \approx x_b^{\nu}$. Combining this fact with \eqref{eq:imvapproxeq}, we deduce
\begin{equation*}
\left| \frac{V''(\tilde{s} \rho x + x_b)}{V_2'(x_b)} \right| \lesssim (\tilde{s} \rho x + x_b)^{\nu} \lesssim x_b^{\nu}, \quad x \in \Omega_{b,\rho}, \quad  b \rightarrow +\infty.
\end{equation*}
For all $x \in \Omega_{b,\rho}$ we have $|\rho x|\ls \delta x_b^{-\nu}$ and therefore
\begin{equation}\label{eq:Rbdecay}
\| x^{-1} R_b \|_{\infty} \lesssim \delta, 
\qquad 
\| x^{-2} R_b\|_{\infty} \lesssim \Upsilon(x_b), \quad b \rightarrow +\infty.
\end{equation}

Let $\opS_b$ be the operator in $\Lt(\R)$
\begin{equation}\label{eq:Sb.def}
S_b = \Dt + V_b(x), \quad \Dom(S_b) = W^{2,2}(\R) \cap \Dom(x).
\end{equation}
Our next aim is to prove that $\opS_b \overset{nrc}{\longrightarrow} \opS_{\infty}$ as $b \rightarrow +\infty$ with $\opS_\infty := \opA_{r,\theta}$ from the statement of Theorem~\ref{thm:iR}.

We begin by showing that there exists $b_0 > 0$ such that $0 \in \cap_{b \ge b_0} \rho(\opS_b)$.
Note that $\opS_b = \opS_\infty + V_b - r e^{i \theta} x = \opS_\infty + (r_b e^{i \theta_b} - r e^{i \theta}) x + R_b$ and, from \eqref{eq:Rbdecay}, we have
\begin{equation}\label{eq:Rb.x}
\| r_b e^{i \theta_b} - r e^{i \theta} + x^{-1} R_b \|_{\infty} 
\le 
|r_b e^{i \theta_b} - r e^{i \theta} | +\| x^{-1} R_b \|_{\infty}  \ls \kappa_b + \delta,
\end{equation}
as $b \to +\infty$.
Note also that it follows from \eqref{eq:airyinequality} that
\begin{equation}\label{eq:x.A}
\| x \opS_\infty^{-1} \| + \| \opS_\infty^{-1} x \| \ls 1;
\end{equation}
in the estimate of the second term we use the fact that $( x (\opS_\infty^*)^{-1})^*$ is bounded and therefore from the property of adjoint $(AB)^* \supset B^* A^*$, if $AB$ is densely defined, we get that $\opS_\infty^{-1} x$ has a bounded extension. Hence, using \eqref{eq:x.A} and \eqref{eq:Rb.x}, we obtain
\begin{equation*}
\| ( V_b - r e^{i \theta} x ) \opS_\infty^{-1} \| 
\le
\| r_b e^{i \theta_b} - r e^{i \theta} + x^{-1} R_b \|_{\infty}  \|x \opS_\infty^{-1} \| \lesssim \kappa_b+ \delta, \quad b \to + \infty.
\end{equation*}
It therefore follows from \eqref{eq:kappa.def} and an appropriate choice of sufficiently small $\delta>0$ (independent of $b$) that, for all large enough $b$, the operator $\opI + ( V_b - r e^{i \theta} x ) \opS_\infty^{-1}$ is invertible and
\begin{equation}
	\label{eq:Sb.inv.local.iR}
	\opS_b^{-1} = ( \opS_\infty + V_b -  r e^{i \theta} x )^{-1} = \opS_\infty^{-1} ( \opI + ( V_b - r e^{i \theta} x ) \opS_\infty^{-1} )^{-1}.
\end{equation}
This shows that indeed $0 \in \rho(\opS_b), \; b \rightarrow + \infty$, as claimed.

Furthermore, using \eqref{eq:x.A} and \eqref{eq:Sb.inv.local.iR} we deduce
\begin{equation}
\label{eq:normestimateSb}\| \opS_b^{-1} \| + \| x \opS_b^{-1} \| +  \| \opS_b^{-1} x \| \lesssim 1, \quad  b \rightarrow +\infty.
\end{equation}
We now prove that $\opS_b \overset{nrc}{\longrightarrow} \opS_{\infty}$ as $b \to +\infty$. Using the second resolvent identity, \eqref{eq:Vb.def}, \eqref{eq:x.A}, \eqref{eq:normestimateSb} and \eqref{eq:Rbdecay},  we obtain
\begin{equation}\label{eq:Sb.Sinf}
\begin{aligned}
\| \opS_b^{-1} - \opS_{\infty}^{-1} \| &
= \| \opS_b^{-1} (V_b - r e^{i \theta} x ) \opS_{\infty}^{-1} \| 
\le \kappa_b \| \opS_b^{-1} x \opS_{\infty}^{-1} \| + \| \opS_b^{-1} R_b \opS_{\infty}^{-1} \|
\\
&\le \kappa_b \| \opS_b^{-1} \| \| x \opS_{\infty}^{-1} \| + \| \opS_b^{-1} x \| \| x^{-2} R_b \|_{\infty} \| x \opS_{\infty}^{-1} \|
\\
&\lesssim \kappa_b +  \Upsilon(x_b), \quad b \rightarrow +\infty.
\end{aligned}
\end{equation}
We therefore conclude that
\begin{equation}
\| \opS_b^{-1} \| = \| \opS_{\infty}^{-1} \| \left( 1 + \BigO \left( \kappa_b + \Upsilon(x_b)\right) \right), \quad b \to +\infty.
\end{equation}
But $\opS_b = \rho^2 \opU_{b,\rho} \widetilde \opH_b \opU_{b,\rho}^{-1}$ and hence there exists $b_0 > 0$ such that for all $b \ge b_0$
\begin{equation*}
\rho^{-2} \| \widetilde \opH_b^{-1} \| = \| \opS_\infty^{-1} \| \left( 1 + \BigO \left( \kappa_b + \Upsilon(x_b)\right) \right).
\end{equation*}
Let $b \ge b_0$ and  $u \in \Dom(\opH)$ such that $\supp u \subset \Omega_b$. Then $u \in \Dom( \widetilde \opH_b)$ and $\|\widetilde \opH_b u \| = \| \opH_b u\|$ (we view a function from $\Lt(\Rplus)$ as belonging to $\Lt(\R)$ using the natural embedding). Finally, with $v := \widetilde \opH_b u \in \Lt(\R)$, we conclude that
\begin{align*}
& \rho^{-2} \| u \| = \rho^{-2} \| \widetilde \opH_b^{-1} v \| \le \| \opS_\infty^{-1} \| \left( 1 + \BigO \left( \kappa_b + \Upsilon(x_b)\right) \right) \left\| \opH_b u \right\|. \qedhere
\end{align*}
\end{proof}

\subsubsection{Step 3: lower estimate}
\label{sssec:step.3.iR}

\begin{proposition}	\label{prop:lbound.iR}
Let the assumptions of Theorem~\ref{thm:iR} hold and let $H_b$ be as in \eqref{eq:Hb.def}. Then there exist functions $0 \neq u_b \in \Dom(\opH)$ such that
\begin{equation*}
\| \opH_b u_b \| = \| \opA_{r,\theta}^{-1} \|^{-1} \left( V_2'(x_b) \right)^{\frac{2}{3}} 
\left( 1 + 
\BigO \big( \kappa_b + \Upsilon(x_b) \big) 
\right) \|u_b\|, \quad b \to + \infty.
\end{equation*}
\end{proposition}
\begin{proof}
We retain the notation introduced in the proof of Proposition~\ref{prop:local.iR}; in particular, $\opS_{\infty} := \opA_{r,\theta}$ and $\opS_b$ is as
defined in \eqref{eq:Sb.def}.

With a sufficiently large $b_0 > 0$, the operators $\opB_b := ( \opS_b^* \opS_b)^{-1}$, $b \in (b_0,\infty]$, on $\Lt(\R)$, are compact, self-adjoint and non-negative. Let $0 < \varsigma_b^2 := \rad(\opB_b) = \max\{z: z \in \sigma(\opB_b)\}$ and let $g_b \in \Lt(\R)$ be a corresponding normalised eigenfunction, \ie~$\| \opB_b \| = \varsigma_b^2$, $\opB_b g_b = \varsigma_b^2 g_b$ and $\| g_b \| = 1$. Note that $g_b \in \Dom ( \opS_b^*\opS_b )$ and it is straightforward to verify that
\begin{equation}\label{eq:la*g*}
\| \opS_{b} g_b\| = \varsigma_b^{-1} = \|\opS_b^{-1} \|^{-1}= \| \opB_b \|^{-\frac12}, \quad b \in (b_0, \infty].
\end{equation}
Moreover, from \eqref{eq:Sb.Sinf}, we obtain 
\begin{equation}\label{eq:lab.conv}
|\varsigma_b - \varsigma_\infty|  = \BigO\left(\kappa_b + \Upsilon (x_b)\right), \quad b \rightarrow +\infty.
\end{equation}
Note also that arguing as in the justification of \eqref{eq:normestimateSb} and recalling \eqref{eq:airyinequality.firstderiv}, we obtain
\begin{equation}\label{eq:Sb.gn}
\| x (S_b^*)^{-1} \| + \| \Ntp S_b^{-1} \| \ls 1, \quad b \to + \infty.
\end{equation}

Let us take $\psi_b \in C_c^{\infty}( (-2\delta_b \rho^{-1}, 2\delta_b \rho^{-1}) )$, $0 \le \psi_b \le 1$,  
$\psi_b = 1$ on $( -\delta_b \rho^{-1}, \delta_b \rho^{-1} )$ and such that
\begin{equation}\label{eq:psib'}
\| \psi_b^{(j)} \|_{\infty} \lesssim ( \delta_b \rho^{-1})^{-j}, \quad j \in \{ 1, 2 \}.
\end{equation}
Using \eqref{eq:deltabdef} and \eqref{eq:rho.def}, we find
\begin{equation}\label{eq:db.rho}
(\delta_b \rho^{-1} )^{-1} \approx \Upsilon(x_b) = o(1), \quad b \to +\infty,
\end{equation}
by Assumption~\ref{asm:iR}~\ref{itm:imvgrowth}. As a consequence, $\psi_b \rightarrow 1$ pointwise in $\R$ as $b \rightarrow +\infty$. 

Since $\psi_b g_b \in \Dom(\opS_b)$, we have
\begin{equation}
S_b \psi_b g_b = S_b g_b + (\psi_b-1) S_b g_b + [S_b, \psi_b] g_b.
\end{equation}
The last two terms can be  estimated using \eqref{eq:la*g*}, \eqref{eq:lab.conv}, \eqref{eq:Sb.gn}, \eqref{eq:psib'} and \eqref{eq:db.rho}
\begin{equation}
\begin{aligned}
	\|(\psi_b-1) S_b g_b \| &\ls \|(\psi_b-1) x^{-1}\|_\infty \|x (S_b^*)^{-1}\| \|S_b^*  S_b g_b \| \ls \Upsilon(x_b),
	\\
	\|[S_b, \psi_b] g_b \| &\ls  \| \psi_b' \|_{\infty} \|\partial_x S_b^{-1} S_b g_b \| + \| \psi_b'' \|_{\infty} \|g_b\|
	\ls  \Upsilon(x_b),
\end{aligned}
\end{equation}
as $b \to +\infty$. Hence
%
$\|S_b \psi_b g_b\| = \varsigma_b^{-1} + \BigO(\Upsilon(x_b))$ as $b \to +\infty$.
%
Similarly, writing $\psi_b g_b = g_b + (\psi_b - 1)g_b$, we obtain
%
$\|\psi_b g_b\| = 1 + \BigO(\Upsilon(x_b))$ as  $b \to +\infty$.
%
Thus using \eqref{eq:lab.conv}, we arrive at
\begin{equation*}
\left| \frac{\left\| \opS_b \psi_b g_b \right\|}{\|\psi_b g_b\|} - \frac1{\varsigma_\infty} \right| 
= \BigO\left(\kappa_b + \Upsilon(x_b)\right), \quad b \rightarrow +\infty.
\end{equation*}

Recalling from the proof of Proposition~\ref{prop:local.iR} that $\opS_b = \rho^2 \opU_{b,\rho} \widetilde \opH_b \opU_{b,\rho}^{-1}$ and letting $u_b := \opU_{b,\rho}^{-1} \psi_b g_b$, then $u_b \in \Dom(\opH)$ with $\supp u_b \subset \Omega_b$ and we conclude
\begin{equation*}
\left| \frac{\left\| \opH_b u_b \right\|}{\|u_b\|} - \frac1{\rho^2 \varsigma_\infty} \right| 
= \BigO\left(\rho^{-2}\big(\kappa_b + \Upsilon(x_b)\big) \right), \quad b \rightarrow +\infty.
\end{equation*}
from which the claim follows.
\end{proof}

\subsubsection{Step 4: combining the estimates}
\label{sssec:step.4.iR}
With $\Omega_b'$, $\Omega_b$ and $\delta_b$ from \eqref{eq:deltabdef}, \eqref{eq:Omega.def}, let $\phi_b \in C_c^{\infty}(\Omega_b)$,  $0 \le \phi_b \le 1$, be such that
\begin{equation}
	\label{eq:phib'}
	\phi_b(x) = 1, \; x \in \Omega'_b, \quad \|\phi_b^{(j)}\|_{\infty} \lesssim \delta_b^{-j}, \quad j \in \{ 1, 2 \},
\end{equation}		
and define
\begin{equation}
	\label{eq:phikb}
	\phi_{b,0}(x) := 1 - \phi_b(x), \quad \phi_{b,1}(x) := \phi_b(x), \quad x \in \Rplus.
\end{equation}
\begin{lemma}
	\label{lem:Hb.phibk.commutator}
	Let the assumptions of Theorem~\ref{thm:iR} hold, with $\nu$ and $\Upsilon$ as in Assumptions~\ref{asm:iR}~\ref{itm:nuasm} and \ref{itm:imvgrowth}, respectively, let $\opH_b$ be as in \eqref{eq:Hb.def} and let $\phi_{b,k}, \; k \in \{0, 1\},$ be as in \eqref{eq:phikb}. Then for all $u \in \Dom(\opH)$ and $k \in \{0, 1\}$, we have
	\begin{equation}
		\label{eq:Hbphikb.commutator.norm.est}
		\| [\opH_b, \phi_{b,k}] u \| \ls \Upsilon(x_b) \| \opH_b u \| + x_b^{2 \nu} (\Upsilon(x_b))^{-1} \| u \|, \quad b \to +\infty.
	\end{equation}
\end{lemma}
\begin{proof}
	Let $u \in \Dom(\opH)$, then
	\begin{align*}
		\langle \opH_b u, \phi_{b,k}'^2 u \rangle &= -\langle u'', \phi_{b,k}'^2 u \rangle + \langle (V - \la) u, \phi_{b,k}'^2 u \rangle\\
		&= 2 \langle \phi_{b,k}' u', \phi_{b,k}'' u \rangle + \| \phi_{b,k}' u' \|^2 + \langle (V_1 - a) u, \phi_{b,k}'^2 u \rangle
		\\
		&\quad + i \langle (V_2 - b) u, \phi_{b,k}'^2 u \rangle,
	\end{align*}
	and hence
	\begin{equation}
		\label{eq:re.phikbu}
		\Re \langle \opH_b u, \phi_{b,k}'^2 u \rangle = 2 \Re \langle \phi_{b,k}' u', \phi_{b,k}'' u \rangle + \| \phi_{b,k}' u' \|^2 + \langle (V_1 - a) u, \phi_{b,k}'^2 u \rangle.
	\end{equation}
	Let $\chi_b(x) := \sgn (V_2(x) - b)$, $x \in \Rplus$, as in the proof of Proposition~\ref{prop:away.iR}. Repeating the above calculations, we deduce
	\begin{align}
		\langle \chi_b \opH_b u, \phi_{b,k}'^2 u \rangle &= 2 \langle \phi_{b,k}' u', \chi_b  
		\phi_{b,k}'' u \rangle + \langle \phi_{b,k}' u', \chi_b \phi_{b,k}' u' \rangle 
		\nonumber
		\\
		&\quad + \langle (V_1 - a) u, \chi_b \phi_{b,k}'^2 u \rangle + i \langle |V_2 - b| u, \phi_{b,k}'^2 u \rangle,
		\nonumber
		\\
		\label{eq:im.phikbu}
		\Im \langle \chi_b \opH_b u, \phi_{b,k}'^2 u \rangle &= 2 \Im \langle \phi_{b,k}' u', \chi_b \phi_{b,k}'' u \rangle + \langle |V_2 - b| u, \phi_{b,k}'^2 u \rangle.
	\end{align}
	By Assumptions~\ref{asm:iR}~\ref{itm:realpart} and \ref{itm:incr.unbd.iR}, there exists $x_1 \ge x_0$ such that
	\begin{equation}
		\label{eq:reV'imV'.est}
		\frac{|V_1'(x)|}{V_2'(x)} < l + 1, \quad x \ge x_1.
	\end{equation}
	Moreover, from \eqref{eq:xb.db}, $x_b - 2 \delta_b \ge x_1$ for sufficiently large $b$. Consequently applying \eqref{eq:reV'imV'.est} and Assumption~\ref{asm:iR}~\ref{itm:incr.unbd.iR}
	\begin{align*}
		| \langle (V_1 - a) u, \phi_{b,k}'^2 u \rangle | &\le \int_{x_b - 2 \delta_b}^{x_b - \delta_b} |V_1(x_b) - V_1(x) | |\phi_{b,k}'(x)|^2|u(x)|^2 \dd x \\
		&\quad + \int_{x_b + \delta_b}^{x_b + 2 \delta_b} |V_1(x) - V_1(x_b) | |\phi_{b,k}'(x)|^2 |u(x)|^2 \dd x \\
		&\le \int_{x_b - 2 \delta_b}^{x_b - \delta_b} \left(\int_{x}^{x_b}| V_1'(s) | \dd s\right) |\phi_{b,k}'(x)|^2 |u(x)|^2 \dd x \\
		&\quad + \int_{x_b + \delta_b}^{x_b + 2 \delta_b} \left(\int_{x_b}^{x}| V_1'(s) | \dd s\right) |\phi_{b,k}'(x)|^2 |u(x)|^2 \dd x \\
		&\le (l + 1) \langle |V_2 - b| u, \phi_{b,k}'^2 u \rangle.
	\end{align*}
	Combining this last finding with \eqref{eq:re.phikbu} and \eqref{eq:im.phikbu}
	\begin{align*}
		\| \phi_{b,k}' u' \|^2 &= \Re \langle \opH_b u, \phi_{b,k}'^2 u \rangle - 2 \Re \langle \phi_{b,k}' u', \phi_{b,k}'' u \rangle - \langle (V_1 - a) u, \phi_{b,k}'^2 u \rangle 
		\\
		&\ls \| \opH_b u \| \| \phi_{b,k}'^2 u \| + \| \phi_{b,k}' u' \| \| \phi_{b,k}'' u \| + \langle |V_2 - b| u, \phi_{b,k}'^2 u \rangle \\
		&\ls \| \opH_b u \| \| \phi_{b,k}'^2 u \| + \| \phi_{b,k}' u' \| \| \phi_{b,k}'' u \|
	\end{align*}
	and therefore for any $\eps > 0$
	\begin{align*}
		\| \phi_{b,k}' u' \| &\ls \| \opH_b u \|^{\frac12} \| \phi_{b,k}'^2 u \|^{\frac12} + \| \phi_{b,k}' u' \|^{\frac12} \| \phi_{b,k}'' u \|^{\frac12} \\
		&\ls \Upsilon(x_b) \| \opH_b u \| + x_b^{2 \nu} (\Upsilon(x_b))^{-1} \| u \| + \eps \| \phi_{b,k}' u' \| + \eps^{-1} x_b^{2 \nu} \| u \|,
	\end{align*}
	where we have applied \eqref{eq:phib'}. Choosing a sufficiently small $\eps$ and using Assumption~\ref{asm:iR}~\ref{itm:imvgrowth} we deduce
	\begin{equation*}
		\| \phi_{b,k}' u' \| \ls \Upsilon(x_b) \| \opH_b u \| + x_b^{2 \nu} (\Upsilon(x_b))^{-1} \| u \|.
	\end{equation*}
	Finally, applying once more Assumption~\ref{asm:iR}~\ref{itm:imvgrowth}
	\begin{equation*}
		\| [\opH_b, \phi_{b,k}] u \| \le 2 \| \phi_{b,k}' u' \| + \| \phi_{b,k}'' u \| 
		\ls \Upsilon(x_b) \| \opH_b u \| + x_b^{2 \nu} (\Upsilon(x_b))^{-1} \| u \|,
	\end{equation*}
	as claimed.
\end{proof}

\begin{proof}[Proof of Theorem~\ref{thm:iR}]
Let $u \in \Dom(\opH)$, with $\phi_{b,k}$, $k \in \{0, 1\}$, as in \eqref{eq:phikb}, and write $u = u_0 + u_1$ where $u_0 := \phi_{b,0} u$ and $u_1 := \phi_{b,1} u$. Then
\begin{equation}
	\label{eq:hbuexpansion}
	\opH_b u_k = \phi_{b,k} \opH_b u + [\opH_b, \phi_{b,k}] u, \quad k \in \{0, 1\},
\end{equation}
and therefore by \eqref{eq:Hbphikb.commutator.norm.est} as $b \to +\infty$
\begin{equation}
	\label{eq:hbuexpansion.est}
	\|\opH_b u_k\| \leq \left( 1 + \BigO \left( \Upsilon(x_b) \right) \right) \|\opH_b u\| + \BigO \left( x_b^{2 \nu} (\Upsilon(x_b))^{-1} \right) \|u\|, \quad k \in \{0, 1\}.
\end{equation}	

Firstly, note that $\supp u_1 \subset \Omega_b$, hence by Proposition~\ref{prop:local.iR}
\begin{equation*}
	\| u_1 \| \le \| \opA_{r,\theta}^{-1}\| \left( V_2'(x_b) \right)^{-\frac{2}{3}} \left( 1 + \BigO \left( \kappa_b + \Upsilon(x_b)\right) \right)  \left\| \opH_b u_1 \right\|, \quad b \rightarrow +\infty.
\end{equation*}
Thus by Assumption~\ref{asm:iR}~\ref{itm:imvgrowth} and \eqref{eq:hbuexpansion.est}, we have as $b \rightarrow +\infty$
\begin{equation}
	\label{eq:u1upperbound}
	\| u_1 \| \le \| \opA_{r,\theta}^{-1}\| \left( V_2'(x_b) \right)^{-\frac{2}{3}} \left( 1 + \BigO \left( \kappa_b + \Upsilon(x_b)\right) \right)  \left\| \opH_b u \right\| + \BigO \left( \Upsilon(x_b) \right) \left\| u \right\|.
\end{equation}

Secondly, since $\supp u_0 \cap \Omega'_b = \emptyset$, by Proposition~\ref{prop:away.iR}
\begin{equation*}
	\| u_0 \| \lesssim \left( V_2'(x_b) \right)^{-\frac{2}{3}} \Upsilon(x_b) \left\| \opH_b u_0 \right\|, \quad b \rightarrow +\infty
\end{equation*}
and applying again Assumption~\ref{asm:iR}~\ref{itm:imvgrowth} and \eqref{eq:hbuexpansion.est}, we have as $b \rightarrow +\infty$
\begin{equation}
	\label{eq:u2upperbound}
	\| u_0 \| \ls \left( V_2'(x_b) \right)^{-\frac{2}{3}} \Upsilon(x_b)  \left\| \opH_b u \right\| + (\Upsilon(x_b))^2 \| u \| .
\end{equation}
Combining \eqref{eq:u1upperbound} and \eqref{eq:u2upperbound} and applying Assumption~~\ref{asm:iR}~\ref{itm:imvgrowth}, we find that as $b \rightarrow +\infty$
\begin{align*}
	\| u \| &\le \| u_0 \| + \| u_1 \|\\
	&\le \| \opA_{r,\theta}^{-1} \| \left( V_2'(x_b) \right)^{-\frac{2}{3}} \left( 1 + \BigO \left( \kappa_b + \Upsilon(x_b) \right) \right) \left\| \opH_b u \right\| + \BigO ( \Upsilon(x_b)) ) \| u \|
\end{align*}
and hence
\begin{equation}
	\label{eq:Hb.est.iR}
	\| u \| \le \| \opA_{r,\theta}^{-1} \| \left( V_2'(x_b) \right)^{-\frac{2}{3}} \left( 1 + \BigO \left( \kappa_b + \Upsilon(x_b) \right) \right) \| \opH_b u \|.
\end{equation}
An appeal to Proposition~\ref{prop:lbound.iR} completes the proof of Theorem~\ref{thm:iR}.
\end{proof}


\section{The norm of the resolvent in the real axis}
\label{sec:R}
\subsection{Assumptions and statement of results}
We begin by describing the class of potentials covered by our estimate for the norm of the resolvent in the real axis.
\begin{asm-sec}
\label{asm:R}
Suppose that $V := i V_2$ with $V_2:\R \rightarrow \overline\Rplus, \; V_2 \in \CiR$ satisfying
\begin{enumerate} [\upshape (i)]
\item \label{itm:even} $V_2$ is even: 
\begin{equation}\label{eq:V.even.R}
V_2(-x) = V_2(x), \quad x \in \R;
\end{equation}

\item \label{itm:oneone} $V_2$ is eventually increasing: 
\begin{equation}\label{eq:V.incr.R}
\exists x_0 > 0, \quad \forall x > x_0, \quad V_2'(x) > 0;
\end{equation}
\item \label{itm:regularlyvarying}
$V_2$ is regularly varying:
\begin{equation}\label{eq:V.beta.def.R}
\exists \beta>0, \quad \forall x>0, \quad 	\lim_{t \to + \infty} W_t(x) = \omega_{\beta}(x),
\end{equation}
where
\begin{equation}\label{eq:Wt.Om.def}
W_t(x):= \frac{V_2(tx)}{V_2(t)}, \quad  \omega_{\beta}(x):=|x|^{\beta}, \quad x \in \R;
\end{equation}
\item \label{itm:symbolclass} $V_2$ has controlled derivatives:
\begin{equation}\label{eq:V.symb}
\forall n \in \N, \quad \exists C_n>0, \quad |V_2^{(n)}(x)| \le C_n \, \left(1 + V_2(x)\right) \, \langle x \rangle^{- n}, \quad x \in \R.
\end{equation}
\end{enumerate}
\end{asm-sec}

For potentials $V$ satisfying Assumption~\ref{asm:R}, we consider the Schr\"odinger operator
\begin{equation}\label{eq:H.real.def}
\opH = \Dt + V,
\quad \Dom(H) = W^{2,2}(\R) \cap \Dom(V),
\end{equation}
as in Sub-section~\ref{ssec:Schr.prelim}. 

To state the result, we define the positive real numbers $t_a$ via the equation 
\begin{equation}\label{eq:ta.def}
t_a V_2(t_a) = 2 \sqrt{a};
\end{equation}	
notice that $ t \mapsto t V_2(t)$ is eventually increasing by Assumption~\eqref{eq:V.incr.R}, thus $a \mapsto t_a$ is well-defined for all sufficiently large $a>0$. Moreover, it follows that $t_a \to +\infty$ as $a \to +\infty$. Finally, let
\begin{equation}
\label{eq:iota.def}
\iota(t) := \| (1 + W_t)^{-1} - (1 + \omega_\beta)^{-1} \|_{\infty}; 
\end{equation}
Lemma~\ref{lem:reg.var.conv} shows that $\iota(t) \to 0$ as $t \to + \infty$.

\begin{theorem}	\label{thm:R}
Let $V = i V_2$ satisfy Assumption~\ref{asm:R} and let $\opH$ be the Schr\"odinger operator \eqref{eq:H.real.def} in $\Lt(\R)$. Furthermore let $\opA_{\beta}$ be the generalised Airy operator \eqref{eq:A.beta.def}, let $t_a$ be as in \eqref{eq:ta.def} and let $\iota$ be as in \eqref{eq:iota.def}. Then as $a \to +\infty$
\begin{equation}
	\label{eq:resnorm.R}
	\|(\opH - a)^{-1} \| = \| \opA_{\beta}^{-1} \| \, V_2(t_a)^{-1} \left( 1 + \BigO \left( \iota(t_a) + (a^{\frac 12} t_a)^{-l_{\beta, \eps}} \right) \right)
\end{equation}
with $0 < \eps < \beta$ arbitrarily small and
\begin{equation}\label{eq:lbeta.def}
	l_{\beta, \eps} := 
	\begin{cases}
		1 - \eps, &  \beta > 1/2, \\[1mm] 
		1/2 + \beta - \eps, &  \beta \in (0,1/2].
	\end{cases}
\end{equation}
\end{theorem}

\subsubsection{Remarks on the assumptions}
As a consequence of \eqref{eq:slowvarying.estimates}, if $V$ satisfies Assumption~\ref{asm:R}~\ref{itm:regularlyvarying}, then
\begin{equation}
	\label{eq:VgrowsInfty.realline}
	\underset{|x| \rightarrow +\infty}{\lim} V_2(x) = +\infty.
\end{equation}

Moreover, by Assumption~\ref{asm:R}~\ref{itm:symbolclass} with $n = 1$, for any arbitrarily small $\eps > 0$
\begin{equation*}
	\frac{|V_2'(x)|}{|V_2(x)|^\frac32} \ls \frac{(1 + V_2(x)) \langle x \rangle^{-1}}{|V_2(x)|^\frac32} \ls \eps, \quad |x| \to +\infty,
\end{equation*}
and it follows that $V$ satisfies condition \eqref{eq:vgrowth}. Hence the graph norm of $\opH$ separates
\begin{equation}
\label{eq:H.real.separation.prop}
\| \opH u \|^2 + \| u \|^2 \gs \| u'' \|^2 + \| V u \|^2 + \| u \|^2, \quad u \in \Dom(\opH).
\end{equation}

Finally, the following estimates for the derivatives of $W_t$ shall be used in Steps 2 and 3 of the proof of Theorem~\ref{thm:R}.

\begin{lemma}\label{lem:Wt.asm}
Let $V = i V_2$ satisfy Assumption~\ref{asm:R} and let $W_t$ be as in \eqref{eq:Wt.Om.def}. Then for each $n \in N$, there exists a constant $D_n$, independent of $t$, such that for all $t>t_0$, with a sufficiently large $t_0>0$, independent of $n$, and all $|x|\geq1$
\begin{equation}\label{eq:Wt.symb}
	|W_t^{(n)}(x)| \le D_n (1 + W_t(x))  \langle x \rangle^{- n}.
\end{equation}
\end{lemma}
\begin{proof}
The claim follows from \eqref{eq:V.symb}, \eqref{eq:VgrowsInfty.realline} and $|x| \geq 1$, namely
\begin{equation*}
|W_t^{(n)}(x)| = t^n \frac{|V_2^{(n)}(t x)|}{V_2(t)} 
\le C_n \frac{(t |x|)^n}{|x|^n \langle t x \rangle^n} \frac{1 + V_2(t  x)}{V_2(t)} 
\le D_n  \frac{1 + W_t(x)}{\langle x \rangle^n}. 
\qedhere
\end{equation*}
\end{proof}

\subsection{Proof of Theorem~\ref{thm:R}}
\label{ssec:proof.R}

We transform the problem to Fourier space and implement there the strategy of Sub-section~\ref{ssec:proof.iR}. To this end, we introduce the operators in $L^2(\R)$
\begin{equation}\label{eq:Hhat.P.def}
\begin{aligned}
\widehat \opH& :=  -i \, \sF \, \opH \, \sF^{-1},& \Dom(\widehat \opH) &:= \{ u \in \Lt(\R) \,: \; \check{u} \in \Dom(\opH) \},
\\
\opwV &:= -i \, \sF \, V \, \sF^{-1},& \Dom(\opwV) &:= \{  u \in \Lt(\R) \, :\, \check{u} \in \Dom(V) \}.
\end{aligned}
\end{equation}

Notice that $\widehat \opH = \opwV - i\, \xi^2$, $\| \widehat \opH  u \| = \| \opH  \check{u} \|$ for all $u \in \Dom(\widehat \opH)$ and $\| \opwV  u \| = \| V \check{u} \|$ for all $u \in \Dom(\opwV)$.
Thus the separation of the graph norm of $\opH$, see \eqref{eq:H.real.separation.prop}, yields
\begin{equation}\label{eq:Hhat.real.separation.prop}
\| \widehat \opH  u \|^2 + \| u \|^2 \gs \|\xi^2 u \|^2 + \| \opwV u \|^2  + \|u\|^2, \quad u \in \Dom(\widehat \opH).
\end{equation}

The proof has an analogous structure to that of Theorem~\ref{thm:iR} but nonetheless some steps are more technical. In particular, our simple estimate of the commutator of $\Dt$ and a cut-off partition of unity in Step 4 of Theorem~\ref{thm:iR} (see Sub-section~\ref{sssec:step.4.iR}) requires more effort here (see Step 0 below).

\subsubsection{Step 0: commutator estimate}
\label{sssec:step.0.R}
The proof of our next lemma specialises that of \cite[Thm.~3.15]{Abels-2011} for the operators that we are interested in.
\begin{lemma}\label{lem:pdo.comp}
Let $F \in \CiR$ and $m > 0$ be such that 
\begin{equation}
\label{eq:Vsymbolclass}
\forall n \in \N_0, \quad \exists C_n>0, \quad |F^{(n)}(x)| \le C_n \, \langle x \rangle^{m - n}, \quad x \in \R,
\end{equation}
and let $\phi \in \CiR \cap \LiR$ be such that $\supp \phi'$ is bounded. For $j \in \N_0$ and $u \in \SchwR$, we define the operators (with $\opP:=\opP^{(0)}$ and $\opQ:=\opQ^{(0)}$)
\begin{align}
\label{eq:composition.pq.def}\opP^{(j)} u := \sF \, F^{(j)} \sF^{-1} u, \qquad \opQ^{(j)} u := \phi^{(j)} u.
\end{align}
Then, for any $N \in \N_0$, we have 
\begin{equation}
\label{eq:compositionformula} [\opP,\opQ]  u = \sum_{j=1}^{N} \frac{i^j}{j!} \opQ^{(j)} \opP^{(j)} u + \opR_{N+1} u, \quad u \in \SchwR,
\end{equation}
where $R_{N+1}$ is a pseudodifferential operator with symbol $r_{N+1} \in \cS^{m-N-1}_{1,0}(\R \times \R)$
\begin{equation}
\label{eq:comp.operator.remainder}\opR_{N+1} \, u(\xi) := \intR e^{-i \xi x} \, r_{N+1}(\xi,x) \, \check{u}(x) \, \odd x. 
\end{equation}
Moreover, for every $N \in \N$ with $N > m$, there exist $l = l(N) \in \N$ and $K_N > 0$, independent of $F$ and $\phi$, such that %
\begin{equation}\label{eq:R_N.est}
\| \opR_{N+1} \, u \| \le K_N \underset{0 \le j \le l}{\max} \left\{ \| \phi^{(N + 1 + j)}\|_\infty \right\} \, \| u \|. 
\end{equation}
\end{lemma}
\begin{proof}
Let $p(\xi,x) := F(x)$ and $q(\xi,x) := \phi(\xi)$, then our hypotheses ensure $p \in \cS^m_{1,0}(\R \times \R)$ and $q \in \cS^0_{1,0}(\R \times \R)$. Moreover (with $\xi \in \R$)
\begin{equation*}
\opP \, u(\xi) = \intR e^{-i \xi x} \, p(\xi,x) \, \check{u}(x) \, \odd x,  \quad 
\opQ \, u(\xi) = \intR e^{-i \xi x} \, q(\xi,x) \, \check{u}(x) \, \odd x,
\end{equation*}	
and therefore both symbols define continuous mappings on $\SchwR$ (see \cite[Thm.~3.6]{Abels-2011}). An analogous claim holds for $\opP^{(j)} u$, $\opQ^{(j)}$, $j \in \N$.
Furthermore, by the composition theorem \cite[Thm.~3.16]{Abels-2011}, $\opP \opQ$ is a pseudo-differential operator with symbol $p\#q \in \cS^m_{1,0}(\R \times \R)$ determined by
\begin{equation*}
p\#q(\xi,x) = \sum_{j=0}^{N} \frac{i^j}{j!}   \phi^{(j)}(\xi) \, F^{(j)}(x) + r_{N+1}(\xi,x),
\end{equation*}
where $r_{N+1} \in \cS^{m-N-1}_{1,0}(\R \times \R)$ for any $N \in \N_0$ and (with $x,x',\xi,\xi' \in \R$)
\begin{align}
	\label{eq:comp.symbol.remainder.1}r_{N+1}(\xi,x) &:= \frac{i^{N+1}}{N!}   \, \Os \iint e^{i x' \xi'} \, a_{\xi,x}(\xi',x') \, \odd x' \, \odd \xi', 
	\\
	\label{eq:comp.symbol.remainder.2}a_{\xi,x}(\xi',x') &:= \phi^{(N+1)}(\xi + \xi') \, \int_{0}^{1} \left(1 - \theta\right)^N F^{(N+1)}(x + \theta x') \, \dd \theta.
\end{align}
Thus the composition formula \eqref{eq:compositionformula} follows by simple manipulations.

In the following, $x,x', \xi, \xi' \in \R$ and $\alpha=(\alpha_1, \alpha_2)$, $\beta=(\beta_1, \beta_2) \in \N_0^2$ are arbitrary.	We define $a(\xi, \xi', x, x') := a_{\xi,x}(\xi',x') \in C^{\infty}(\R^2 \times \R^2)$ with $a_{\xi,x}$ given by \eqref{eq:comp.symbol.remainder.2}. Using the assumption \eqref{eq:Vsymbolclass}, we obtain 
\begin{equation}\label{eq:a.est}
\begin{aligned}
& \left| \partial^{\alpha}_{(\xi,\xi')} \partial^{\beta}_{(x,x')} a(\xi, \xi', x, x') \right| 
\\ 
& \qquad = \left| \phi^{(N+1+|\alpha|)}(\xi + \xi') \int_{0}^{1} \left(1 - \theta\right)^N \theta^{\beta_2}  F^{(N+1+|\beta|)}(x + \theta x') \, \dd \theta \right| 
\\
&\qquad \le C_{N, \beta}  \| \phi^{(N+1+|\alpha|)} \|_{\infty} \int_{0}^{1} \left(1 - \theta\right)^N \theta^{\beta_2} \, \langle x + \theta x' \rangle^{m-N-1} \, \dd \theta 
\\
&\qquad \le C'_{N, \beta} \| \phi^{(N+1+|\alpha|)} \|_{\infty} \langle (x, x') \rangle^{|m-N-1|},
%
%
\end{aligned}
\end{equation}
where in the last step we have used the fact
\begin{equation*}
	\langle x + \theta x' \rangle^{m-N-1} \le \langle x + \theta x' \rangle^{|m-N-1|} \lesssim \langle (x, x') \rangle^{|m-N-1|}, \quad \theta \in [0,1], \quad x, x' \in \R.
\end{equation*}
Notice that $C'_{N, \beta}$ is independent of $\xi, x, \xi', x'$ and $\theta$ and therefore \eqref{eq:a.est} shows that $a \in \cA^{|m-N-1|}_{0}(\R^2 \times \R^2)$. Applying Fubini's theorem for oscillatory integrals \cite[Thm.~3.13]{Abels-2011} to \eqref{eq:comp.symbol.remainder.1}, we deduce that for any $\alpha_1, \beta_1 \in \N_0$ and $\xi, x \in \R$
\begin{equation*}
\partial^{\alpha_1}_{\xi} \partial^{\beta_1}_{x} r_{N+1}(\xi,x) = \frac{i^{N+1}}{N!}   \, \Os \iint e^{i x' \xi'}  \partial^{\alpha_1}_{\xi} \partial^{\beta_1}_{x} a_{\xi,x}(\xi',x') \, \odd x' \, \odd \xi'.
\end{equation*}
Moreover, by Peetre's inequality (see \cite[Lem.~3.7]{Abels-2011})
\begin{equation*}
	\langle x + \theta x' \rangle^{m-N-1} \lesssim \langle x \rangle^{m-N-1} \, \langle x' \rangle^{|m-N-1|}, \quad \theta \in [0,1], \quad x, x' \in \R.
\end{equation*}
Therefore \eqref{eq:a.est} also implies that, for any $\xi, x \in \R$, $\partial^{\alpha_1}_{\xi} \partial^{\beta_1}_{x} a_{\xi,x} \in \cA^{|m-N-1|}_{0}(\R \times \R)$ w.r.t.~$(\xi', x')$ and, for any $l \in \N_0$, there exists $C_{N,\beta_1,l} > 0$ such that
\begin{equation}
\label{eq:compsymbol.seminorm.ubound.2}
|\partial^{\alpha_1}_{\xi} \partial^{\beta_1}_{x} a_{\xi,x}|_{\cA^{|m-N-1|}_{0},l} \le C_{N, \beta_1, l} \, \underset{0 \le j \le l}{\max} \, \| \phi^{(N+1+\alpha_1+j)} \|_{\infty} \, \langle x \rangle^{m-N-1}.
\end{equation}
Hence by \cite[Thm.~3.9]{Abels-2011}, for a sufficiently large $l \in \N$ (depending on $N$)
\begin{equation}\label{eq:remainder.symbol.ubound}
\begin{aligned}
\left| \partial^{\alpha_1}_{\xi} \partial^{\beta_1}_{x} r_{N+1}(\xi,x) \right| &= \frac{1}{N!} \left|\Os \iint e^{i x' \xi'}  \partial^{\alpha_1}_{\xi} \partial^{\beta_1}_{x} a_{\xi,x}(\xi',x') \, \odd x' \, \odd \xi'\right|
\\
&\le 
C_N  |\partial^{\alpha_1}_{\xi} \partial^{\beta_1}_{x} a_{\xi,x}|_{\cA^{|m-N-1|}_{0},l}
\\
&\le 
C'_{N, \beta_1, l} \, \underset{0 \le j \le l}{\max} \, \| \phi^{(N+1+\alpha_1+j)} \|_{\infty} \, \langle x \rangle^{m-N-1},
\end{aligned}
\end{equation}
with $C'_{N, \beta_1, l} > 0$ independent of $F$ and $\phi$. Since $r_{N+1} \in \cS^{m-N-1}_{1,0}(\R \times \R)$, it follows that, for any $N > m$, $\xi \in \R$ and $\beta_1 \in \N_0$, $\partial_x^{\beta_1}r_{N+1}(\xi, \cdot) \in L^1(\R)$ and therefore
\begin{equation*}
k(\xi,z) := \intR e^{-i z x} \, r_{N+1}(\xi,x) \, \odd x, \quad \xi, z \in \R,
\end{equation*}
is well-defined. Moreover, by \eqref{eq:remainder.symbol.ubound}, for large enough $l \in \N$ and some $C_{N, l} > 0$ (independent of $F$ and $\phi$)
\begin{equation*}
\left|(1+z^2) k(\xi,z)\right| = \left| \intR e^{-i z x} (1- \partial_x^{2}) r_{N+1}(\xi,x) \, \odd x \right| \le C_{N,l} \, \underset{0 \le j \le l}{\max} \| \phi^{(N+1+j)} \|_{\infty}.
\end{equation*}
%
%
Hence
\begin{equation}
\label{eq:g.remainder.estimate}
g(z) := \underset{\xi \in \R}{\sup} \, |k(\xi,z)| \le C_{N,l} \, \underset{0 \le j \le l}{\max} \, \| \phi^{(N+1+j)} \|_{\infty} \, \left( 1 + z^2 \right)^{-1} \in L^1(\R)
\end{equation}
and
\begin{align*}
|\opR_{N+1} \, u(\xi)| &=\left| \intR e^{-i \xi x} \, r_{N+1}(\xi,x) \, \check{u}(x) \, \odd x\right| 
\ \
\le \intR |k(\xi,\xi-\eta)| \, |u(\eta)| \, \odd \eta 
\\	
&	\le \intR g(\xi-\eta) \, |u(\eta)| \, \odd \eta = \left(g*|u|\right)(\xi).
\end{align*}
The claim \eqref{eq:R_N.est} follows by Young's inequality and \eqref{eq:g.remainder.estimate}.
\end{proof}

\subsubsection{Step 1: estimate outside the neighbourhoods of $\pm \xi_a$}

For $a \in \Rplus$, we shall denote
\begin{equation}
\label{eq:deltaa.def}
\Omega'_{a,\pm} := ( \pm \xi_a - \delta_a,  \pm \xi_a + \delta_a ), \quad \xi_a := \sqrt{a}, \quad \delta_a := \delta  \xi_a, \quad 0 < \delta < \frac14,
\end{equation}
where the parameter $\delta$ will be specified in Proposition~\ref{prop:local.R} and
\begin{equation}
\label{eq:Ha.def}
\opH_a := \opH - a, \quad \widehat \opH_a := -i  \sF  \opH_a \, \sF^{-1} = \widehat \opH + i \, a = \widehat{V} - i (\xi^2 - a).
\end{equation}

\begin{proposition}\label{prop:away.R}
Let $\Omega'_{a,\pm}$ be defined by~\eqref{eq:deltaa.def}, let the assumptions of Theorem~\ref{thm:R} hold and let $\widehat \opH_a$ be as in \eqref{eq:Ha.def}. Then as $a \to +\infty$
\begin{equation}\label{eq:away.est}
a \lesssim \inf \left\{ \frac{\| \widehat \opH_a  u \|}{\|u\|} \, : \, 0 \neq u \in \Dom(\widehat \opH), \; \supp u \cap (\Omega'_{a,+} \cup \Omega'_{a,-}) = \emptyset \right\}.
\end{equation}
\end{proposition}
\begin{proof}
In what follows, we shall assume $a$ to be large and positive. Let $0 \ne u \in \Dom(\widehat \opH)$ with $\supp u \cap (\Omega'_{a,+} \cup \Omega'_{a,-}) = \emptyset$ and consider
\begin{align*}
	\| \widehat \opH_a u \|^2 &= \| \widehat V u \|^2 + \| (\xi^2 - a)  u \|^2 + 2 \Re \langle \widehat V u, -i(\xi^2 - a) u \rangle\\
	&\ge \| \widehat V u \|^2 + \frac12 \| (\xi^2 - a)  u \|^2 + \frac12 \| (\xi^2 - a)  u \|^2 - 2 | \Re \langle \widehat V u, -i(\xi^2 - a) u \rangle |.
\end{align*}
Note that
\begin{equation*}
	|\Re \langle \widehat V u, -i(\xi^2 - a) u \rangle| = |\Re \langle \widehat V u, -i \xi^2 u \rangle| \le |\langle V_2' \check u, \check u' \rangle| \ls \| (1 + \widehat V) u \| \| \xi u \|,
\end{equation*}
appealing to Assumption~\ref{asm:R}~\ref{itm:symbolclass} with $n = 1$ for the last estimate. Furthermore, for any $\eps > 0$, there exist $C_\eps > 0$ such that
\begin{equation*}
\| \widehat V u \| \| \xi u \|  +\| u \| \| \xi u \| \le \eps \| \widehat V u \|^2 + \eps \| \xi^2 u \|^2 + C_\eps \| u \|^2.
\end{equation*}
Noting also that, for any $\xi \in \supp u$, there exists $C_\delta' > 0$ such that
\begin{gather*}
	|\xi^2 - a| = |\xi + \xi_a| |\xi - \xi_a| \ge \delta_a^2 = \delta^2 a,\\
	|\xi| \le |\xi \pm \xi_a| + \xi_a \le (1 + 1/\delta) |\xi \pm \xi_a| \implies |\xi^2 - a| \ge C_\delta' \xi^2.
\end{gather*}
Hence, with an appropriate choice of $\eps$, we conclude that there exists $C_\delta > 0$ such that
\begin{equation}
	\label{eq:Hahat.away.est.R}
	\| \widehat \opH_a u \|^2 \ge C_\delta \left( \| \xi^2 u \|^2 + \| \widehat V u \|^2 + a^2 \| u \|^2 \right),
\end{equation}
which proves the claim.
\end{proof}

\subsubsection{Step 2: estimate near $\pm \xi_a$}
We start with three lemmas used in the proof of Proposition~\ref{prop:local.R} below.

\begin{lemma}\label{lem:reg.var.comp}
Let $V = i V_2$ satisfy Assumption~\ref{asm:R} and let $W_t$, $\omega_\beta$ be as in \eqref{eq:Wt.Om.def}. Then for any  $a, b \in \R$, with $a < b$, we have 
$\| ( W_t - \omega_\beta ) \, \chi_{[a, b]} \|_{\infty} \to 0$ as $t \to + \infty$.
\end{lemma}
\begin{proof}
Because of Assumption~\ref{asm:R}~\ref{itm:even}, it suffices to consider $a \ge 0$. Assume firstly that $a > 0$ and let $L$ be the slowly varying function such that $V_2 = \omega_\beta L$ (see  \eqref{eq:V.regvarying.decomposition}--\eqref{eq:slowvarying.def}). Then for all $x \in [a, b]$
\begin{equation*}
	\left| W_t(x) - \omega_\beta(x) \right| = \omega_\beta(x) \, \left| \frac{L(t \, x)}{L(t)} - 1 \right| \le \omega_\beta(b)  \underset{a \le x \le b}{\max}  \left| \frac{L(t \, x)}{L(t)} - 1 \right|,
\end{equation*}
and the claim follows by the locally uniform convergence for $L$ (see Sub-section~\ref{ssec:reg.var}).

For $[0, b]$, let $\eps > 0$ be arbitrarily small and take $b' \in (0, b]$ such that $0 \le \omega_\beta(x) < \eps$ for any $x \in [0, b']$. If $x_0$ is as in Assumption~\ref{asm:R}~\ref{itm:oneone}, then, for any $x \in [0, b']$ and $t > \tau_0 := x_0/b'$, we have
\begin{equation*}
	0 \le \frac{V_2(t \, x)}{V_2(t)} \le \frac{\underset{0 \le y \le x_0}{\max} V_2(y) + \underset{x_0 \le y \le b' t}{\max} V_2(y)}{V_2(t)} \le \frac{\underset{0 \le y \le x_0}{\max} V_2(y)}{V_2(t)} + \frac{V_2(b' \, t)}{V_2(t)},
\end{equation*}
where we have used the assumption that $V_2$ is increasing in $[x_0, \, +\infty)$. Therefore, by \eqref{eq:VgrowsInfty.realline} and Assumption~\ref{asm:R}~\ref{itm:regularlyvarying}, there exists $\tau_1 \ge \tau_0$ such that
\begin{equation*}
	0 \le \frac{V_2(t \, x)}{V_2(t)} \le \eps + \omega_\beta(b') + \eps < 3 \eps, \quad x \in [0, b'], \quad t > \tau_1.
\end{equation*}
Hence
\begin{equation*}
	\left| W_t(x) - \omega_\beta(x) \right| \le W_t(x) + \omega_\beta(x) < 4 \eps, \quad x \in [0, b'], \quad t > \tau_1.
\end{equation*}
If $b' < b$, then we use the first part of the proof to find $\tau_2 \ge \tau_1$ such that
\begin{equation*}
\left| W_t(x) - \omega_\beta(x) \right| < \eps, \quad x \in [b', b], \quad t > \tau_2,
\end{equation*}
which concludes the proof for $[0, b]$.
\end{proof}

\begin{lemma}\label{lem:reg.var.conv}
Let $V = i V_2$ satisfy Assumption~\ref{asm:R} and let $\iota$ be as in \eqref{eq:iota.def}. Then $\iota(t) = o(1)$ as $t \to + \infty$.
\end{lemma}
\begin{proof}
By Assumption~\ref{asm:R}~\ref{itm:even}, it is enough to consider what happens to
\begin{equation*}
\underset{0 \le x < +\infty}{\sup} \, \left| (1 + W_t(x))^{-1} - (1 + \omega_\beta(x))^{-1} \right|, \quad  t \rightarrow + \infty.
\end{equation*}
Let $\eps > 0$, then there exists $M_1 > 1$ such that
\begin{equation}
\label{eq:Vbeta.Vt.infnorm.est1}
(1 + \omega_\beta(x))^{-1} < \eps, \quad x > M_1.
\end{equation}
Let $L$ be the slowly varying function such that $V_2 = \omega_\beta L$ (see \eqref{eq:V.regvarying.decomposition}--\eqref{eq:slowvarying.def}) and consider $\gamma \in(0,\beta)$. Using the representation of $L$ in \eqref{eq:slowvarying.reptheorem} and properties of $a$ and $\epsilon$ (see \eqref{eq:a.eps.regv}), there exists $\tau_1 > 1$ such that for all $t > \tau_1$ and $x > 1$, we have
\begin{equation}
	\label{eq:Ltx.Lt.nonlocal.lbound}
	\frac{L(tx)}{L(t)} = \frac{a(t  x)}{a(t)} \exp \left(\int_{t}^{t x} \frac{\epsilon(y)}{y} \dd y\right) \ge \frac12  \exp \left( -\gamma \, \int_{t}^{t  x} \frac{\dd y}{y}  \right) = \frac12  x^{-\gamma}.
\end{equation}	
Therefore by \eqref{eq:V.regvarying.decomposition}
\begin{equation}
	\label{eq:Wt.nonlocal.lbound}
	1 + W_t(x) = 1 + \omega_\beta(x) \, \frac{L(t \, x)}{L(t)} \ge 1 + \frac12 \, x^{\beta - \gamma}, \quad x > 1, \quad t > \tau_1
\end{equation}
and we conclude that there exists $M_2 \geq M_1$ such that
\begin{equation}
\label{eq:Vbeta.Vt.infnorm.est2}
(1 + W_t(x))^{-1} < \eps, \quad x > M_2, \quad t > \tau_1.
\end{equation}
Combining \eqref{eq:Vbeta.Vt.infnorm.est1} and \eqref{eq:Vbeta.Vt.infnorm.est2}, we find that
\begin{equation}
\label{eq:Vbeta.Vt.infnorm.est3}
\underset{M_2 < x < +\infty}{\sup} \, \left| (1 + W_t(x))^{-1} - (1 + \omega_\beta (x))^{-1} \right| < \eps, \quad t > \tau_1.			
\end{equation}
Notice that for any $x \ge 0$ and $t > 0$
\begin{equation}
\label{eq:Vbeta.Vt.infnorm.est4}
\left| (1 + W_t(x))^{-1} - (1 + \omega_\beta (x))^{-1} \right|
\le \left| W_t(x) - \omega_\beta(x) \right|.
\end{equation}
We now apply Lemma~\ref{lem:reg.var.comp} to $[0, M_2]$ to deduce that there exists $\tau_2 \ge \tau_1$ such that
\begin{equation}
\label{eq:Vbeta.Vt.infnorm.est5}
\underset{0 \le x \le M_2}{\sup} \, \left| W_t(x) - \omega_\beta(x) \right| < \eps, \quad t > \tau_2,
\end{equation}
which, in conjunction with \eqref{eq:Vbeta.Vt.infnorm.est3} and \eqref{eq:Vbeta.Vt.infnorm.est4}, yields the desired claim.
\end{proof}

\begin{lemma}\label{lem:St.graphn}
Let $V = i V_2$ satisfy Assumption~\ref{asm:R}, $W_t$ be as in \eqref{eq:Wt.Om.def} and $\opS_t^0$ be the operator in $\Lt(\R)$ determined by
\begin{equation}
\label{eq:St.realline.def}
\opS_t^0 = \Nt + W_t
\end{equation}		
as in \eqref{eq:genairyde}. Then as $t \rightarrow +\infty$, we have
$
\Dom(\opS_t^0) = W^{1,2}(\R) \cap \Dom(V)
$
and	there exists $C>0$, independent of $t$, such that 
\begin{equation}
\label{eq:St.graphnorm.realline.def}
\| \opS_t^0 u \|^2 + \| u \|^2 \geq C ( \| u' \|^2 + \left\| W_t  u \right\|^2 + \| u \|^2), \quad u \in \Dom(\opS_t^0).
\end{equation}
The same statements hold true for $(\opS_t^0)^*$.
\end{lemma}
\begin{proof}
First observe that \eqref{eq:V.symb} with $n=1$ and \eqref{eq:VgrowsInfty.realline} imply that
\begin{equation}
\frac{|V_2'(s)|}{|V_2(s)|} \ls \frac 1s, \quad s \to +\infty,
\end{equation}	
and therefore for every $t>1$ and all sufficiently large $x$
\begin{equation}\label{eq:Wt.ln}
\log \frac{V_2(tx)}{V_2(x)} \leq \int_{x}^{tx} \frac{|V_2'(s)|}{|V_2(s)|} \, \dd s 
\ls \log t.
\end{equation}
Hence for every $t>1$, $\Dom(W_t) = \Dom(V)$.

Next, consider $\phi \in C_c^\infty((-2,2))$, $0 \le \phi \le 1$ such that $\phi = 1$ on $(-1, 1)$ and denote $\tilde{\phi}:= 1 - \phi$. We split $W_t$ as $W_t =  \phi W_t + \tilde \phi W_t$
and show that $\phi W_t$ is uniformly bounded and $\tilde \phi W_t$ satisfies \eqref{eq:A.W.sep.asm} uniformly in $t$. The claims then follow from Proposition~\ref{prop:A.sep}.

Firstly, by the locally uniform convergence of $W_t$ to $\omega_\beta$ (see Lemma~\ref{lem:reg.var.comp})
\begin{equation}\label{eq:Sttilde.boundedperturbation.realline}
	\| \phi \, W_t \|_\infty \le \| \phi (W_t - \omega_{\beta}) \|_\infty + \| \phi \omega_{\beta} \|_\infty \ls 1, \quad t \to +\infty.
\end{equation}
Secondly, 
\begin{equation}\label{eq:phi.W'.est}
	|(\tilde \phi(x)  W_t(x))' |  \le \| \tilde \phi' W_t \|_\infty + |\tilde\phi(x) W_t'(x)| \ls 1 +	|\tilde\phi(x) W_t'(x)|, \quad t \to +\infty,
\end{equation}
since $\supp \tilde{\phi}'$ is bounded and $W_t$ converges to $\omega_{\beta}$ locally uniformly. Moreover the last term in \eqref{eq:phi.W'.est} is estimated using \eqref{eq:Wt.symb} with $n=1$ and the fact that $\supp \tilde{\phi}$ is outside $(-1,1)$. Thus altogether we obtain
\begin{equation}
	|(\tilde \phi (x) W_t(x) )' | \ls 1 + \frac{\tilde{\phi}(x) W_t(x)}{\langle x \rangle},
\end{equation}
thus \eqref{eq:A.W.sep.asm} is indeed satisfied (uniformly for all sufficiently large $t$).
\end{proof}

\begin{proposition}\label{prop:local.R}
Define
\begin{equation}
\label{eq:Omega.def.realline}
\Omega_{a,\pm} := \left( \pm \xi_a - 2 \delta_a, \pm \xi_a + 2 \delta_a \right), 
\end{equation}
with $\xi_a$, $\delta_a$ as in \eqref{eq:deltaa.def}. Let the assumptions of Theorem~\ref{thm:R} hold and let $\widehat \opH$, $\widehat \opH_a$, $\opA_{\beta}$, $t_a$ and $\iota$ be as in \eqref{eq:Hhat.P.def}, \eqref{eq:Ha.def}, \eqref{eq:A.beta.def}, \eqref{eq:ta.def} and \eqref{eq:iota.def}, respectively. Then as $a \to + \infty$
\begin{equation}\label{eq:localestimate.realline}
\begin{aligned}
& \| \opA_{\beta}^{-1} \|^{-1} \, V_2(t_a) \, \left( 1 - \BigO \left( \iota(t_a) + a^{-\frac 12} \, t_a^{-1} \right) \right)
\\
& \qquad \qquad \le \inf \left\{ \frac{\| \widehat \opH_a u \|}{\|u\|}: \; 0 \neq u \in \Dom(\widehat \opH), \, \supp u \subset \Omega_{a,\pm} \right\}.
\end{aligned}
\end{equation}
\end{proposition}
\begin{proof}
We shall derive estimate \eqref{eq:localestimate.realline} for $u$ such that $\supp u \subset \Omega_{a,+}$. The procedure when $\supp u \subset \Omega_{a,-}$ is similar (see our comments at the end of the proof).

Clearly $\xi^2 - a = 2 \, \xi_a \, \left(\xi - \xi_a\right) + \left(\xi -\xi_a\right)^2$ and we introduce
\begin{equation*}
\label{eq:V.hat.realline.def}
\widetilde V_a(\xi) := -i  ( 2  \xi_a (\xi - \xi_a) + \left(\xi -\xi_a\right)^2 \chi_{\Omega_{a,+}}(\xi) ), \quad \xi \in \R.
\end{equation*}
With $\opwV$ as in \eqref{eq:Hhat.P.def}, let us define the following operator in $\Lt(\R)$
\begin{equation*}
\widetilde \opH_a = \opwV + \widetilde V_a(\xi), \quad \Dom(\widetilde \opH_a) = \left\{ u \in \Lt(\R) \, : \, \check{u} \in W^{1,2}(\R) \cap \Dom(V) \right\}.
\end{equation*}
Given $t > 0$ to be chosen below, we define a unitary operator on $L^2(\R)$ by
\begin{equation*}
(\opU_{a,t} \, u)(\xi) := t^{-\frac{1}{2}} \, u(t^{-1} \xi + \xi_a), \quad \xi \in \R.
\end{equation*}
Then with $\Omega_{a,t} := (-2  \delta_a  t,  2  \delta_a t)$
\begin{equation}\label{eq:Ha.til.def}
\frac{1}{V_2(t)} \opU_{a,t} \widetilde \opH_a \opU_{a,t}^{-1} = \sF W_t \sF^{-1} - i \frac{2\xi_a}{t V_2(t)} \xi - i \frac{1}{t^2 V_2(t)} \xi^2 \chi_{\Omega_{a,t}}(\xi).
\end{equation}
In what follows, we select $t$ as $t := t_a$, where $t_a$ is defined by equation \eqref{eq:ta.def}, \ie~
$t_a V_2(t_a) = 2 \xi_a$, and we recall that $t_a \to +\infty$ as $a \to +\infty$. We denote
\begin{equation}
\label{eq:Wahat.def}
\widehat R_a (\xi) := -\frac{i}{2} \frac{\xi^2}{\xi_a t_a}  \, \chi_{\Omega_{a,t_a}}(\xi), 
\end{equation}
and, from \eqref{eq:Wahat.def} and $\delta_a = \delta \xi_a$, we obtain
\begin{equation}\label{eq:Wahat.infnorm}
	\|\xi^{-1}\widehat R_a\|_\infty 
	=  \frac{\|\xi \, \chi_{\Omega_{a,t_a}}\|_\infty}{2 \xi_a t_a}  \le \delta, \qquad \| \xi^{-2} \widehat R_a  \|_{\infty} = \frac{1}{2 \xi_a  t_a}.
\end{equation}
We further denote 
\begin{equation}\label{eq:Sthat.realline.def}
\begin{aligned}
\widehat \opS_a^0 &:= \sF \, \opS_a^0 \, \sF^{-1} = \sF \, W_a \, \sF^{-1} - i \, \xi, 
\\
\Dom(\widehat \opS_a^0) &= \left\{ u \in \Lt(\R) \, : \, \check{u} \in W^{1,2}(\R) \cap \Dom(V) \right\},
\end{aligned}
\end{equation}
where (with an abuse of notation) $S_a^0:=S_{t_a}^0$ and $W_a:=W_{t_a}$ from Lemma~\ref{lem:St.graphn}. Our next aim is to show that 
\begin{equation}\label{eq:whSt.def}
\widehat \opS_{a}:= \widehat \opS_{a}^0 + \widehat R_{a}	
\end{equation}
converges to $\widehat \opS_\infty := \sF  \opA_{\beta}  \sF^{-1}$ in the norm resolvent sense as $a \to +\infty$. 

The spectra of $A_\beta$ and $\opS_{a}^0$, and hence those of $\widehat \opS_{\infty}$ and $\widehat \opS_a^0$, are empty, see Lemma~\ref{lem:St.graphn} and Proposition~\ref{prop:A.gen}. Moreover
\begin{equation}
\label{eq:Sthatminus1.convergence.Stminus1}
\| (\widehat \opS_a^0+1)^{-1} - (\widehat \opS_{\infty}+1)^{-1} \| 
=  \| (\opS_{a}^0+1)^{-1} - (\opA_{\beta}+1)^{-1}\|.
\end{equation}
Take $\phi_1, \, \phi_2 \in \SchwR$ and define 
\begin{equation}
\begin{aligned}
\psi_1 &:= \left(\opS_{a}^0 + 1\right)^{-1} \phi_1 \in \Dom(\opS_{a}^0) = W^{1,2}(\R) \cap \Dom(V),	
\\
\psi_2 &:= ((\opA_{\beta}+ 1)^{-1})^{*} \phi_2 \in \Dom(\opA_{\beta}^{*}) = W^{1,2}(\R) \cap \Dom(\omega_{\beta}).
\end{aligned}
\end{equation}
Then
\begin{align*}
& \langle ((\opS_{a}^0 + 1)^{-1} - (\opA_{\beta} + 1)^{-1}) \phi_1,  \phi_2 \rangle 
\\
& \qquad = 
\langle \psi_1,  \opA_{\beta}^{*} \psi_2 \rangle - \langle \opS_{a}^0  \psi_1,  \psi_2 \rangle
= \langle \psi_1, \omega_{\beta}  \psi_2 \rangle - \langle W_{a}  \psi_1,  \psi_2 \rangle
&\\
& \qquad = \intR (\omega_{\beta}(x) - W_{a}(x))  \psi_1(x)  \overline \psi_2 (x) \, \dd x
\\
&\qquad = \intR \left(\left(1 + W_{a}(x)\right)^{-1} - \left(1 + \omega_{\beta}(x)\right)^{-1}\right)
\varphi_1(x)  \overline \varphi_2 (x) \, \dd x,
\end{align*}
with $\varphi_1 := (1 + W_{a}) \psi_1$ and $\varphi_2 := (1 + \omega_{\beta})  \psi_2$. 
From the graph norm estimates \eqref{eq:St.graphnorm.realline.def} and \eqref{eq:Abeta.graphnorm.realline.def}, we obtain
\begin{equation*}
\| \varphi_1 \| = \| (1 + W_{a}) \, (\opS_{a}^0 + 1)^{-1}  \phi_1 \| \lesssim \|\phi_1 \|,
\quad
\| \varphi_2 \| = \| (1 + \omega_{\beta})  (\opA_{\beta}^* + 1)^{-1}  \phi_2 \| \lesssim \|\phi_2 \|.
\end{equation*}
Therefore, with $\iota$ from \eqref{eq:iota.def} and $\iota_a:=\iota(t_a)$,
\begin{equation*}
| \langle ((\opS_{a}^0 + 1)^{-1} - (\opA_{\beta} + 1)^{-1}) \phi_1,  \phi_2 \rangle | 
\le \iota_a  \| \varphi_1 \|  \| \varphi_2 \| 
\lesssim \iota_a \|\phi_1 \|  \| \phi_2 \|.
\end{equation*}
Hence by Lemma~\ref{lem:reg.var.conv}, the density of $\SchwR$ in $L^2(\R)$ and a standard resolvent identity argument, see \eg~the proof of \cite[Lem.~2.6.1]{Davies-1995}, we arrive at (employing \eqref{eq:Sthatminus1.convergence.Stminus1})
\begin{align}\label{eq:St.realline.normconvergence} 
\|(\widehat \opS_a^0)^{-1} - \widehat \opS_{\infty}^{-1} \| \lesssim \iota_a = o(1), 
\quad \| (\widehat \opS_a^0)^{-1} \| = \| \opA_{\beta}^{-1} \| (1 + \BigO(\iota_a))
\end{align}
as $a \to +\infty$. We transport the graph-norm estimate \eqref{eq:St.graphnorm.realline.def} to the Fourier side
\begin{equation}\label{eq:wSt0.sep}
\| \widehat \opS_a^0 u \|^2 + \| u \|^2 \gs  \| \xi \, u \|^2 + \| \sF  W_{a}  \sF^{-1}  u \|^2 + \| u \|^2, \quad u \in \Dom(\widehat \opS_a^0)
\end{equation}
and thus in particular (similarly as in the justification of \eqref{eq:x.A})
\begin{equation}\label{eq:Sthatminus1.realline.ubound1}
\| \xi (\widehat \opS_a^0)^{-1} \| + \| (\widehat \opS_a^0)^{-1}  \xi \| \lesssim 1, \quad a \rightarrow +\infty.
\end{equation}
Combining \eqref{eq:Sthatminus1.realline.ubound1} and \eqref{eq:Wahat.infnorm}, we deduce that $\|  \widehat R_a (\widehat \opS_a^0)^{-1} \| \ls \delta$ as $a \rightarrow +\infty$.

It follows, by choosing a sufficiently small $\delta > 0$, independently of $a$, that the bounded operator $\opI + \widehat R_a (\widehat \opS_a^0)^{-1}$ is invertible, for all large enough $a$, and
\begin{equation}\label{eq:wSt.St0}
\widehat \opS_{a}^{-1} =  (\widehat \opS_a^0)^{-1} ( \opI + \widehat R_a (\widehat \opS_a^0)^{-1} )^{-1}.
\end{equation}
This shows that $0 \in \rho(\widehat \opS_{a})$ for $a \rightarrow +\infty$ and furthermore, using \eqref{eq:Sthatminus1.realline.ubound1}, we deduce
\begin{equation}\label{eq:Sahatminus1.realline.ubound}
\| \xi \widehat \opS_{a}^{-1} \| + \| \widehat \opS_{a}^{-1} \xi \| \ls 1, \quad  a \rightarrow +\infty.  
\end{equation}
By the second resolvent identity, we have as $a \to +\infty$
\begin{equation*}
\| \widehat \opS_{a}^{-1} - (\widehat \opS_a^0)^{-1} \| 
= 
\| \widehat \opS_a^{-1}  \xi  \xi^{-2} \widehat R_a  \, \xi  (\widehat \opS_a^0)^{-1} \| 
\lesssim \| \xi^{-2}  \widehat R_a \|_{\infty}
\end{equation*}
where, for the last estimate, we have applied \eqref{eq:Sthatminus1.realline.ubound1} and \eqref{eq:Sahatminus1.realline.ubound}.
Thus from \eqref{eq:St.realline.normconvergence} and \eqref{eq:Wahat.infnorm}, we find
\begin{equation}\label{eq:hSt.conv}
\begin{aligned}
\| \widehat \opS_a^{-1} - \widehat \opS_{\infty}^{-1} \| 
&\lesssim \iota_a + (\xi_a t_a)^{-1} = o(1), \\
\|  \widehat \opS_a^{-1} \| &= \| \opA_{\beta}^{-1} \|  \left( 1 + \BigO \left(\iota_a + (\xi_a t_a)^{-1} \right) \right), \quad a \rightarrow +\infty.
\end{aligned}
\end{equation}

Noticing that $\widehat \opS_a = V_2(t_a)^{-1} \, \opU_{a,{t_a}} \, \widetilde \opH_a \, \opU_{a, {t_a}}^{-1}$ (see \eqref{eq:Ha.til.def}) and that $\| \widetilde \opH_a  u \| = \| \widehat \opH_a u \|$ for $0 \ne u \in \Dom(\widehat{\opH})$ such that $\supp u \subset \Omega_{a,+}$, we arrive at
\begin{equation*}
V_2(t_a) \| u \| = V_2(t_a)  \| \widetilde \opH_a^{-1} \widetilde \opH_a u \| 
\le \| \opA_{\beta}^{-1} \|  \left( 1 + \BigO \left(\iota_a + (\xi_a t_a)^{-1} \right) \right)  \| \widehat \opH_a  u \|,
\end{equation*}
as required.

For the case $\supp u \subset \Omega_{a,-}$, we repeat the above arguments but defining instead $\widetilde V_a(\xi) := i (2 \xi_a (\xi + \xi_a) - \left(\xi + \xi_a\right)^2 \chi_{\Omega_{a,-}}(\xi))$, $(\opU_{a,t} u)(\xi) := t^{-\frac{1}{2}} u(t^{-1} \xi - \xi_a)$, $\widehat \opS_a^0 := \sF (\opS_a^0)^* \sF^{-1} = \sF W_a \sF^{-1} + i \xi$ and $\widehat \opS_\infty := \sF \opA_{\beta}^* \sF^{-1}$.		
\end{proof}

\subsubsection{Step 3: lower estimate}

\begin{proposition}	\label{prop:lbound.R}
Let the assumptions of Theorem~\ref{thm:R} hold and let $\widehat \opH$, $\widehat \opH_a$, $\opA_{\beta}$, $t_a$ and $\iota$ be as in \eqref{eq:Hhat.P.def}, \eqref{eq:Ha.def}, \eqref{eq:A.beta.def}, \eqref{eq:ta.def} and \eqref{eq:iota.def}, respectively. Then there exist functions $0 \neq u_a \in \Dom(\widehat \opH)$ such that
\begin{equation*}
	\| \widehat \opH_a u_a \| = \| \opA_{\beta}^{-1} \|^{-1} \, V_2(t_a) \,  
	\left( 1 + \BigO \left( \iota(t_a) + (a^{\frac 12} \, t_a)^{-l_{\beta}} \right)  \right) \|u_a\|, \quad a \to + \infty,
\end{equation*}
where for any arbitrarily small $0 < \eps < \beta$
\begin{equation}
	\label{eq:lbeta.def.R}
	l_\beta := 
	\begin{cases}
		1, &  \beta > 1/2, \\[1mm] 
		1/2 + \beta - \eps, &  \beta \in (0,1/2].
	\end{cases}
\end{equation}
\end{proposition}
\begin{proof}
We retain the notation introduced in the proof of Proposition~\ref{prop:local.R}; in particular, $\widehat \opS_{\infty} = \sF A_\beta \sF^{-1}$ and $\widehat \opS_a$ from \eqref{eq:whSt.def}. The proof follows the steps of that of Proposition~\ref{prop:lbound.iR}.

With a sufficiently large $a_0$, let $g_a \in \Dom(\widehat \opS_{a}^*\widehat \opS_{a})$, $\|g_a\|=1$, $a \in (a_0, +\infty]$, such that 
\begin{equation}\label{eq:la.g.R}
\| \widehat \opS_{a} g_a\| = \varsigma_a^{-1} = \|\widehat \opS_{a}^{-1}\|^{-1}.
\end{equation}
Note that from \eqref{eq:hSt.conv} we obtain 
\begin{equation}\label{eq:lab.conv.R}
|\varsigma_a - \varsigma_\infty|  = \BigO \left( \iota_a + (\xi_a t_a)^{-1} \right), \quad  a \to +\infty.
\end{equation}

Consider $\psi_a \in C_c^{\infty}(( -2  \delta_a  {t_a},  2  \delta_a {t_a}) )$, $0 \le \psi_a \le 1$,  
$\psi_a = 1$ on $( - \delta_a {t_a}, \delta_a {t_a} )$ and such that
\begin{equation}\label{eq:psia'}
\| \psi_a^{(j)} \|_{\infty} \lesssim (\delta_a {t_a})^{-j} \approx (\xi_a t_a)^{-j}, \quad j \in \{1, 2, \dots, N + 1 + l\},
\end{equation}
with $N := \lceil\beta\rceil + 1$ and sufficiently large $l \in \N$ (see the statement of Lemma~\ref{lem:pdo.comp} and, in particular, \eqref{eq:R_N.est}). Recall that $t_a \to +\infty$ as $a \to +\infty$ (see \eqref{eq:ta.def}), hence $\psi_a \rightarrow 1$ pointwise in $\R$ as $a \to +\infty$. 

Next, we justify that $\psi_a g_a \in \Dom(\sF \, W_a \sF^{-1})$ and therefore $\psi_a g_a \in \Dom(\widehat \opS_{a})$. Similarly to \eqref{eq:Ltx.Lt.nonlocal.lbound}--\eqref{eq:Wt.nonlocal.lbound} (but estimating instead an upper bound), and using the locally uniform convergence of $W_a$ to $\omega_\beta$ (see Lemma~\ref{lem:reg.var.comp}), we find that $W_a(x) \ls \langle x \rangle^{\beta + \gamma}$, $x \in \R$, with any arbitrarily small $0 < \gamma < \beta$, for all sufficiently large $a$. Moreover, as in the proof of Lemma~\ref{lem:St.graphn}, consider $\phi \in C_c^\infty((-2,2))$, $0 \le \phi \le 1$ such that $\phi = 1$ on $(-1, 1)$ and denote $\tilde{\phi}:= 1 - \phi$. Then the estimate \eqref{eq:Wt.symb} and Leibniz rule show that there exist $C_n', C_n''>0$, independent of $a$, such that for all sufficiently large $a$, 
\begin{equation}\label{eq:phi.Wt.n}
|(\tilde \phi(x) W_a(x))^{(n)}| \leq C_n' (1+W_a(x))\langle x \rangle^{-n} \leq  C_n'' \langle x \rangle^{\beta + \gamma - n}, \quad x \in \R, \ n \in \N_0.
\end{equation}
Thus for sufficiently large $a$, $F:=\tilde \phi W_{a}$ satisfies the assumptions of Lemma~\ref{lem:pdo.comp} (with constants independent of $a$). Hence, for all $u \in \SchwR$, we have
\begin{equation}
\sF \tilde \phi W_{a} \sF^{-1}  \psi_a u = \psi_a \sF \tilde \phi W_{a} \sF^{-1} u   + 
[\sF \tilde \phi W_{a} \sF^{-1}, \psi_a] u
\end{equation}
and, using \eqref{eq:phi.Wt.n}, \eqref{eq:compositionformula}, \eqref{eq:comp.operator.remainder} and \eqref{eq:R_N.est},
\begin{align*}
	\|[\sF \tilde \phi W_{a} \sF^{-1}, \psi_a] u\| &\le \sum_{j=1}^{N} \frac{1}{j!} \| \psi_a^{(j)} \|_{\infty} \| \sF (\tilde \phi W_{a})^{(j)} \sF^{-1} u \| + \| \opR_{N+1} u \|\\
	&\le \sum_{j=1}^{N} C_j \| \psi_a^{(j)} \|_{\infty} \| \sF (1 + W_{a}) \sF^{-1} u \|\\
	&\quad + K_N \underset{0 \le j \le l}{\max} \left\{ \| \psi_a^{(N + 1 + j)}\|_\infty \right\} \| u \|\\
	&\le C'_N \underset{1 \le j \le N+1+l}{\max} \left\{ \| \psi_a^{(j)}\|_\infty \right\} (\| \sF W_{a} \sF^{-1} u \| + \| u \|),
\end{align*}
with $C'_N > 0$ independent of $a$. Hence by \eqref{eq:psia'}
\begin{equation}
	\label{eq:wa.phia.commutator.est}
	\|[\sF \, \tilde \phi \, W_{a} \sF^{-1}, \psi_a] u\| \ls (\xi_a t_a)^{-1} \left(\| \sF W_{a} \sF^{-1} u \| +  \|u\|\right).
\end{equation}
Since $W_a$ converges to $\omega_\beta$ uniformly on bounded sets (see Lemma~\ref{lem:reg.var.comp}), we have
\begin{equation}\label{eq:phiWt.b}
\|\sF \phi W_{a} \sF^{-1}\|  \ls 1.	
\end{equation}
Moreover, $\SchwR$ is a core for $\sF W_{a} \sF^{-1}$ and we conclude that $[\sF W_{a} \sF^{-1},  \psi_a]$ is relatively bounded w.r.t.~$\sF W_{a} \sF^{-1}$. Hence we have indeed $\psi_a g_a \in \Dom(\widehat \opS_{a})$.

Next, we write
\begin{equation}
\begin{aligned}
\widehat \opS_{a} \psi_a g_a &= \widehat \opS_{a} g_a + (\psi_a-1) \widehat \opS_{a} g_a + 
[\sF \, \tilde \phi \, W_{a} \sF^{-1}, \psi_a] g_a 
\\ & \quad + \sF \, \phi \, W_{a} \sF^{-1} (\psi_a-1) g_a - 	(\psi_a-1) \sF \, \phi \, W_{a} \sF^{-1} g_a
\end{aligned}
\end{equation}
and we proceed to estimate all the above terms but the first one. Employing \eqref{eq:lab.conv.R}, \eqref{eq:Sahatminus1.realline.ubound} as well as the graph norm separation as in \eqref{eq:wSt0.sep} for $\widehat S_a^0$  (and analogously for the adjoint $\widehat S_a^*$), we obtain as $a \to +\infty$
\begin{equation*}
\begin{aligned}
\| (\psi_a-1) \widehat \opS_{a} g_a \| &\ls \| (\psi_a-1) \xi^{-1} \|_\infty \| \xi (\widehat S_{a}^*)^{-1} \| \| \widehat S_{a}^* \widehat S_{a} g_a \| 
\ls (\xi_a t_a)^{-1},
\\
\|[\sF \, \tilde \phi \, W_{a} \sF^{-1}, \psi_a] g_a\|  
&\ls  (\xi_a t_a)^{-1} (\| \widehat S_{a} g_a \| +  \|g_a\|)
\ls  (\xi_a t_a)^{-1}, 
\\
\|\sF \, \phi \, W_{a} \sF^{-1} (\psi_a-1) g_a\| & \ls \|(\psi_a-1) \xi^{-1}\|_\infty \|\xi \widehat S_{a}^{-1}\| \| \widehat S_{a} g_a \|
\ls (\xi_a t_a)^{-1};
\end{aligned}
\end{equation*}
in the last two estimates we have also used \eqref{eq:wa.phia.commutator.est} and \eqref{eq:phiWt.b}, respectively. Since furthermore $\|\phi (W_{a} -\omega_\beta)\|_\infty \ls \iota_a$, then
\begin{equation*}
	\|(\psi_a-1) \sF \, \phi \, W_{a} \sF^{-1} g_a\| \ls \iota_a + \|(\psi_a-1) \sF \, \phi \, \omega_\beta \sF^{-1} g_a\|.
\end{equation*}

For $\beta>1/2$, we have
\begin{equation*}
\begin{aligned}
\|(\psi_a-1) \sF \, \phi \, \omega_{\beta} \sF^{-1} g_a\| &\ls \|(\psi_a-1) \xi^{-1}\|_\infty \| \xi \sF \, \phi \, \omega_\beta \sF^{-1} g_a\| 
\\		
&\ls (\xi_a t_a)^{-1} \|  (\phi \, \omega_\beta \check g_a)'\| \ls (\xi_a t_a)^{-1},
\end{aligned}
\end{equation*}
where in the last step we use $\| (\phi \, \omega_\beta)' \| \ls 1$,  $\| \phi \, \omega_\beta \|_{\infty} \ls 1$ and
\begin{equation*}
\begin{aligned}
\|\check g_a\|_\infty &\ls \|g_a\|_1 \ls \|\langle \xi \rangle g_a\| \ls 1, \qquad 
\|\check g_a'\| \ls  \|\xi g_a\| \ls 1.	
\end{aligned}
\end{equation*}

For $\beta \in (0, 1/2]$ and $0 < \eps < \beta$, we define $1/2 < l_{\beta,\eps} := 1/2 + \beta - \eps < 1$ and we fix some $0 < \hat \eps < \eps$. Then
\begin{equation*}
	\begin{aligned}
		\|(\psi_a-1) \sF \, \phi \, \omega_\beta \sF^{-1} g_a \| &\le \|(\psi_a-1) \langle \xi \rangle^{-l_{\beta,\eps}}\|_\infty \| \langle \xi \rangle^{l_{\beta,\eps}} \sF \, \phi \, \omega_\beta \sF^{-1} g_a \|\\
		&\ls (\xi_a t_a)^{-l_{\beta,\eps}} \| \langle \xi \rangle^{l_{\beta,\eps}} \sF \, \phi \, \omega_\beta \check g_a \|\\
		&\le (\xi_a t_a)^{-l_{\beta,\eps}} \left( \| \chi_{\{|\xi| < 1\}} \langle \xi \rangle^{l_{\beta,\eps}} \sF \, \phi \, \omega_\beta \check g_a \| \right.\\
		&\qquad\qquad\qquad  \left. + \| \chi_{\{|\xi| \ge 1\}} \langle \xi \rangle^{l_{\beta,\eps}} \sF \, \phi \, \omega_\beta \check g_a \| \right).
	\end{aligned}
\end{equation*}
Since $\| \phi \, \omega_{\beta} \check g_a \| \ls \| \check g_a \| = 1$, we conclude that
\begin{equation*}
	 \| \chi_{\{|\xi| < 1\}} \langle \xi \rangle^{l_{\beta,\eps}} \sF \, \phi \, \omega_\beta \check g_a \| \ls 1.
\end{equation*}
For the second term, we use the facts that $\| \phi \, \omega_\beta \check g_a' \| \ls \| \check g_a' \| \ls 1$, $\| \phi' \, \omega_\beta \check g_a \| \ls \| \check g_a \| = 1$ and $\| \phi \, \omega_\beta' \check g_a \|_p \le \| \check g_a \|_{\infty} \| \phi \, \omega_\beta' \|_p \ls 1$, with $p := (1 - \beta + \hat \eps)^{-1} \in (1, (1 - \beta)^{-1}) \subset (1, 2)$. 
Then, by the Hausdorff-Young inequality (see e.g. \cite[Prop.~2.2.16]{Grafakos-2014-249}), we have that $\| \sF \, \phi \, \omega_\beta' \check g_a \|_q \ls 1$, with $q = p/(p-1) \in (\beta^{-1}, \infty) \subset (2, \infty)$. Thus we find
\begin{align*}
	\| \chi_{\{|\xi| \ge 1\}} \langle \xi \rangle^{l_{\beta,\eps}} \sF \, \phi \, \omega_\beta \check g_a \| &\ls \| \chi_{\{|\xi| \ge 1\}} \langle \xi \rangle^{l_{\beta,\eps} - 1} \sF \, (\phi \, \omega_\beta \check g_a)' \|\\
	&\ls 1 + \| \chi_{\{|\xi| \ge 1\}} \langle \xi \rangle^{l_{\beta,\eps} - 1} \sF \, \phi \, \omega_\beta' \check g_a \|\\
	&\le 1 + \| \langle \xi \rangle^{2 (l_{\beta,\eps} - 1)} \|_{p'}^{\frac12} \| \sF \, \phi \, \omega_\beta' \check g_a \|_q \ls 1,
\end{align*}
where in the last step we have applied H\"{o}lder's inequality with $p' = p/(2 - p)$. Hence when $\beta \in (0, 1/2]$ we have
\begin{equation*}
	\| (\psi_a - 1) \sF \, \phi \, \omega_{\beta} \check g_a\| \ls (\xi_a t_a)^{-l_{\beta,\eps}}.
\end{equation*}

In summary, writing $\psi_a g_a = g_a + (1-\psi_a)g_a$, we obtain as $a \to +\infty$
\begin{equation}
\|\widehat \opS_{a} \psi_a g_a\| = \varsigma_a^{-1} + \BigO(\iota_a+  (\xi_a t_a)^{-l_{\beta}} ),
\quad 
\|\psi_a g_a\| = 1 + \BigO((\xi_a t_a)^{-1}).
\end{equation}
Thus using \eqref{eq:lab.conv.R}, we arrive at
\begin{equation*}
\left| \frac{\left\| \widehat \opS_{a} \psi_a g_a \right\|}{\| \psi_a g_a\|} - \frac1{\varsigma_\infty} \right| 
=  \BigO(\iota_a+  (\xi_a t_a)^{-l_{\beta}} ), \quad a \rightarrow +\infty.
\end{equation*}
Recalling that $\widehat \opS_a = V_2(t_a)^{-1} \, \opU_{a,{t_a}} \, \widetilde \opH_a \, \opU_{a, {t_a}}^{-1}$ (see \eqref{eq:Ha.til.def}) and letting $u_a := \opU_{a, {t_a}}^{-1} \psi_a g_a$, then $u_a \in \Dom(\widehat \opH)$ with $\supp u_a \subset \Omega_{a,+}$. We therefore conclude
\begin{equation*}
	\left| \frac{\| \widehat \opH_{a} u_a \|}{\| u_a \|} - \frac{V_2(t_a)}{\varsigma_\infty} \right| 
	=  \BigO( V_2(t_a) (\iota_a+  (\xi_a t_a)^{-l_{\beta}}) ), \quad a \rightarrow +\infty
\end{equation*}
and the claim follows.
\end{proof}

\subsubsection{Step 4: combining the estimates}
\label{sssec:step.4.R}
With $\Omega_{a,\pm}'$, $\Omega_{a,\pm}$ and $\delta_a$ from \eqref{eq:deltaa.def}, \eqref{eq:Omega.def.realline}, let $\phi_{a,\pm} \in C_c^{\infty}(\Omega_{a,\pm})$, $0 \le \phi_{a,\pm} \le 1$, be such that
\begin{equation}
	\label{eq:phia}
	\phi_{a,\pm}(\xi) = 1, \; \xi \in \Omega'_{a,\pm}, \quad \|\phi_{a,\pm}^{(j)}\|_{\infty} \ls \delta_a^{-j}, \quad j \in \{1, 2, \dots, N + 1 + l\},
\end{equation}		
with $N := \max\{\lceil\beta\rceil + 1, 3\}$ and sufficiently large $l \in \N$ (see the statement of Lemma~\ref{lem:pdo.comp} and, in particular, \eqref{eq:R_N.est}), and define
\begin{equation}
	\label{eq:phiak}
	\phi_{a,0} := 1 - (\phi_{a,+} + \phi_{a,-}), \quad \phi_{a,1} := \phi_{a,+}, \quad \phi_{a,2} := \phi_{a,-}.
\end{equation}
\begin{lemma}
	\label{lem:Vp.holder}
	Let $V = i V_2$ satisfy Assumption~\ref{asm:R} with $\beta > 1$ and let $p_\beta := 1 + 1 / (\beta - 1) - \eps$, with $0 < \eps < 1 / (\beta -1)$ arbitrarily small, and $q_\beta := p_\beta / (p_\beta - 1)$. Then for any $u \in \SchwR$ and $j \in \N$
	\begin{equation}
		\label{eq:Vp.holder.R}
		\| V^{(j)} \check u \| \ls \| u \| + \| (1 + V_2) \check u \|^{\frac{1}{p_\beta}} \| u \|^{\frac{1}{q_\beta}}.
	\end{equation}
\end{lemma}
\begin{proof}
	Let $u \in \SchwR$ and $j \in \N$, then by \eqref{eq:V.symb} and H\"{o}lder's inequality
	\begin{equation*}
		\| V_2^{(j)} \check u \| \ls \| (1 + V_2) \langle x \rangle^{-j} \check u \| \ls \| (1 + V_2) \langle x \rangle^{-1} \check u \| \ls \| u \| + \| (V_2 \langle x \rangle^{-1})^{p{_\beta}} \check u \|^{\frac{1}{p_\beta}} \| u \|^{\frac{1}{q_\beta}}.
	\end{equation*}
	From Assumption~\ref{asm:R}~\ref{itm:regularlyvarying} (note also \eqref{eq:V.regvarying.decomposition} and \eqref{eq:slowvarying.estimates}), we have for any $\gamma > 0$ and sufficiently large $|x|$
	\begin{equation*}
	\langle x \rangle^{\beta - 1 - \gamma} \ls	V_2(x) \langle x \rangle^{-1} \ls \langle x \rangle^{\beta - 1 + \gamma}.
	\end{equation*}
	Therefore, given $\eps > 0$, by choosing $\gamma > 0$ sufficiently small we have
	\begin{equation*}
		(V_2(x) \langle x \rangle^{-1})^{p_{\beta}} \ls \langle x \rangle^{\beta - \gamma} \ls V_2(x), \quad |x| \to +\infty,
	\end{equation*}
	and \eqref{eq:Vp.holder.R} follows.
\end{proof}
\begin{lemma}
	\label{lem:Vhat.phiak.commutator}
	Let the assumptions of Theorem~\ref{thm:R} hold, with $\widehat V$, $\widehat \opH_a$, $t_a$ and $\beta$ as in \eqref{eq:Hhat.P.def}, \eqref{eq:Ha.def}, \eqref{eq:ta.def} and \eqref{eq:V.beta.def.R}, respectively, and let $\phi_{a,k}, \; k \in \{0, 1, 2\},$ be as in \eqref{eq:phiak}. Then for all $u \in \SchwR$ and $k \in \{0, 1, 2\}$, we have
	\begin{equation}
		\label{eq:Hahatphiak.commutator.norm.est}
		\| [\opwV, \phi_{a,k}] u \| \ls a^{-\frac12} t_a^{-1} \| \widehat \opH_a u \| + \Theta(a, \eps) \| u \|, \quad a \to +\infty,
	\end{equation}
	where for any arbitrarily small $\eps > 0$
	\begin{equation}
		\label{eq:Hahatphiak.commutator.coeff}
		\Theta(a, \eps) := 
		\begin{cases}
			a^{-1}, &  \beta < 2, \\[1mm]
			a^{-\frac12} t_a^{\beta - 1 + \eps}, &  \beta \ge 2.
		\end{cases}
	\end{equation}
	Moreover
	\begin{equation}
		\label{eq:Hahatphiak.commutator.decay}
		\Theta(a, \eps) (V_2(t_a))^{-1} (a^{\frac12} t_a)^{1 - \eps} = o(1), \quad a \to +\infty.
	\end{equation}
\end{lemma}
\begin{proof}
	Let $u \in \SchwR \subset \Dom(\widehat \opH)$, then
	\begin{equation}
		\label{eq:HaV.est}
		\| \opwV u \| \ls \| \widehat \opH_a u \| + a \| u \|,
	\end{equation}
	(see \eqref{eq:Hhat.real.separation.prop} and \eqref{eq:Ha.def}). Note also that
	\begin{equation}
		\label{eq:HaPVF.est}
		\| V_2^{\frac12}  \check{u} \| = \langle V_2 \check{u}, \check{u}\rangle^{\frac12} 
		= \langle \opwV  u, u \rangle^{\frac12} 
		= ( \Re  \langle \widehat \opH_a  u, u \rangle )^{\frac12} \le \| \widehat \opH_a u \|^{\frac12} \| u \|^{\frac12}.
	\end{equation}
	Furthermore, by Assumption~\ref{asm:R}~\ref{itm:regularlyvarying}, and recalling our earlier remarks on regularly varying functions, in particular \eqref{eq:V.regvarying.decomposition} and \eqref{eq:slowvarying.estimates}, we obtain from \eqref{eq:ta.def}
	\begin{equation}
		\label{eq:ta.est}
		t_a^{\beta + 1 - \gamma} < t_a V_2(t_a) = 2 a^{\frac12} < t_a^{\beta + 1 + \gamma},
	\end{equation}
	for any arbitrarily small $\gamma > 0$ and any sufficiently large $a$.
	
	In the case $\beta < 2$, appealing to \eqref{eq:compositionformula}, \eqref{eq:R_N.est} and \eqref{eq:phia} and noting that $\sF V_2^{(j)} \sF^{-1}$ are bounded operators for $j \ge 2$ (recall Assumption~\ref{asm:R}~\ref{itm:symbolclass}), we have for any $k \in \{0, 1, 2\}$
	\begin{align*}
		\| [\opwV, \phi_{a,k}] u \| &\ls \| \phi_{a,k}' \sF V_2' \sF^{-1} u \| + a^{-1} \| u \|\\
		&\ls \| \sF V_2' \sF^{-1} \phi_{a,k}' u \| + \| [\sF V_2' \sF^{-1}, \phi_{a,k}'] u \| + a^{-1} \| u \|\\
		&\ls \| \sF V_2' \sF^{-1} \phi_{a,k}' u \| + a^{-1} \| u \|.
	\end{align*}
	Moreover, using \eqref{eq:V.symb}, \eqref{eq:HaPVF.est}, $\beta < 2$ and the fact that $t_a \to +\infty$ as $a \to +\infty$
	\begin{align*}
		\| \sF V_2' \sF^{-1} \phi_{a,k}' u \| &\ls \| (1 + V_2) \langle x \rangle^{-1} \sF^{-1} \phi_{a,k}' u \| \ls \| \phi_{a,k}' u \| + \| V_2^{\frac12} \sF^{-1} \phi_{a,k}' u \|\\
		&\ls t_a^{-1} \| \widehat \opH_a \phi_{a,k}' u \| + t_a \| \phi_{a,k}' u \|.
	\end{align*}
Since $\supp \phi_{a,k}' u \cap (\Omega'_{a,+} \cup \Omega'_{a,-}) = \emptyset$, we have applying \eqref{eq:Hahat.away.est.R}
	\begin{equation*}
		\| \phi_{a,k}' u \| \ls a^{-1} \| \widehat \opH_a \phi_{a,k}' u \|
	\end{equation*}
	and, since $t_a^2 a^{-1} \to 0$ as $a \to +\infty$ (see \eqref{eq:ta.def} and \eqref{eq:VgrowsInfty.realline}), we conclude
	\begin{equation*}
		\| \sF V_2' \sF^{-1} \phi_{a,k}' u \| \ls t_a^{-1} \| \widehat \opH_a \phi_{a,k}' u \| \ls t_a^{-1} (a^{-\frac12} \| \widehat \opH_a u \| + \| [\opwV, \phi_{a,k}'] u \|).
	\end{equation*}
	Furthermore, by \eqref{eq:compositionformula}, \eqref{eq:R_N.est}, \eqref{eq:V.symb}, \eqref{eq:HaPVF.est} and $\beta < 2$
	\begin{align*}
		\| [\opwV, \phi_{a,k}'] u \| &\ls a^{-1} \| \sF V_2' \sF^{-1} u \| + a^{-\frac32} \| u \| \ls a^{-1} (\| u \| + \| V_2^{\frac12} \check{u} \|) + a^{-\frac32} \| u \|\\
		&\ls a^{-1} (\| \widehat \opH_a u \| + \| u \|).
	\end{align*}
	Hence
	\begin{equation}\label{eq:FV'F.phi.est}
		\| \sF V_2' \sF^{-1} \phi_{a,k}' u \| \ls a^{-\frac12} t_a^{-1} \| \widehat \opH_a u \| + a^{-1} t_a^{-1} \| u \|
	\end{equation}
	and
	\begin{equation}
		\label{eq:commutator.Vphi.betalt2.est}
		\| [\opwV, \phi_{a,k}] u \| \ls a^{-\frac12} t_a^{-1} \| \widehat \opH_a u \| + a^{-1} \| u \|
	\end{equation}
	as claimed. Moreover, using \eqref{eq:ta.def} and \eqref{eq:VgrowsInfty.realline}
	\begin{equation*}
		a^{-1} (V_2(t_a))^{-1} a^{\frac12} t_a = 2 (V_2(t_a))^{-2} \to 0, \quad a \to +\infty.
	\end{equation*}

	For $\beta \ge 2$, applying \eqref{eq:compositionformula} we obtain
	\begin{equation}
		\label{eq:Vhat.phiak.commutator.est}
		\| [\opwV, \phi_{a,k}] u \| \ls \| \phi_{a,k}' \sF V_2' \sF^{-1} u \| + \| \sum_{j = 2}^{N} \phi_{a,k}^{(j)} \sF V_2^{(j)} \sF^{-1} u + \opR_{N+1,k} u \|.
	\end{equation}
	In order to estimate $\| \phi_{a,k}' \sF V_2' \sF^{-1} u \|$, we introduce cut-off functions $\eta_{a,k} \in \CcR$, $0 \le \eta_{a,k} \le 1$, $k \in \{0, 1, 2\}$, satisfying
	\begin{equation}
		\label{eq:etaak}	
		\begin{aligned}
			\eta_{a,k}(\xi) &= 1, \quad \xi \in \supp \phi_{a,k}', \quad \| \eta_{a,k}^{(j)}\|_{\infty} \ls a^{-\frac j2}, \quad j \in \{1, \dots, N + 1 + l\},\\
			\supp \eta_{a,k} &\cap \left( (-\xi_a - \delta' \xi_a, -\xi_a + \delta' \xi_a) \cup (\xi_a - \delta' \xi_a, \xi_a + \delta' \xi_a) \right) = \emptyset,
		\end{aligned}
	\end{equation}
	with $0 < \delta' < \delta$. Then applying Lemma~\ref{lem:Vp.holder} and Young's inequality for products (note that $p_\beta, q_\beta \in (1, +\infty)$) and using the fact that $t_a \to +\infty$ as $a \to +\infty$
	\begin{equation}
		\label{eq:phiakp.Vp.est}
		\begin{aligned}
			\| \phi_{a,k}' \sF V_2' \sF^{-1} u \| &= \| \phi_{a,k}' \eta_{a,k} \sF V_2' \sF^{-1} u \| \ls a^{-\frac12} \| \eta_{a,k} \sF V_2' \sF^{-1} u \|\\
			&\ls a^{-\frac12} ( \| \sF V_2' \sF^{-1} \eta_{a,k} u \| + \| [\sF V_2' \sF^{-1}, \eta_{a,k}] u \| )\\
			&\ls a^{-\frac12} ( \| \eta_{a,k} u \| + \| (1 + \opwV) \eta_{a,k} u \|^{\frac{1}{p_\beta}} \| \eta_{a,k} u \|^{\frac{1}{q_\beta}}\\
			&\qquad\qquad + \| [\sF V_2' \sF^{-1}, \eta_{a,k}] u \| )\\
			&\ls a^{-\frac12} ( t_a^{-1} \| \opwV \eta_{a,k} u \| + t_a^{\frac{1}{p_{\beta} - 1}}\| u \| + \| [\sF V_2' \sF^{-1}, \eta_{a,k}] u \| ).
		\end{aligned}
	\end{equation}
	Since $\supp \eta_{a,k} u \cap \left( (-\xi_a - \delta' \xi_a, -\xi_a + \delta' \xi_a) \cup (\xi_a - \delta' \xi_a, \xi_a + \delta' \xi_a) \right) = \emptyset$, using \eqref{eq:Hahat.away.est.R} we have
	\begin{equation*}
		\| \opwV \eta_{a,k} u \| \ls \| \widehat \opH_a \eta_{a,k} u \| \ls \| \widehat \opH_a u \| + \| [\opwV, \eta_{a,k}] u \|.
	\end{equation*}
	Applying \eqref{eq:compositionformula} to $[\opwV, \eta_{a,k}]$ and using \eqref{eq:etaak} and the fact that, by \eqref{eq:V.symb} and \eqref{eq:HaV.est}, $\| \sF V_2^{(j)} \sF^{-1} u \| \ls \| \widehat \opH_a u \| + a \| u \|$ for any $j \in \N$, we obtain
	\begin{equation}
		\label{eq:Vhatetaak.commutator}
		\begin{aligned}
			\| [\opwV, \eta_{a,k}] u \| &\ls \| \eta_{a,k}' \sF V_2' \sF^{-1} u \| + \sum_{j = 2}^{N} \| \eta_{a,k}^{(j)} \sF V_2^{(j)} \sF^{-1} u \| + \| \widetilde \opR_{N+1,k} u \|\\
			&\ls a^{-\frac12} \| \sF V_2' \sF^{-1} u \| + a^{-1} \| \widehat \opH_a u \| + \| u \|.
		\end{aligned}
	\end{equation}
	Furthermore, noting firstly that $a^{\frac12} t_a^{-2} \approx V_2(t_a) t_a^{-1} \to +\infty$ and secondly that, for sufficiently small $\eps, \gamma > 0$, $a t_a^{-2 \beta - \BigO(\eps)} \gs t_a^{2 - 2\gamma - \BigO(\eps)} \to +\infty$ as $a \to +\infty$ (see \eqref{eq:ta.def} and \eqref{eq:ta.est}), we have by \eqref{eq:Vp.holder.R}, \eqref{eq:HaV.est} and Young's inequality for products
	\begin{align*}
		\| [\opwV, \eta_{a,k}] u \| &\ls a^{-\frac12} ( \| u \| + t_a^{-2} \| (1 + \opwV) u \| + t_a^{\frac{2}{p_{\beta} - 1}} \|u \| ) + a^{-1} \| \widehat \opH_a u \| + \| u \|\\
		&\ls a^{-\frac12} ( t_a^{-2} \| \widehat \opH_a u \| + a t_a^{-2} \| u \| + t_a^{2 (\beta - 1) + \BigO(\eps)} \| u \| ) + a^{-1} \| \widehat \opH_a u \|\\
		&\ls a^{-\frac12} t_a^{-2} \| \widehat \opH_a u \| + a^{\frac12} t_a^{-2} \| u \|,
	\end{align*}
	which yields
	\begin{equation}
		\label{eq:Vhat.etaak.est}
		\| \opwV \eta_{a,k} u \| \ls \| \widehat \opH_a u \| + a^{\frac12} t_a^{-2} \| u \|.
	\end{equation}
	Moreover, repeating the same arguments which we have used for $[\opwV, \eta_{a,k}] u$, we find
	\begin{equation}
		\label{eq:Vphatetaak.commutator}
		\begin{aligned}
			\| [\sF V_2' \sF^{-1}, \eta_{a,k}] u \| &\ls a^{-\frac12} \| \sF V_2'' \sF^{-1} u \| + a^{-1} \| \widehat \opH_a u \| + \| u \|\\
			&\ls a^{-\frac12} t_a^{-2} \| \widehat \opH_a u \| + a^{\frac12} t_a^{-2} \| u \|.
		\end{aligned}
	\end{equation}
	Returning to \eqref{eq:phiakp.Vp.est} with \eqref{eq:Vhat.etaak.est} and \eqref{eq:Vphatetaak.commutator} and noting that for any small but fixed $\eps > 0$ we can always find $\gamma > 0$ such that 
\begin{equation*}
a^{-\frac12} t_a^{\beta + 1 + \BigO(\eps)} \approx t_a^{\beta + \BigO(\eps)} (V_2(t_a))^{-1} \gs t_a^{\BigO(\eps) - \gamma} \to +\infty, \quad a \to +\infty,
\end{equation*}
 we obtain
	\begin{equation}
		\label{eq:phiakp.Vphat.est}
		\begin{aligned}
			\| \phi_{a,k}' \sF V_2' \sF^{-1} u \| &\ls a^{-\frac12} ( t_a^{-1} \| \widehat \opH_a u \| + t_a^{\beta - 1 + \BigO(\eps)} \| u \| + a^{\frac12} t_a^{-2} \| u \| )\\
			&\ls a^{-\frac12} t_a^{-1} \| \widehat \opH_a u \| + a^{-\frac12} t_a^{\beta - 1 + \BigO(\eps)} \| u \|.
		\end{aligned}
	\end{equation}
	Next we estimate the second term in the right-hand side of \eqref{eq:Vhat.phiak.commutator.est} using \eqref{eq:phia}, \eqref{eq:Vp.holder.R}, \eqref{eq:HaV.est}, $N \ge 3$, Young's inequality for products and $a t_a^{-2 \beta - \BigO(\eps)} \gs t_a^{2 - 2\gamma - \BigO(\eps)} \to +\infty$ as $a \to +\infty$
	\begin{equation}
		\label{eq:Vhat.phiak.commutator.est2}
		\begin{aligned}
			&\| \sum_{j = 2}^{N} \phi_{a,k}^{(j)} \sF V_2^{(j)} \sF^{-1} u + \opR_{N+1,k} u \| 
			\\ 
			& \quad \ls a^{-1} \sum_{j = 2}^{N} \| \sF V_2^{(j)} \sF^{-1} u \| + a^{-2} \| u \| \ls a^{-1} ( t_a^{-2} \| (1 + \opwV) u \| + t_a^{\frac{2}{p_{\beta} - 1}} \| u \| ) \\
			&\quad \ls a^{-1} ( t_a^{-2} \| \widehat \opH_a u \| + t_a^{-2} a \| u \| + t_a^{2 (\beta - 1) + \BigO(\eps)} \| u \| )
			\\
		&\quad \ls a^{-1} t_a^{-2} \| \widehat \opH_a u \| + t_a^{-2} \| u \|.
		\end{aligned}
	\end{equation}
	Combining \eqref{eq:Vhat.phiak.commutator.est}, \eqref{eq:phiakp.Vphat.est} and \eqref{eq:Vhat.phiak.commutator.est2}, we have
	\begin{equation*}
		\| [\opwV, \phi_{a,k}] u \| \ls a^{-\frac12} t_a^{-1} \| \widehat \opH_a u \| + a^{-\frac12} t_a^{\beta - 1 + \BigO(\eps)} \| u \|,
	\end{equation*}
	as required. Finally, using \eqref{eq:ta.def} and \eqref{eq:ta.est}
	\begin{equation*}
		t_a^{\beta + \eps} (V_2(t_a))^{-1} (a^{\frac12} t_a)^{-\eps} \approx t_a^{\beta - \eps} (V_2(t_a))^{-1 - \eps} \ls t_a^{-(1 + \beta - \gamma) \eps + \gamma} \to 0, \quad a \to +\infty,
	\end{equation*}
	since $\gamma > 0$ can be chosen arbitrarily small. This concludes the proof.
\end{proof}
\begin{lemma}
	\label{lem:Hhat.u1u2.quasiorthogonal}
	Let the assumptions of Theorem~\ref{thm:R} hold, with $\widehat \opH_a$, $t_a$ and $\Theta(a, \eps)$ as in \eqref{eq:Ha.def}, \eqref{eq:ta.def} and \eqref{eq:Hahatphiak.commutator.coeff}, respectively, and let $\phi_{a,k}, \; k \in \{1, 2\},$ be as in \eqref{eq:phiak}. Then for all $u \in \SchwR$ and any arbitrarily small $\eps > 0$, we have as $a \to +\infty$
	\begin{equation}
		\label{eq:Hhat.u1u2.quasiorth.norm.est}
		(\| \widehat \opH_a \phi_{a,1} u \|^2 + \| \widehat \opH_a \phi_{a,2} u \|^2)^{\frac12} = \| \widehat \opH_a (\phi_{a,1} + \phi_{a,2}) u \| + \BigO(a^{-\frac12} t_a^{-1} \| \widehat \opH_a u \| + \Theta(a, \eps) \| u \|).
	\end{equation}
\end{lemma}
\begin{proof}
	Let $u \in \SchwR$ and $u_k := \phi_{a,k} u$ with $k \in \{1,2\}$. Applying \eqref{eq:compositionformula} with $F = V_2$ and $\phi = \phi_{a,k}$, we have for $k \in \{1, 2\}$
	\begin{equation*}
		\widehat \opH_a u_k = \phi_{a,k} \widehat \opH_a u + [\widehat V, \phi_{a,k}] u = \opB_{N,k} u + \opR_{N+1,k} u
	\end{equation*}
	with
	\begin{equation*}
		\opB_{N,k} u := \phi_{a,k} \widehat \opH_{a} u + \sum_{j=1}^{N} \frac{i^j}{j!} \phi_{a,k}^{(j)} \widehat V^{(j)} u
	\end{equation*}
	and $\opR_{N+1,k} u$ as in \eqref{eq:comp.operator.remainder}, \eqref{eq:comp.symbol.remainder.1} and \eqref{eq:comp.symbol.remainder.2}. Noting that $\supp (\opB_{N,1} u) \subset \Omega_{a,+}$ and $\supp (\opB_{N,2}) u \subset \Omega_{a,-}$, and consequently $\opB_{N,1} u \perp \opB_{N,2} u$ in $L^2$, we get
	\begin{equation*}
		\begin{aligned}
			\| \widehat \opH_{a} (u_1 + u_2) \|^2 &= \| \widehat \opH_{a} u_1 \|^2 + \| \widehat \opH_{a} u_2 \|^2 + 2 \Re \langle \widehat \opH_{a} u_1, \widehat \opH_{a} u_2 \rangle\\
			&= \| \widehat \opH_{a} u_1 \|^2 + \| \widehat \opH_{a} u_2 \|^2 + 2 \Re \langle \opB_{N,1} u, \opR_{N+1,2} u \rangle\\
			&\quad + 2 \Re \langle \opR_{N+1,1} u, \opB_{N,2} u \rangle + 2 \Re \langle \opR_{N+1,1} u, \opR_{N+1,2} u \rangle.
		\end{aligned}
	\end{equation*}
	Hence
	\begin{equation}
		\label{eq:Hhat.u1u2.quasiorth.est.1}
		\begin{aligned}
			|(\| \widehat \opH_{a} u_1 \|^2 + \| \widehat \opH_{a} u_2 \|^2)^{\frac12} - \| \widehat \opH_{a} (u_1 + u_2) \||  &\ls \| \opB_{N,1} u \|^{\frac12} \| \opR_{N+1,2} u \|^{\frac12}\\
			&\quad + \| \opR_{N+1,1} u \|^{\frac12} \| \opB_{N,2} u \|^{\frac12}\\
			&\quad + \| \opR_{N+1,1} u \|^{\frac12} \| \opR_{N+1,2} u \|^{\frac12}.
		\end{aligned}
	\end{equation}
	Applying \eqref{eq:R_N.est} and \eqref{eq:phia} and recalling $N \ge 3$, we find $\| \opR_{N+1,k} u \| \ls a^{-2} \| u \|$ for $k \in \{1, 2\}$ and $a \to +\infty$. Moreover
	\begin{equation}
		\label{eq:Bnk.comm.expansion}
		\| \opB_{N,k} u \| \le \| \widehat \opH_{a} u \| + \| \phi_{a,k}' \sF V_2' \sF^{-1} u \| + \sum_{j=2}^{N} \frac1{j!} \| \phi_{a,k}^{(j)} \widehat V^{(j)} u \|,
	\end{equation}
	for $k \in \{1, 2\}$. The second and higher order terms in the right-hand side of the above inequality have already been estimated in Lemma~\ref{lem:Vhat.phiak.commutator} (see \eqref{eq:FV'F.phi.est}, \eqref{eq:commutator.Vphi.betalt2.est}, \eqref{eq:phiakp.Vphat.est} and \eqref{eq:Vhat.phiak.commutator.est2}); we have for $k \in \{1, 2\}$ and any arbitrarily small $\eps > 0$
	\begin{equation*}
		\| \opB_{N,k} u \| \le (1 + \BigO(a^{-\frac12} t_a^{-1})) \| \widehat \opH_{a} u \| + \BigO(\Theta(a, \eps)) \| u \|, \quad a \to +\infty.
	\end{equation*}
	We proceed to estimate the first term in the right-hand side of \eqref{eq:Hhat.u1u2.quasiorth.est.1} as $a \to +\infty$
	\begin{equation*}
		\begin{aligned}
			\| \opB_{N,1} u \|^{\frac12} \| \opR_{N+1,2} u \|^{\frac12} &\le a^{-\frac12} t_a^{-1} \| \opB_{N,1} u \| + a^{\frac12} t_a \| \opR_{N+1,2} u \|\\
			&\ls a^{-\frac12} t_a^{-1} \| \widehat \opH_{a} u \| + a^{-\frac12} t_a^{-1} \Theta(a, \eps) \| u \| + a^{-\frac32} t_a \| u \|\\
			&\ls a^{-\frac12} t_a^{-1} \| \widehat \opH_{a} u \| + \Theta(a, \eps) \| u \|,
		\end{aligned}
	\end{equation*}
	using the fact that $a^{-\frac12} t_a = 2 V_2(t_a)^{-1} \to 0$ as $a \to +\infty$ in the last step. A similar estimate can be derived for $\| \opB_{N,2} u \|^{\frac12} \| \opR_{N+1,1} u \|^{\frac12}$. Applying both of them and $\| \opR_{N+1,1} u \|^{\frac12} \| \opR_{N+1,2} u \|^{\frac12} \ls a^{-2} \| u \|$ in \eqref{eq:Hhat.u1u2.quasiorth.est.1} yields the desired result.
\end{proof}

\begin{proof}[Proof of Theorem~\ref{thm:R}]
	Let $0 \ne u \in \SchwR \subset \Dom(\widehat \opH)$ and let us write $u = u_0 + u_1 + u_2$, where $u_k := \phi_{a,k} u$ with $k \in \{0,1,2\}$ and $\phi_{a,k}$ as in \eqref{eq:phiak}. Then we have
	\begin{equation*}
		\widehat \opH_a u_k = \phi_{a,k} \widehat \opH_a u + [\opwV, \phi_{a,k}] u, \quad k \in \{0,1,2\},
	\end{equation*}
	and therefore, noting that $\supp \phi_{a,1} \cap \supp \phi_{a,2} = \emptyset$ and using Lemma~\ref{lem:Vhat.phiak.commutator}, we obtain as $a \to +\infty$
	\begin{equation}
		\label{eq:hauk.norm.est}
		\begin{aligned}
			\| \widehat \opH_a u_0 \| &\le (1 + \BigO(a^{-\frac12} t_a^{-1})) \| \widehat \opH_a u \| + \BigO(\Theta(a, \eps)) \| u \|,\\
			\| \widehat \opH_a (u_1 + u_2) \| &\le (1 + \BigO(a^{-\frac12} t_a^{-1})) \| \widehat \opH_a u \| + \BigO(\Theta(a, \eps)) \| u \|,
		\end{aligned}
	\end{equation}
	with small $\eps > 0$ and $\Theta(a, \eps)$ as in \eqref{eq:Hahatphiak.commutator.coeff}.

	Firstly, note that $\supp u_1 \subset \Omega_{a,+}$ and $\supp u_2 \subset \Omega_{a,-}$ and therefore $u_1 \perp u_2$. Moreover, by Proposition~\ref{prop:local.R} and Lemma~\ref{lem:Hhat.u1u2.quasiorthogonal}, as $a \to +\infty$
	\begin{equation}
		\label{eq:uk.local.est1}
		\begin{aligned}
			V_2(t_a) \| u_1 + u_2 \| &\le \| \opA_{\beta}^{-1}\| ( 1 + \BigO ( \iota(t_a) + a^{-\frac12} t_a^{-1} ) ) (\| \widehat \opH_a u_1 \|^2 + \| \widehat \opH_a u_2 \|^2)^{\frac12}\\
			&\le \| \opA_{\beta}^{-1}\| ( 1 + \BigO ( \iota(t_a) + a^{-\frac12} t_a^{-1} ) ) \| \widehat \opH_a (u_1 + u_2) \|\\
			&\quad + \BigO (a^{-\frac12} t_a^{-1}) \| \widehat \opH_a u \| + \BigO (\Theta(a, \eps)) \| u \|.
		\end{aligned}
	\end{equation}
	Thus by \eqref{eq:hauk.norm.est}, \eqref{eq:ta.def} and Lemma~\ref{lem:reg.var.conv}, we have as $a \to +\infty$
	\begin{equation}
		\label{eq:uk.local.est}
		V_2(t_a) \| u_1 + u_2 \| \le \| \opA_{\beta}^{-1}\| ( 1 + \BigO ( \iota(t_a) + a^{-\frac12} t_a^{-1} ) )  \| \widehat \opH_a u \| + \BigO(\Theta(a, \eps)) \| u \|.
	\end{equation}

	Secondly, since $\supp u_0 \cap (\Omega'_{a,+} \cup \Omega'_{a,-}) = \emptyset$, then by Proposition~\ref{prop:away.R}
	\begin{equation*}
		a \| u_0 \| \ls \| \widehat{\opH}_a u_0 \|, \quad a \to +\infty,
	\end{equation*}
	and applying \eqref{eq:hauk.norm.est} and \eqref{eq:ta.def}, we have as $a \to +\infty$
	\begin{equation}
		\label{eq:uk.away.est}
		V_2(t_a) \| u_0 \| \ls a ^{-\frac12} t_a^{-1} \| \widehat{\opH}_a u_0 \| \ls a ^{-\frac12} t_a^{-1} \| \widehat{\opH}_a u \| + a ^{-\frac12} t_a^{-1} \Theta(a, \eps) \| u \|.
	\end{equation}

	Combining \eqref{eq:uk.local.est} and \eqref{eq:uk.away.est}, we find that as $a \to +\infty$
	\begin{align*}
		V_2(t_a) \| u \| &\le V_2(t_a) \left(\| u_0 \| + \| u_1 + u_2 \| \right)\\
		&\le \| \opA_{\beta}^{-1} \| ( 1 + \BigO ( \iota(t_a) + a^{-\frac12} t_a^{-1} ) )  \| \widehat \opH_a u \| + \BigO(\Theta(a, \eps)) \| u \|.
	\end{align*}
	Using \eqref{eq:Hahatphiak.commutator.decay}, we arrive at
	\begin{equation}
		\label{eq:Hb.est.R}
		\| u \| \le \| \opA_{\beta}^{-1} \| (V_2(t_a))^{-1} ( 1 + \BigO ( \iota(t_a) + (a^{\frac12} t_a)^{-1 + \eps} ) ) \| \widehat \opH_a u \|.
	\end{equation}
	Since $\SchwR$ is a core for $\opH$ and, equivalently, for $\widehat{\opH}$, we can extend the above estimate to any $u \in \Dom(\widehat{\opH})$ using a standard approximation argument. The proof of the theorem follows by an appeal to Proposition~\ref{prop:lbound.R} and the use of the inverse Fourier transform to take the result back to $x$-space.
\end{proof}


\section{Extensions and further remarks}
\label{sec:cor}
\subsection{The norm of the resolvent inside the numerical range}
\label{ssec:inside.Num}

A simple application of the triangle inequality allows us to obtain estimates for the resolvent norm in regions adjacent to the imaginary and real axes as well as to include further bounded perturbations. 

In detail, for an operator $H$ (as in Sections~\ref{sec:iR}, \ref{sec:R}), $\la, \mu \in \C$ and a bounded operator $W$, we get 
\begin{equation}\label{eq:triang}
\|(H-\la -\mu + W) u \| \geq \|(H-\la)u\| - |\mu| \|u\| - \|W \| \|u\|, \quad u \in \Dom(\opH). 	
\end{equation}

In particular, for $H$ as in Section \ref{sec:iR} with purely imaginary $V$ satisfying Assumption~\ref{asm:iR}, Theorem~\ref{thm:iR} and \eqref{eq:triang} with $\la = ib$, $\mu =a\geq 0$, $W=0$ yield 
\begin{equation}
	\label{eq:iR.criticalregion}
	\| (\opH - a - i b) u \| \geq  \left(\| \opA_{1,\frac\pi 2}^{-1} \|^{-1} ( V_2'(x_b) )^{\frac{2}{3}} ( 1 - \BigO ( \Upsilon(x_b) ) ) - a \right) \| u \|
\end{equation}
as $b \to +\infty$. Thus assuming that $V_2$ does not grow too slowly (\eg~$V_2'$ is bounded below by a strictly positive constant), that $b$ is large enough and that $\eps, \eps' > 0$ are sufficiently small, the region in the first quadrant of $\C$ (which contains the numerical range of the operator and its spectrum, if any) determined by
\begin{equation}
	\label{eq:res.bounded.region.iR}
	0 \leq a < \| \opA_{1,\frac\pi 2}^{-1} \|^{-1} ( V_2'(x_b) )^{\frac{2}{3}} (1 - \eps') - \eps, \quad b \to +\infty,
\end{equation}
with $V_2(x_b) = b$, is non-empty and unbounded. Moreover, the resolvent satisfies
$\| (\opH - a-i b)^{-1} \| \le 1/\eps$ inside this region.

For $H$ as in Section~\ref{sec:R}, one obtains from Theorem~\ref{thm:R} and \eqref{eq:triang} with $\la = a$, $\mu =ib$, $b > 0$, $W=0$ that
\begin{equation}
\| (\opH - a - i b) u \| \geq  \left( \| \opA_{\beta}^{-1} \|^{-1}  V_2(t_a) ( 1 - \BigO ( \iota(t_a) + (a^{\frac 12} t_a)^{-l_{\beta, \eps}} ) - b \right) \|u\|
\end{equation}
 as $a \to + \infty$. Thus the resolvent satisfies $\| (\opH - a-i b)^{-1} \| \le 1/\eps$ for
\begin{equation}\label{eq:res.bounded.region.R}
	0\leq b < \| \opA_{\beta}^{-1} \|^{-1} V_2(t_a)  (1 - \eps') - \eps, \quad a \to +\infty.
\end{equation}

In both cases, bounded perturbations $W$ can be included in an analogous way.

\subsubsection{The norm of the resolvent along curves adjacent to the imaginary axis}
\label{sssec:resnorm.adj.iR}
A more precise examination of the proof of Theorem~\ref{thm:iR} reveals that it is in fact possible to estimate $\| (\opH - \la)^{-1} \|$ along curves 
\begin{equation}\label{eq:lab.def}
\la_b := a(b) + ib, \quad b \to \infty,	
\end{equation}
even outside the region determined by \eqref{eq:res.bounded.region.iR}.
Let for simplicity $V = i V_2$ obey Assumption~\ref{asm:iR} and, with $\rho$ and $\Upsilon$ as defined in \eqref{eq:rho.def} and in Assumption~\ref{asm:iR}~\ref{itm:imvgrowth}, respectively, let $a: \Rplus \to \Rplus$ satisfy
\begin{equation}
	\label{eq:Phib.asm}
	\Phi_b := \langle \mu_b \rangle^2 \| (\opA_{1,\frac{\pi}{2}} - \mu_b)^{-1} \| \Upsilon(x_b) = o(1), \quad b \to +\infty,
\end{equation}
where (with $a_b \equiv a(b)$)
\begin{equation}
	\label{eq:mub.def}
	\mu_b := \rho^2 a_b = (V_2'(x_b))^{-\frac23} a_b.
\end{equation}
In our analysis, we shall be mainly concerned with two types of curves:
\begin{enumerate}[\upshape (1)]
	\item $\la_b$ with $a_b \ls V_2'(x_b)^\frac23$ for $b \to +\infty$, corresponding asymptotically to the critical region \eqref{eq:res.bounded.region.iR}, i.e. where $\mu_b$ satisfies
	\begin{equation}
		\label{eq:gral.cuves.crit.in}
		\mu_b \ls 1, \quad b \to +\infty;
	\end{equation}
	\item $\la_b$ with $V_2'(x_b)^\frac23 = o(a_b)$, $b \to +\infty$, and therefore $\la_b$ grows away from the critical region, i.e. where $\mu_b$ satisfies
	\begin{equation}
		\label{eq:gral.cuves.crit.out}
		\mu_b \to +\infty, \quad b \to +\infty.
	\end{equation}
\end{enumerate}
Note that, in the first case, we have $\Phi_b = \BigO(\Upsilon(x_b))$ due to the fact that $\|(\opA_{1,\pi/2} - z)^{-1}\|$ is bounded on compact sets in $\C$ and therefore, by Assumption~\ref{asm:iR}~\ref{itm:imvgrowth}, condition \eqref{eq:Phib.asm} holds automatically.

We further observe that, for any $z \in \C$, it can be shown that $\| (\opA_{1,\frac{\pi}{2}} - z)^{-1} \| = \| (\opA_{1,\frac{\pi}{2}} - \Re z)^{-1} \|$ (see \cite[Sec.~14.3.1]{Helffer-2013-book}) and that there exists a precise asymptotic estimate for $z \in \C_+$ (see \cite[Cor.~1.4]{BordeauxMontrieux-2013})
\begin{equation}
	\label{eq:airy.res.bm}
	\begin{aligned}
		\| (\opA_{1,\frac{\pi}{2}} - \Re z)^{-1} \| &= \sqrt{\frac{\pi}{2}} (\Re z)^{-\frac14} \exp\left(\frac43 (\Re z)^{\frac32} \right) \left(1 + \BigO((\Re z)^{-\frac32})\right)\\
		&\qquad + \BigO((\Re z)^{-\frac14}), \quad \Re z \to +\infty.
	\end{aligned}
\end{equation}
For any $\mu \ge 0$, applying standard arguments it is also possible to extend the graph-norm estimate \eqref{eq:airyinequality}
\begin{equation}
	\label{eq:airyinequality.ext}
	\| (\opA_{1,\frac{\pi}{2}} - \mu) u \|^2 + (1 + \mu^2) \| u \|^2 \gs  \| u'' \|^2 + \| x u \|^2 + \| u \|^2, \quad u \in \Dom(\opA_{1,\frac{\pi}{2}}),
\end{equation}
and to deduce from this (see e.g.~\eqref{eq:x.A}, \eqref{eq:Sb.gn})
\begin{equation}
	\label{eq:airyinequality.ext.cor}
	\begin{aligned}
		&\| \Dtp (\opA_{1,\frac{\pi}{2}} - \mu)^{-1} \| + \| (\opA_{1,\frac{\pi}{2}} - \mu)^{-1} \Dtp \| 
		\\ & \qquad \qquad + \| x (\opA_{1,\frac{\pi}{2}} - \mu)^{-1} \| + \| (\opA_{1,\frac{\pi}{2}} - \mu)^{-1} x \| \ls \langle \mu \rangle \| (\opA_{1,\frac{\pi}{2}} - \mu)^{-1} \|.
	\end{aligned}
\end{equation}
\begin{proposition}\label{prop:gral.curves}
Let $V = i V_2$ satisfy Assumption~\ref{asm:iR}, let $\opH$ be the Schr\"odinger operator \eqref{eq:H.def} in $\Lt(\Rplus)$, let $\la_b$ be as in \eqref{eq:lab.def} and let \eqref{eq:Phib.asm} hold with $\mu_b$ satisfying either \eqref{eq:gral.cuves.crit.in} or \eqref{eq:gral.cuves.crit.out}. Then 
\begin{equation}
	\label{eq:resnorm.iR.gralcurve}
	\|(\opH - \la_b)^{-1} \| = \| (\opA_{1,\frac{\pi}{2}} - \mu_b)^{-1} \| (V_2'(x_b))^{-\frac 23} (1 + \BigO(\Phi_b)), \quad b \to + \infty. 
\end{equation}
\end{proposition}
\begin{proof}[Sketch of proof]
We shall sketch the proof of this result by closely following the steps in Sub-section~\ref{ssec:proof.iR}, keeping the notation introduced there but omitting details whenever the arguments used earlier remain valid.	

\underline{Step 1}

Repeating the reasoning in Proposition~\ref{prop:away.iR} (replacing $\opH_b$ with $\opH_b - a_b = \opH - \la_b$), we find that for all $u \in \Dom(\opH)$ such that $\supp u \cap \Omega_b' = \emptyset$
\begin{equation}
	\label{eq:resnorm.away.iR.gralcurve}
	\delta \left( V_2'(x_b) \right)^{\frac{2}{3}} (\Upsilon(x_b))^{-1} \| u \| \lesssim \| (\opH_b - a_b) u \|, \quad b \to +\infty.
\end{equation}

\underline{Step 2}

With $\widetilde \opH_b$ and $\opS_b$ as in Proposition~\ref{prop:local.iR}, it is clear that (recall $\opS_\infty = \opA_{1,\pi/2}$)
\begin{equation}
	\label{eq:Hbhat.Sb.gralcurves}
	\rho^2 \opU_{b,\rho} (\widetilde \opH_b - a_b) \opU_{b,\rho}^{-1} = \opS_b - \mu_b = \opS_{\infty} - \mu_b + \opR_b.
\end{equation}

We shall prove next that $\mu_b \in \rho(\opS_b)$ as $b \to +\infty$. For any $\mu_b > 0$, the operator $\opK_{b,\infty} := \opI - \mu_b \opS_{\infty}^{-1} = \opS_{\infty}^{-1} (\opS_{\infty} - \mu_b) = (\opS_{\infty} - \mu_b) \opS_{\infty}^{-1}$ is bounded and invertible and moreover by \eqref{eq:airyinequality.ext.cor} we have
\begin{equation}
	\label{eq:Kbiminus1.est}
	\| \opK_{b,\infty}^{-1} \| \ls \langle \mu_b \rangle \| (\opS_{\infty} - \mu_b)^{-1} \|.
\end{equation}
Recalling from Proposition~\ref{prop:local.iR} that $0 \in \rho(\opS_b)$ for large enough $b$ and defining $\opK_b := \opI - \mu_b \opS_b^{-1} = \opS_b^{-1} (\opS_b - \mu_b) = (\opS_b - \mu_b) \opS_b^{-1}$, we find
\begin{equation*}
	\begin{aligned}
		\opK_b = \opK_{b,\infty} (\opI - \mu_b \opK_{b,\infty}^{-1} (\opS_b^{-1} - \opS_\infty^{-1})).
	\end{aligned}
\end{equation*}
Moreover, by \eqref{eq:Kbiminus1.est}, \eqref{eq:Sb.Sinf} and \eqref{eq:Phib.asm}, we have
\begin{equation*}
	\| \mu_b \opK_{b,\infty}^{-1} (\opS_b^{-1} - \opS_\infty^{-1}) \| \ls \Phi_b = o(1), \quad b \to +\infty.
\end{equation*}
It follows that $\opK_b$ is invertible and $\| \opK_b^{-1} \| \ls \| \opK_{b,\infty}^{-1} \|$ as $b \to +\infty$. Since $\opS_b - \mu_b = \opK_b \opS_b = \opS_b \opK_b$, we conclude that $\mu_b \in \rho(\opS_b)$ for $b \to +\infty$, as claimed. Moreover, $(\opS_b - \mu_b)^{-1} = \opS_b^{-1} \opK_b^{-1} = \opK_b^{-1} \opS_b^{-1}$ and, using \eqref{eq:normestimateSb}, \eqref{eq:Sb.gn} and \eqref{eq:Kbiminus1.est}, we deduce as $b \to +\infty$
\begin{equation}
	\label{eq:xSb.dxSb.normres.gralcurves}
	\| x (\opS_b - \mu_b)^{-1} \| + \| x (\opS_b^* - \mu_b)^{-1} \| + \| \Ntp (\opS_b - \mu_b)^{-1} \| \ls \langle \mu_b \rangle \| (\opS_\infty - \mu_b)^{-1} \|.
\end{equation}

Furthermore, we have
\begin{equation}
	\label{eq:gral.curves.step2.res.id}
	\begin{aligned}
		((\opS_b - \mu_b)^{-1} - (\opS_\infty - \mu_b)^{-1}) \opK_b &= \opS_b^{-1} - (\opS_\infty - \mu_b)^{-1} \opK_b\\
		&= \opS_b^{-1} - (\opS_\infty - \mu_b)^{-1} (\opK_{b,\infty} - \mu_b (\opS_b^{-1} - \opS_\infty^{-1}))\\
		&= \opS_b^{-1} - \opS_{\infty}^{-1} + \mu_b (\opS_{\infty} - \mu_b)^{-1} (\opS_b^{-1} - \opS_\infty^{-1})\\
		&= \opK_{b,\infty}^{-1} (\opS_b^{-1} - \opS_{\infty}^{-1}).
	\end{aligned}
\end{equation}
Hence
\begin{equation*}
	\begin{aligned}
		(\opS_b - \mu_b)^{-1} - (\opS_\infty - \mu_b)^{-1} &= \opK_{b,\infty}^{-1} (\opS_b^{-1} - \opS_{\infty}^{-1}) \opK_b^{-1}, \quad b \to +\infty,
	\end{aligned}
\end{equation*}
and therefore by \eqref{eq:Sb.Sinf} and \eqref{eq:Kbiminus1.est} and using the fact that $\mu_b$ satisfies \eqref{eq:gral.cuves.crit.in} or \eqref{eq:gral.cuves.crit.out}
\begin{equation*}
	\begin{aligned}
		\| (\opS_b - \mu_b)^{-1} - (\opS_\infty - \mu_b)^{-1} \| &\ls \langle\mu_b \rangle^2 \| (\opS_\infty - \mu_b)^{-1} \|^2 \Upsilon(x_b)
		\\
		&\ls \| (\opS_\infty - \mu_b)^{-1} \| \Phi_b, \quad b \to +\infty.
	\end{aligned}
\end{equation*}
It follows that
\begin{equation}
	\label{eq:Sb.Sinf.resnorms.gralcurves}
	\| (\opS_b - \mu_b)^{-1} \| = \| (\opS_\infty - \mu_b)^{-1} \| (1 + \BigO(\Phi_b)), \quad b \to +\infty,
\end{equation}
and hence from \eqref{eq:Hbhat.Sb.gralcurves} as $b \to +\infty$
\begin{equation*}
	\rho^{-2} \| (\widetilde \opH_b - a_b)^{-1} \| = \| (\opS_b - \mu_b)^{-1} \| = \| (\opS_\infty - \mu_b)^{-1} \| (1 + \BigO(\Phi_b)).
\end{equation*}
Arguing as in the last stage of Proposition~\ref{prop:local.iR}, this yields as $b \to +\infty$
\begin{equation}
	\label{eq:resnorm.local.iR.gralcurve}
	\begin{aligned}
		&\| (\opS_\infty - \mu_b)^{-1} \|^{-1} (V_2'(x_b))^{\frac{2}{3}} (1 - \BigO(\Phi_b))\\
		& \qquad  \qquad \le \inf \left\{ \frac{\left\| (\opH_b - a_b) u \right\|}{\|u\|}: \; 0 \neq u \in \Dom(\opH), \; \supp u \subset \Omega_b \right\}.
	\end{aligned}	
\end{equation}

\underline{Step 3}

We follow the proof of Proposition~\ref{prop:lbound.iR}, replacing $\opS_b$ with $\opS_b - \mu_b$, to find $g_b \in \Dom((\opS_b^* - \mu_b) (\opS_b - \mu_b))$ such that
\begin{equation*}
	\| (\opS_b - \mu_b) g_b \| = \varsigma_b^{-1} = \| (\opS_b - \mu_b)^{-1} \|^{-1}, \quad b \to +\infty.
\end{equation*}
Moreover, with $\varsigma_{b,\infty} := \| (\opS_\infty - \mu_b)^{-1} \|$, we have (see \eqref{eq:Sb.Sinf.resnorms.gralcurves})
\begin{equation}
	\label{eq:lab.conv.gralcurve}
	\varsigma_b = \varsigma_{b,\infty} (1 + \BigO(\Phi_b)), \quad b \to +\infty.
\end{equation}

Recalling the cut-off functions $\psi_b$, we write
\begin{equation*}
	(S_b - \mu_b) \psi_b g_b = (S_b - \mu_b) g_b + (\psi_b-1) (S_b - \mu_b) g_b + [S_b, \psi_b] g_b
\end{equation*}
and, applying \eqref{eq:xSb.dxSb.normres.gralcurves} and \eqref{eq:lab.conv.gralcurve} (refer also to \eqref{eq:psib'} and \eqref{eq:db.rho}), we deduce
\begin{align*}
	\| (\psi_b-1) (S_b - \mu_b) g_b \| &\ls \|(\psi_b-1) x^{-1}\|_\infty \|x (S_b^* - \mu_b)^{-1}\| \|(S_b^* - \mu_b) (S_b - \mu_b) g_b \|\\
	&\ls \Upsilon(x_b) \langle\mu_b \rangle \varsigma_{b,\infty}^{-1} \ls \Phi_b \langle \mu_b \rangle^{-1} \varsigma_{b,\infty}^{-2},
	\\
	\|[S_b, \psi_b] g_b \| &\ls  \| \psi_b' \|_{\infty} \|\partial_x (S_b - \mu_b)^{-1} (S_b - \mu_b) g_b \| + \| \psi_b'' \|_{\infty} \|g_b\|\\
	&\ls \Upsilon(x_b) \langle \mu_b \rangle + \Upsilon(x_b)^2 \ls \Upsilon(x_b) \langle \mu_b \rangle \ls \Phi_b \langle \mu_b \rangle^{-1} \varsigma_{b,\infty}^{-1},
\end{align*}
as $b \to +\infty$. Hence, noting that $\varsigma_{b,\infty}$ is bounded below by a positive constant when $\mu_b \in \Rplus$, we have
\begin{equation*}
	\| (S_b - \mu_b) \psi_b g_b \| = \varsigma_b^{-1} + \BigO(\varsigma_{b,\infty}^{-1} \Phi_b), \quad b \to +\infty.
\end{equation*}
Similarly $\| \psi_b g_b \| = 1 + \BigO(\varsigma_{b,\infty}^{-1} \Phi_b)$ as $ b \to +\infty$ and consequently
\begin{equation*}
	\left|	\frac{\| (S_b - \mu_b) \psi_b g_b \|}{\| \psi_b g_b \|} - \frac1 {\varsigma_{b,\infty}} \right| = \BigO(\varsigma_{b,\infty}^{-1} \Phi_b), \quad b \to +\infty.
\end{equation*}

As before, we set $u_b := \opU_{b,\rho}^{-1} \psi_b g_b \in \Dom(\opH)$. Then $\supp u_b \subset \Omega_b$ and
\begin{equation}
	\label{eq:resnorm.lbound.iR.gralcurve}
	\frac{\| (\opH_b - a_b) u_b\|}{\| u_b \|} = \| (\opS_\infty - \mu_b)^{-1} \|^{-1} (V_2'(x_b))^{\frac23} (1 + \BigO(\Phi_b)), \quad b \to +\infty.
\end{equation}

\underline{Step 4}

Repeating the commutator calculations in the proof of Lemma~\ref{lem:Hb.phibk.commutator} for $\opH_b - a_b$, we find for all $u \in \Dom(\opH)$ and $k \in \{0, 1\}$
\begin{equation*}
	\Re \langle (\opH_b - a_b) u, \phi_{b,k}'^2 u \rangle = 2 \Re \langle \phi_{b,k}' u', \phi_{b,k}'' u \rangle + \| \phi_{b,k}' u' \|^2 - a_b \| \phi_{b,k}' u \|^2
\end{equation*}
which we use to estimate (with small $\eps > 0$ and $b \to +\infty$)
\begin{equation*}
	\begin{aligned}
		\| \phi_{b,k}' u' \| &\ls \| (\opH_b - a_b) u \|^{\frac12} \| \phi_{b,k}'^2 u \|^{\frac12} + \| \phi_{b,k}' u' \|^{\frac12} \| \phi_{b,k}'' u \|^{\frac12} + a_b^{\frac12} \| \phi_{b,k}' u \|\\
		&\ls \Upsilon(x_b) \| (\opH_b - a_b) u \| + x_b^{2 \nu} (\Upsilon(x_b))^{-1} \| u \| + \eps \| \phi_{b,k}' u' \| + \eps^{-1} x_b^{2 \nu} \| u \|\\
		&\quad + a_b^{\frac12} x_b^{\nu} \| u \| \implies\\
		\| \phi_{b,k}' u' \| &\ls \Upsilon(x_b) \| (\opH_b - a_b) u \| + (x_b^{2 \nu} (\Upsilon(x_b))^{-1} + x_b^{\nu} a_b^{\frac12}) \| u \|.
	\end{aligned}
\end{equation*}
Hence
\begin{equation*}
	\| [\opH_b - a_b, \phi_{b,k}] u \| \ls \Upsilon(x_b) \| (\opH_b - a_b) u \| + (x_b^{2 \nu} (\Upsilon(x_b))^{-1} + x_b^{\nu} a_b^{\frac12}) \| u \|, \quad b \to +\infty.
\end{equation*}
Then for any $u \in \Dom(\opH)$, $u = u_0 + u_1$, we have for $k \in \{0, 1\}$
\begin{equation*}
	(\opH_b - a_b) u_k = \phi_{b,k} (\opH_b - a_b) u + [\opH_b - a_b, \phi_{b,k}] u,
\end{equation*}
and therefore as $b \to +\infty$
\begin{equation*}
	\| (\opH_b - a_b) u_k \| \leq ( 1 + \BigO \left( \Upsilon(x_b) \right) ) \| (\opH_b - a_b) u \| + \BigO ( x_b^{2 \nu} (\Upsilon(x_b))^{-1} + x_b^{\nu} a_b^{\frac12} ) \|u\|.
\end{equation*}	

As in the proof of Theorem~\ref{thm:iR}, we separately consider $u_1$, $\supp u_1 \subset \Omega_b$, and $u_0$, $\supp u_0 \cap \Omega_b' = \emptyset$, respectively applying \eqref{eq:resnorm.local.iR.gralcurve} and \eqref{eq:resnorm.away.iR.gralcurve}
\begin{align*}
	\| u_1 \| &\le \| (\opS_\infty - \mu_b)^{-1} \| (V_2'(x_b))^{-\frac{2}{3}} (1 + \BigO(\Phi_b) \| (\opH_b - a_b) u_1 \|\\
	&\le \| (\opS_\infty - \mu_b)^{-1} \| (V_2'(x_b))^{-\frac{2}{3}} (1 + \BigO(\Phi_b)) \| (\opH_b - a_b) u \| + \BigO(\Phi_b) \| u \|,\\
	\| u_0 \| &\ls (V_2'(x_b))^{-\frac{2}{3}} \Upsilon(x_b) \| (\opH_b - a_b) u_0 \|\\
	&\ls (V_2'(x_b))^{-\frac{2}{3}} \Upsilon(x_b) \| (\opH_b - a_b) u \| + \Upsilon(x_b)^2 (1 + \mu_b^{\frac12}) \| u \|,
\end{align*}
as $b \to +\infty$. Combining these estimates, we get as $b \to +\infty$ 
\begin{equation*}
	\begin{aligned}
		\| u \| &\le \| u_0 \| + \| u_1 \| 
		\\ & \le \| (\opS_\infty - \mu_b)^{-1} \| (V_2'(x_b))^{-\frac{2}{3}} (1 + \BigO(\Phi_b)) \| (\opH_b - a_b) u \| + \BigO(\Phi_b) \| u \|,
		\end{aligned}
\end{equation*}
and hence
\begin{equation*}
	\| u \| \le \| (\opS_\infty - \mu_b)^{-1} \| (V_2'(x_b))^{-\frac{2}{3}} (1 + \BigO(\Phi_b) \| (\opH_b - a_b) u \|, \quad b \to +\infty.
\end{equation*}

This last inequality and \eqref{eq:resnorm.lbound.iR.gralcurve} yield \eqref{eq:resnorm.iR.gralcurve}.
\end{proof}

We conclude this subsection with a general construction for the level curves of the resolvent (some examples will be shown in Section~\ref{sec:examples}). Letting $\zeta_b := \mu_b^{\frac74} \| (\opA_{1,\frac{\pi}{2}} - \mu_b)^{-1} \|$, we note (see \eqref{eq:airy.res.bm}) that $\zeta_b \to +\infty$ as $\mu_b \to +\infty$, i.e. when $\la_b$ lies outside the region \eqref{eq:res.bounded.region.iR}. Applying \eqref{eq:airy.res.bm} again, we find
\begin{equation*}
	\frac43 \mu_b^{\frac32} \exp\left(\frac43 \mu_b^{\frac32}\right) = \frac43 \sqrt{\frac2{\pi}} \zeta_b (1 + o(1)), \quad b \to +\infty.
\end{equation*}
The above equation can be rewritten as
\begin{equation*}
	\frac43 \mu_b^{\frac32} = W_0\left(\frac43 \sqrt{\frac2{\pi}} \zeta_b (1 + o(1))\right), \quad b \to +\infty,
\end{equation*}
where $W_0(x)$ is the Lambert function that solves $y \exp(y) = x$ for $x \ge 0$. Although $W_0(x)$ cannot be written in terms of elementary functions, the following bounds have been found for $x \in [e, \infty)$ (see \cite[Thm.~2.7]{Hoorfar-2008-9})
\begin{equation*}
	\log x - \log\log x + \frac12 \frac{\log\log x}{\log x} \le W_0(x) \le \log x - \log\log x + \frac{e}{e-1} \frac{\log\log x}{\log x},
\end{equation*}
and thus we deduce
\begin{equation}
	\label{eq:gralcurve.est.iR}
	\mu_b = \left(\frac34\right)^{\frac23} (\log(\| (\opA_{1,\frac{\pi}{2}} - \mu_b)^{-1} \|))^{\frac23} (1 + o(1)), \quad b \to +\infty.
\end{equation}

From \eqref{eq:resnorm.iR.gralcurve}, we have $\| (\opA_{1,\frac{\pi}{2}} - \mu_b)^{-1} \| = \rho^{-2} \| (\opH - \la_b)^{-1} \| (1 + o(1))$ and hence
\begin{equation*}
	\mu_b = \left(\frac34\right)^{\frac23} (\log(\rho^{-2} \| (\opH - \la_b)^{-1} \|))^{\frac23} (1 + o(1)), \quad b \to +\infty.
\end{equation*}
Substituting $\| (\opH - \la_b)^{-1} \| = \eps^{-1}$, with $\eps > 0$, we obtain
\begin{equation}
	\label{eq:resnorm.iR.levelcurves}
	a_b = \left(\frac34\right)^{\frac23} \rho^{-2} (\log(\rho^{-2} \eps^{-1}))^{\frac23} (1 + o(1)), \quad b \to +\infty.
\end{equation}
We remark that as expected formula \eqref{eq:resnorm.iR.levelcurves} indicates that the level curves of a sub-linear potential (where $\rho^{-2} \to 0$ as $b \to +\infty$) will cross the imaginary axis into $\C_-$. We also note that, when $V_2(x) = x^2$ (\ie~$\opH$ is the Davies operator), then $x_b = b^{\frac12}$ and the above equation becomes
\begin{equation*}
	a_b = \left(\frac32\right)^{\frac23} b^{\frac13} \left(\log(b^{\frac13} \eps^{-1})\right)^{\frac23} (1 + o(1)), \quad a \to +\infty,
\end{equation*}
(compare these curves with \eqref{eq:resnorm.iR.levelcurves.powerlike} for $n = 2$ and with the known formulas derived in \cite[Prop.~4.6]{BordeauxMontrieux-2013}).

\subsubsection{The norm of the resolvent along curves adjacent to the real axis}
\label{sssec:resnorm.adj.R}
We can similarly estimate $\| (\opH - \la)^{-1} \|$, for $\opH$ as in \eqref{eq:H.real.def} and potential $V := i V_2$ satisfying Assumption~\ref{asm:R}, along general curves adjacent to the real axis
\begin{equation}
	\label{eq:laa.def}
	\la_a :=  a + i b(a), \quad a \to +\infty,
\end{equation}
with $b : \Rplus \to \Rplus$ satisfying
\begin{gather}
	\label{eq:gral.curves.real.adj.asm} b_a a^{-1} = o(1), \quad a \to +\infty,\\
	\label{eq:gral.curves.real.Phia.asm} \Phi_a := \langle \mu_a \rangle^2 \| (\opA_{\beta} - \mu_a)^{-1} \| (\iota(t_a) + (a^{\frac12} t_a)^{-l_{\beta, \eps}}) = o(1), \quad a \to +\infty,
\end{gather}
where
\begin{equation*}
	\mu_a := b_a (V_2(t_a))^{-1},
\end{equation*}
$b_a \equiv b(a)$, and $\opA_{\beta}$, $\iota$ and $l_{\beta, \eps}$ are as defined in \eqref{eq:A.beta.def}, \eqref{eq:iota.def} and \eqref{eq:lbeta.def}, respectively. We are interested in two types of curves:
\begin{enumerate}[\upshape (1)]
	\item $\la_a$ with $b_a \ls V_2(t_a)$ for $a \to +\infty$, corresponding asymptotically to the critical region \eqref{eq:res.bounded.region.R}, \ie~where $\mu_a$ satisfies
	\begin{equation}
		\label{eq:gral.cuves.real.crit.in}
		\mu_a \ls 1, \quad a \to +\infty;
	\end{equation}
	\item $\la_a$ with $V_2(t_a) = o(b_a)$, $a \to +\infty$, and therefore $\la_a$ grows away from the critical region, \ie~where $\mu_a$ satisfies
	\begin{equation}
		\label{eq:gral.cuves.real.crit.out}
		\mu_a \to +\infty, \quad a \to +\infty.
	\end{equation}
\end{enumerate}
Note that, in the first type above, we have $\Phi_a = \BigO(\iota(t_a) + (a^{\frac12} t_a)^{-l_{\beta, \eps}})$ due to the fact that $\|(\opA_{\beta} - z)^{-1}\|$ is bounded on compact sets in $\C$ and therefore, by Lemma~\ref{lem:reg.var.conv} and the fact that $t_a \to +\infty$ as $a \to +\infty$, condition \eqref{eq:gral.curves.real.Phia.asm} holds automatically.

In \cite[Ex.~4.3]{ArSi-generalised-2022}, the following asymptotic estimate was found
\begin{equation}
	\label{eq:Abeta.resnorm.ex.poly}
	\| (\opA_{\beta} - \mu)^{-1} \| = \sqrt{\frac{\pi}{\beta}} \mu^{\frac{1 - \beta}{2 \beta}} \exp\left(\frac{2 \beta}{\beta + 1} \mu^{\frac{1 + \beta}{\beta}}\right) (1 + o(1)), \quad \mu \to +\infty.
\end{equation}

Before formulating our result, we also note that, for any $\mu \ge 0$, it is possible to extend the graph-norm estimates \eqref{eq:genAiry.graph.norm} applying standard arguments to
\begin{equation}
	\label{eq:genAiry.graph.norm.ext}
	\| (\opA_{\beta} - \mu) u \|^2 + (1 + \mu^2) \| u \|^2 \gs  \| u' \|^2 + \| |x|^{\beta} u \|^2 + \| u \|^2, \quad u \in \Dom(\opA_{\beta}),
\end{equation}
and it follows (see \eg~\eqref{eq:x.A}, \eqref{eq:Sb.gn}, \eqref{eq:Sthatminus1.realline.ubound1})
\begin{equation}
	\label{eq:genAiry.ineq.ext.cor}
	\begin{aligned}
		&\| \Ntp (\opA_{\beta} - \mu)^{-1} \| + \| (\opA_{\beta} - \mu)^{-1} \Ntp \| 
		\\ & \qquad \qquad + \| |x|^{\beta} (\opA_{\beta} - \mu)^{-1} \| + \| (\opA_{\beta} - \mu)^{-1} |x|^{\beta} \| \ls \langle \mu \rangle \| (\opA_{\beta} - \mu)^{-1} \|.
	\end{aligned}
\end{equation}
\begin{proposition}
	\label{prop:gral.curves.real}
	Let $V = i V_2$ satisfy Assumption~\ref{asm:R}, let $\opH$ be the Schr\"odinger operator \eqref{eq:H.real.def} in $\Lt(\R)$, let $\la_a$ be as in \eqref{eq:laa.def} and let \eqref{eq:gral.curves.real.adj.asm}-\eqref{eq:gral.curves.real.Phia.asm} hold with $\mu_a$ satisfying either \eqref{eq:gral.cuves.real.crit.in} or \eqref{eq:gral.cuves.real.crit.out}. Then
	\begin{equation}
		\label{eq:resnorm.R.gralcurve}
		\|(\opH - \la_a)^{-1} \| = \| (\opA_{\beta} - \mu_a)^{-1} \| (V_2(t_a))^{-1} (1 + \BigO(\Phi_a)), \quad a \to + \infty. 
	\end{equation}
\end{proposition}
\begin{proof}[Sketch of proof]
	We shall sketch the proof of this result by closely following the steps in Sub-section~\ref{ssec:proof.R}, keeping the notation introduced there but omitting details whenever the arguments used earlier remain valid. We introduce the operators (refer also to \eqref{eq:Hhat.P.def} and \eqref{eq:Ha.def})
	\begin{equation}
		\label{eq:Hlahat.P.def}
		\opH_{\la_a} := \opH - \la_a = \opH_{a} - i b_a,\quad \widehat\opH_{\la_a} :=  -i \, \sF \, \opH_{\la_a} \, \sF^{-1} = \widehat\opH_{a} - b_a.
	\end{equation}

	\underline{Step 1}
	
	Repeating the reasoning in Proposition~\ref{prop:away.R} (replacing $\widehat\opH_a$ with $\widehat\opH_{\la_a}$) and applying \eqref{eq:Hahat.away.est.R}, we find that for all $u \in \Dom(\widehat\opH)$ such that $\supp u \cap (\Omega_{a,+}' \cup \Omega_{a,-}') = \emptyset$ (with $\eps > 0$ arbitrarily small and some $C_{\eps} > 0$)
	\begin{equation}
		\label{eq:Hlaahat.away.est.R}
		\begin{aligned}
			\| \widehat\opH_{\la_a} u \|^2 &= \| \widehat\opH_a u \|^2 + b_a^2 \| u \|^2 - 2 b_a \Re \langle \widehat V u, u \rangle\\
			&\ge \| \widehat\opH_a u \|^2 + b_a^2 \| u \|^2 - \eps \| \widehat V u \|^2 - C_{\eps} b_a^2 \| u \|^2\\
			&\ge (C_{\delta} - \eps) (\| \xi^2 u \|^2 + \| \widehat V u \|^2) + C_{\delta} a^2 \left(1 - \frac{C_{\eps}}{C_{\delta}} \left(\frac{b_a}{a}\right)^2\right) \| u \|^2\\
			&\gs \| \xi^2 u \|^2 + \| \widehat V u \|^2 + a^2 \| u \|^2, \quad a \to +\infty,
		\end{aligned}
	\end{equation}
	using assumption \eqref{eq:gral.curves.real.adj.asm} in the last step. This proves that \eqref{eq:away.est} continues to hold when we replace $\widehat\opH_a$ with $\widehat\opH_{\la_a}$.
	
	\underline{Step 2}
	
	We retain the notation in Proposition~\ref{prop:local.R}. From \eqref{eq:Ha.til.def}, \eqref{eq:Wahat.def} and \eqref{eq:whSt.def}, we have
	\begin{equation}
		\label{eq:Hlaahat.Sahat.gralcurves}
		\frac{1}{V_2(t_a)} \opU_{a,t_a} \widetilde\opH_{\la_a} \opU_{a,t_a}^{-1} = \widehat\opS_a^0 - \mu_a + \widehat\opR_a = \widehat\opS_a - \mu_a.
	\end{equation}
	Our next aim is to prove that $\mu_a \in \rho(\widehat\opS_a)$ as $a \to +\infty$. To do this, we argue as in Step 2 of Proposition~\ref{prop:gral.curves}. For any $\mu_a > 0$, the operator $\widehat\opK_{a,\infty} := \opI - \mu_a \widehat\opS_{\infty}^{-1} = \widehat\opS_{\infty}^{-1} (\widehat\opS_{\infty} - \mu_a) = (\widehat\opS_{\infty} - \mu_a) \widehat\opS_{\infty}^{-1}$ is bounded and invertible and moreover by \eqref{eq:genAiry.ineq.ext.cor} we have
	\begin{equation}
		\label{eq:Kahatiminus1.est}
		\| \widehat\opK_{a,\infty}^{-1} \| \ls \langle \mu_a \rangle \| (\opA_{\beta} - \mu_a)^{-1} \|.
	\end{equation}
	Recalling from Proposition~\ref{prop:local.R} that $0 \in \rho(\widehat\opS_a)$ for large enough $a$ and defining $\widehat\opK_a := \opI - \mu_a \widehat\opS_a^{-1} = \widehat\opS_a^{-1} (\widehat\opS_a - \mu_a) = (\widehat\opS_a - \mu_a) \widehat\opS_a^{-1}$, we find
	\begin{equation*}
		\widehat\opK_a = \widehat\opK_{a,\infty} (\opI - \mu_a \widehat\opK_{a,\infty}^{-1} (\widehat\opS_a^{-1} - \widehat\opS_\infty^{-1})).
	\end{equation*}
	Moreover, by \eqref{eq:hSt.conv}, \eqref{eq:Kahatiminus1.est} and \eqref{eq:gral.curves.real.Phia.asm}, we have
	\begin{equation*}
		\| \mu_a \widehat\opK_{a,\infty}^{-1} (\widehat\opS_a^{-1} - \widehat\opS_\infty^{-1}) \| \ls \Phi_a = o(1), \quad a \to +\infty.
	\end{equation*}
	It follows that $\widehat\opK_a$ is invertible and $\| \widehat\opK_a^{-1} \| \ls \| \widehat\opK_{a,\infty}^{-1} \|$ as $a \to +\infty$. Since $\widehat\opS_a - \mu_a = \widehat\opK_a \widehat\opS_a = \widehat\opS_a \widehat\opK_a$, we conclude that $\mu_a \in \rho(\widehat\opS_a)$ for $a \to +\infty$, as claimed. Moreover, $(\widehat\opS_a - \mu_a)^{-1} = \widehat\opS_a^{-1} \widehat\opK_a^{-1} = \widehat\opK_a^{-1} \widehat\opS_a^{-1}$ and, using \eqref{eq:Sahatminus1.realline.ubound} and \eqref{eq:Kahatiminus1.est}, we deduce as $a \to +\infty$
	\begin{equation}
		\label{eq:xiSahat.normres.gralcurves}
		\| \xi (\widehat\opS_a - \mu_a)^{-1} \| + \| \xi (\widehat\opS_a^* - \mu_a)^{-1} \| \ls \langle \mu_a \rangle \| (\opA_{\beta} - \mu_a)^{-1} \|.
	\end{equation}

	Furthermore, we have (see the argument in \eqref{eq:gral.curves.step2.res.id})
	\begin{equation*}
		((\widehat\opS_a - \mu_a)^{-1} - (\widehat\opS_\infty - \mu_a)^{-1}) \widehat\opK_a = \widehat\opK_{a,\infty}^{-1} (\widehat\opS_a^{-1} - \widehat\opS_{\infty}^{-1}).
	\end{equation*}
	Hence
	\begin{equation*}
		(\widehat\opS_a - \mu_a)^{-1} - (\widehat\opS_\infty - \mu_a)^{-1} = \widehat\opK_{a,\infty}^{-1} (\widehat\opS_a^{-1} - \widehat\opS_{\infty}^{-1}) \widehat\opK_a^{-1}, \quad a \to +\infty,
	\end{equation*}
	and therefore by \eqref{eq:hSt.conv} and \eqref{eq:Kahatiminus1.est} and using the fact that $\mu_a$ satisfies \eqref{eq:gral.cuves.real.crit.in} or \eqref{eq:gral.cuves.real.crit.out}
	\begin{equation*}
		\| (\widehat\opS_a - \mu_a)^{-1} - (\widehat\opS_\infty - \mu_a)^{-1} \| \ls \| (\opA_{\beta} - \mu_a)^{-1} \| \Phi_a, \quad a \to +\infty.
	\end{equation*}
	It follows that
	\begin{equation}
		\label{eq:Sahat.Sinfhat.resnorms.gralcurves}
		\| (\widehat\opS_a - \mu_a)^{-1} \| = \| (\opA_{\beta} - \mu_a)^{-1} \| (1 + \BigO(\Phi_a)), \quad a \to +\infty,
	\end{equation}
	and hence from \eqref{eq:Hlaahat.Sahat.gralcurves} and \eqref{eq:Sahat.Sinfhat.resnorms.gralcurves} as $a \to +\infty$
	\begin{equation*}
		V_2(t_a) \| \widetilde\opH_{\la_a}^{-1} \| = \| (\widehat\opS_a - \mu_a)^{-1} \| = \| (\opA_{\beta} - \mu_a)^{-1} \| (1 + \BigO(\Phi_a)).
	\end{equation*}
	Arguing as in the last stage of Proposition~\ref{prop:local.R}, this yields as $a \to +\infty$
	\begin{equation}
		\label{eq:resnorm.local.R.gralcurve}
		\begin{aligned}
			&\| (\opA_{\beta} - \mu_a)^{-1} \|^{-1} V_2(t_a) (1 - \BigO(\Phi_a))\\
			& \qquad  \qquad \le \inf \left\{ \frac{\left\| \widehat\opH_{\la_a} u \right\|}{\|u\|}: \; 0 \neq u \in \Dom(\widehat\opH), \; \supp u \subset \Omega_{a,\pm} \right\}.
		\end{aligned}	
	\end{equation}

	\underline{Step 3}
	
	We follow the proof of Proposition~\ref{prop:lbound.R}, replacing $\widehat\opS_a$ with $\widehat\opS_a - \mu_a$, to find $g_a \in \Dom((\widehat\opS_a^* - \mu_a) (\widehat\opS_a - \mu_a))$ such that
	\begin{equation}
		\label{eq:varsiga.def}
		\| (\widehat\opS_a - \mu_a) g_a \| = \varsigma_a^{-1} = \| (\widehat\opS_a - \mu_a)^{-1} \|^{-1}, \quad a \to +\infty.
	\end{equation}
	Moreover, with $\varsigma_{a,\infty} := \| (\opA_{\beta} - \mu_a)^{-1} \|$, we have (see \eqref{eq:Sahat.Sinfhat.resnorms.gralcurves})
	\begin{equation}
		\label{eq:laa.conv.gralcurve}
		\varsigma_a = \varsigma_{a,\infty} (1 + \BigO(\Phi_a)), \quad a \to +\infty.
	\end{equation}

	Recalling the cut-off functions $\psi_a$, we write
	\begin{equation}
		\label{eq:Sahat.mua.psia.ga}
		\begin{aligned}
			(\widehat\opS_{a} - \mu_a) \psi_a g_a &= (\widehat\opS_{a} - \mu_a) g_a + (\psi_a - 1) (\widehat\opS_{a} - \mu_a) g_a\\
			&\quad + [\sF \, \tilde \phi \, W_{a} \sF^{-1}, \psi_a] g_a + \sF \, \phi \, W_{a} \sF^{-1} (\psi_a - 1) g_a\\
			&\quad - (\psi_a - 1) \sF \, \phi \, W_{a} \sF^{-1} g_a
		\end{aligned}
	\end{equation}
	and we proceed to estimate the terms in the right-hand side, except the first one, as we did in Proposition~\ref{prop:lbound.R}. Applying \eqref{eq:xiSahat.normres.gralcurves}, \eqref{eq:varsiga.def} and \eqref{eq:laa.conv.gralcurve}, we have as $a \to +\infty$
	\begin{equation}
		\label{eq:Sahat.mua.psia.ga.est.1}
		\begin{aligned}
			\| (\psi_a - 1) (\widehat\opS_a - \mu_a) g_a \| &\ls \| (\psi_a - 1) \xi^{-1} \|_\infty \| \xi (\widehat\opS_a^* - \mu_a)^{-1} \| \| (\widehat\opS_a^* - \mu_a) (\widehat\opS_a - \mu_a) g_a \|\\
			&\ls (\xi_a t_a)^{-1} \langle \mu_a \rangle \| (\opA_{\beta} - \mu_a)^{-1} \| \varsigma_{a}^{-2} \ls (\xi_a t_a)^{-1} \langle \mu_a \rangle \varsigma_{a,\infty}^{-1}.
		\end{aligned}
	\end{equation}
	From the proof of Proposition~\ref{prop:lbound.R}, the fact that $\widehat\opS_a = \widehat\opK_a^{-1} (\widehat\opS_a - \mu_a)$ for large enough $a > 0$ (refer to Step 2 above, noting also \eqref{eq:Kahatiminus1.est}), \eqref{eq:varsiga.def} and \eqref{eq:laa.conv.gralcurve}, we find as $a \to +\infty$
	\begin{equation}
		\label{eq:Sahat.mua.psia.ga.est.2}
		\begin{aligned}
			\|[\sF \, \tilde \phi \, W_{a} \sF^{-1}, \psi_a] g_a\| &\ls (\xi_a t_a)^{-1} (\| \widehat\opS_a g_a \| + \|g_a\|)\\
			&\ls (\xi_a t_a)^{-1} (\| \widehat\opK_a^{-1} \| \| (\widehat\opS_a - \mu_a) g_a \| + 1)\\
			&\ls (\xi_a t_a)^{-1} (\langle \mu_a \rangle \| (\opA_{\beta} - \mu_a)^{-1} \| \varsigma_a^{-1} + 1)\\
			&\ls (\xi_a t_a)^{-1} \langle \mu_a \rangle.
		\end{aligned}
	\end{equation}
	Using again the proof of Proposition~\ref{prop:lbound.R}, estimate \eqref{eq:Sahatminus1.realline.ubound} and the above arguments, we have as $a \to +\infty$
	\begin{equation}
		\label{eq:Sahat.mua.psia.ga.est.3}
		\begin{aligned}
			\|\sF \, \phi \, W_{a} \sF^{-1} (\psi_a - 1) g_a\| & \ls \| (\psi_a - 1) \xi^{-1} \|_\infty \| \xi \widehat\opS_a^{-1} \| \| \widehat\opS_a g_a \|\\
			&\ls (\xi_a t_a)^{-1} \| \widehat\opK_a^{-1} \| \| (\widehat\opS_a - \mu_a) g_a \| \ls (\xi_a t_a)^{-1} \langle \mu_a \rangle.
		\end{aligned}
	\end{equation}
	To estimate the last term in the right-hand side of \eqref{eq:Sahat.mua.psia.ga}, we adapt the corresponding section of the proof of Proposition~\ref{prop:lbound.R}.
	\begin{equation}
		\label{eq:Sahat.mua.psia.ga.est.4}
		\|(\psi_a-1) \sF \, \phi \, W_{a} \sF^{-1} g_a\| \ls \iota(t_a) + \|(\psi_a-1) \sF \, \phi \, \omega_\beta \sF^{-1} g_a\|.
	\end{equation}
	For $\beta > 1/2$, using \eqref{eq:wSt0.sep} and arguing as in \eqref{eq:Sahat.mua.psia.ga.est.2}, we have as $a \to +\infty$
	\begin{equation*}
		\begin{aligned}
			\|(\psi_a-1) \sF \, \phi \, \omega_{\beta} \sF^{-1} g_a\| &\ls \|(\psi_a-1) \xi^{-1}\|_\infty \| \xi \sF \, \phi \, \omega_\beta \sF^{-1} g_a\|\\
			&\ls (\xi_a t_a)^{-1} \| (\phi \, \omega_\beta \check g_a)' \| \ls (\xi_a t_a)^{-1} (\| \check g_a \|_{\infty} + \| \check g_a' \|)\\
			&\ls (\xi_a t_a)^{-1} \| \langle \xi_a \rangle g_a \| \ls (\xi_a t_a)^{-1} (\| \widehat\opS_a g_a \| + \| g_a \|)\\
			&\ls (\xi_a t_a)^{-1} \langle \mu_a \rangle.
		\end{aligned}
	\end{equation*}
	For $\beta \in (0, 1/2]$, using $(\phi \omega_{\beta} \check g_a)' = \phi' \omega_{\beta} \check g_a + \phi \omega_{\beta}' \check g_a + \phi \omega_{\beta} \check g_a'$, estimating each term in turn as in the proof of Proposition~\ref{prop:lbound.R} and then proceeding as in the $\beta > 1/2$ case, we obtain as $a \to +\infty$
	\begin{equation*}
		\begin{aligned}
			\|(\psi_a-1) \sF \, \phi \, \omega_\beta \sF^{-1} g_a \| &\ls (\xi_a t_a)^{-l_{\beta,\eps}} \left( 1 + \| \chi_{\{|\xi| \ge 1\}} \langle \xi \rangle^{l_{\beta,\eps}} \sF \, \phi \, \omega_\beta \check g_a \| \right)\\
			&\ls (\xi_a t_a)^{-l_{\beta,\eps}} \left( 1 + \| \chi_{\{|\xi| \ge 1\}} \langle \xi \rangle^{l_{\beta,\eps} - 1} \sF \, (\phi \, \omega_\beta \check g_a)' \| \right)\\
			&\ls (\xi_a t_a)^{-l_{\beta,\eps}} \left( 1 + \| \check g_a \| + \| \check g_a \|_{\infty} + \| \check g_a' \| \right)\\
			&\ls (\xi_a t_a)^{-l_{\beta,\eps}} \langle \mu_a \rangle.
		\end{aligned}
	\end{equation*}
	Hence, going back to \eqref{eq:Sahat.mua.psia.ga.est.4}, we have as $a \to +\infty$
	\begin{equation}
		\label{eq:Sahat.mua.psia.ga.est.5}
		\|(\psi_a-1) \sF \, \phi \, W_{a} \sF^{-1} g_a\| \ls \iota(t_a) + (\xi_a t_a)^{-l_{\beta}} \langle \mu_a \rangle
	\end{equation}
	and therefore, returning to \eqref{eq:Sahat.mua.psia.ga} with the individual term estimates, we obtain
	\begin{equation}
		\label{eq:Sahat.mua.psia.ga.est.final}
		\| (\widehat\opS_{a} - \mu_a) \psi_a g_a \| = \varsigma_a^{-1} + \BigO(\varsigma_{a,\infty}^{-1} \Phi_a), \quad a \to +\infty.
	\end{equation}
	Writing $\psi_a g_a = g_a + (\psi_a - 1) g_a$, we deduce that $\| \psi_a g_a \| = 1 + \BigO(\varsigma_{a,\infty}^{-1} \Phi_a)$ as $a \to +\infty$ (see \eqref{eq:Sahat.mua.psia.ga.est.3}) and hence, applying \eqref{eq:Sahat.mua.psia.ga.est.final} and \eqref{eq:laa.conv.gralcurve}, we arrive at
	\begin{equation*}
		\left| \frac{\left\| (\widehat\opS_{a} - \mu_a) \psi_a g_a \right\|}{\| \psi_a g_a\|} - \frac1{\varsigma_{a,\infty}} \right| = \BigO(\varsigma_{a,\infty}^{-1} \Phi_a), \quad a \to +\infty.
	\end{equation*}
	Recalling that $\widehat\opS_a - \mu_a = (V_2(t_a))^{-1} \, \opU_{a,{t_a}} \, \widetilde\opH_{\la_a} \, \opU_{a, {t_a}}^{-1}$ (see \eqref{eq:Hlaahat.Sahat.gralcurves}) and letting $u_a := \opU_{a, {t_a}}^{-1} \psi_a g_a$, then $u_a \in \Dom(\widehat \opH)$ with $\supp u_a \subset \Omega_{a,+}$. It follows
	\begin{equation*}
		\left| (V_2(t_a))^{-1} \frac{\| \widehat \opH_{\la_a} u_a \|}{\| u_a \|} - \frac1{\varsigma_{a,\infty}} \right| = \BigO(\varsigma_{a,\infty}^{-1} \Phi_a), \quad a \to +\infty
	\end{equation*}
	and hence
	\begin{equation}
		\label{eq:Hlaahat.low.bound.gral.curves}
		\frac{\| \widehat \opH_{\la_a} u_a \|}{\| u_a \|} = \| (\opA_{\beta} - \mu_a)^{-1} \|^{-1} V_2(t_a) (1 + \BigO(\Phi_a)), \quad a \to +\infty.
	\end{equation}

	\underline{Step 4}
	
	We begin by updating the commutator estimate \eqref{eq:Hahatphiak.commutator.norm.est} provided in Lemma~\ref{lem:Vhat.phiak.commutator}. Starting with \eqref{eq:HaV.est} and applying \eqref{eq:Hlahat.P.def} followed by \eqref{eq:gral.curves.real.adj.asm}, we find
	\begin{equation}
		\label{eq:HlaaV.est}
		\| \opwV u \| \ls \| \widehat\opH_a u \| + a \| u \| \ls \| \widehat\opH_{\la_a} u \| + (a + b_a) \| u \| \ls \| \widehat\opH_{\la_a} u \| + a \| u \|.
	\end{equation}
	Furthermore
	\begin{equation}
		\label{eq:HlaaPVF.est}
		\| V_2^{\frac12}  \check{u} \| = \langle \opwV  u, u \rangle^{\frac12} 
		= (\Re  \langle \widehat\opH_{\la_a}  u, u \rangle + b_a \| u \|^2)^{\frac12} \le (\| \widehat\opH_{\la_a} u \| + b_a \| u \|)^{\frac12} \| u \|^{\frac12}.
	\end{equation}

	As in the proof of Lemma~\ref{lem:Vhat.phiak.commutator}, we consider firstly the case $\beta < 2$. The initial commutator estimate remains valid
	\begin{equation}
		\label{eq:Vhat.phiak.gral.curve.com.est.1}
		\| [\opwV, \phi_{a,k}] u \| \ls \| \sF V_2' \sF^{-1} \phi_{a,k}' u \| + a^{-1} \| u \|.
	\end{equation}
	Turning to the first term in the right-hand side and applying \eqref{eq:V.symb} with $n = 1$, \eqref{eq:HlaaPVF.est}, \eqref{eq:Hlaahat.away.est.R} from Step 1, \eqref{eq:gral.curves.real.adj.asm} and the fact that $t_a^2 a^{-1} = o(1)$ as $a \to +\infty$
	\begin{equation*}
		\begin{aligned}
			\| \sF V_2' \sF^{-1} \phi_{a,k}' u \| &\ls \| \phi_{a,k}' u \| + \| V_2^{\frac12} \sF^{-1} \phi_{a,k}' u \|\\
			&\ls t_a^{-1} \| \widehat\opH_{\la_a} \phi_{a,k}' u \| + (t_a^{-1} b_a + t_a) \| \phi_{a,k}' u \|\\
			&\ls (t_a^{-1} + t_a^{-1} b_a a^{-1} + t_a a^{-1}) \| \widehat\opH_{\la_a} \phi_{a,k}' u \|\\
			&\ls t_a^{-1} \| \widehat\opH_{\la_a} \phi_{a,k}' u \| \ls t_a^{-1} (a^{-\frac12} \| \widehat\opH_{\la_a} u \| + \| [\opwV, \phi_{a,k}'] u \|).
		\end{aligned}
	\end{equation*}
	Furthermore using \eqref{eq:compositionformula}, \eqref{eq:V.symb} with $n = 1$ and \eqref{eq:HlaaPVF.est}, we have
	\begin{equation*}
		\| [\opwV, \phi_{a,k}'] u \| \ls a^{-1} (\| u \| + \| V_2^{\frac12} \check{u} \|) + a^{-\frac32} \| u \| \ls a^{-1} (t_a^{-1} \| \widehat\opH_{\la_a} u \| + (b_a t_a^{-1} + t_a) \| u \|).
	\end{equation*}
	Hence
	\begin{equation*}
		\| \sF V_2' \sF^{-1} \phi_{a,k}' u \| \ls a^{-\frac12} t_a^{-1} \| \widehat\opH_{\la_a} u \| + a^{-1} (b_a t_a^{-2} + 1) \| u \|
	\end{equation*}
	and, returning to \eqref{eq:Vhat.phiak.gral.curve.com.est.1}, this yields the estimate
	\begin{equation*}
		\| [\opwV, \phi_{a,k}] u \| \ls a^{-\frac12} t_a^{-1} \| \widehat\opH_{\la_a} u \| + (a^{-1} + b_a a^{-1} t_a^{-2}) \| u \|.
	\end{equation*}
	We already established that $a^{-1} (V_2(t_a))^{-1} a^{\frac12} t_a = o(1)$ as $a \to +\infty$. Moreover by \eqref{eq:gral.curves.real.Phia.asm}
	\begin{equation*}
		b_a a^{-1} t_a^{-2} (V_2(t_a))^{-1} a^{\frac12} t_a = \mu_a (a^{\frac12} t_a)^{-1} \ls \Phi_a = o(1), \quad a \to +\infty.
	\end{equation*}

	We note that estimate \eqref{eq:HaPVF.est} plays no role in the proof of Lemma~\ref{lem:Vhat.phiak.commutator} for the case $\beta \ge 2$. On the other hand, estimate \eqref{eq:HaV.est} is indeed used but its replacement here \eqref{eq:HlaaV.est} is simply a matter of substituting $\| \widehat\opH_a u \|$ with $\| \widehat\opH_{\la_a} u \|$; the same substitution happens between \eqref{eq:Hahat.away.est.R} (from Proposition~\ref{prop:away.R}), which is also used in Lemma~\ref{lem:Vhat.phiak.commutator} when $\beta \ge 2$, and \eqref{eq:Hlaahat.away.est.R} (Step 1 above). We can therefore repeat the proof for $\beta \ge 2$ to derive
	\begin{equation*}
		\| [\opwV, \phi_{a,k}] u \| \ls a^{-\frac12} t_a^{-1} \| \widehat\opH_{\la_a} u \| + a^{-\frac12} t_a^{\beta - 1 + \BigO(\eps)} \| u \|.
	\end{equation*}

	Combining both cases, we have
	\begin{equation*}
		\| [\opwV, \phi_{a,k}] u \| \ls a^{-\frac12} t_a^{-1} \| \widehat\opH_{\la_a} u \| + \widetilde\Theta(a,\eps) \| u \|, \quad a \to +\infty,
	\end{equation*}
	where for any arbitrarily small $\eps > 0$
	\begin{equation}
		\label{eq:Hahatphiak.gral.curves.commutator.coeff}
		\widetilde\Theta(a, \eps) := 
		\begin{cases}
			a^{-1} + b_a a^{-1} t_a^{-2}, &  \beta < 2, \\[1mm]
			a^{-\frac12} t_a^{\beta - 1 + \eps}, &  \beta \ge 2.
		\end{cases}
	\end{equation}
	Moreover
	\begin{equation}
		\label{eq:Hahatphiak.gral.curves.commutator.decay}
		\widetilde\Theta(a, \eps) (V_2(t_a))^{-1} (a^{\frac12} t_a)^{1 - \eps} = o(1), \quad a \to +\infty.
	\end{equation}

	We also need to revise estimate \eqref{eq:Hhat.u1u2.quasiorth.norm.est} in Lemma~\ref{lem:Hhat.u1u2.quasiorthogonal}. We repeat the original argument, with $\widehat\opH_{\la_a}$ instead of $\widehat\opH_a$, to arrive at \eqref{eq:Hhat.u1u2.quasiorth.est.1}. Taking $N \ge 5$ in \eqref{eq:phia}, we have $\| \opR_{N+1,k} u \| \ls a^{-3} \| u \|$ for $k \in \{1, 2\}$ and $a \to +\infty$. Moreover, expanding $\| \opB_{N,k} u \|$ as in \eqref{eq:Bnk.comm.expansion} and using our new commutator estimates, we have for any arbitrarily small $\eps > 0$ and $k \in \{1, 2\}$
	\begin{equation*}
		\| \opB_{N,k} u \| \le (1 + \BigO(a^{-\frac12} t_a^{-1})) \| \widehat\opH_{\la_a} u \| + \BigO(\widetilde\Theta(a, \eps)) \| u \|, \quad a \to +\infty.
	\end{equation*}
	We can now estimate the first term in the right-hand side of \eqref{eq:Hhat.u1u2.quasiorth.est.1} as $a \to +\infty$
	\begin{equation*}
		\begin{aligned}
			\| \opB_{N,1} u \|^{\frac12} \| \opR_{N+1,2} u \|^{\frac12} &\ls a^{-\frac12} t_a^{-1} \| \widehat\opH_{\la_a} u \| + (a^{-\frac12} t_a^{-1} \widetilde\Theta(a, \eps) + a^{-\frac52} t_a) \| u \|\\
			&\ls a^{-\frac12} t_a^{-1} \| \widehat\opH_{\la_a} u \| + \widetilde\Theta(a, \eps) \| u \|,
		\end{aligned}
	\end{equation*}
	where in the last step we have used the fact that as $a \to +\infty$ we have
	\begin{align*}
		&\beta < 2, \qquad \widetilde\Theta(a,\eps) a^{\frac52} t_a^{-1} = a^{\frac32} t_a^{-1} + b_a a^{\frac32} t_a^{-3} \approx a V_2(t_a) + \mu_a (V_2(t_a))^4 \to +\infty,\\
		&\beta \ge 2, \qquad \widetilde\Theta(a,\eps) a^{\frac52} t_a^{-1} = a^2 t_a^{\beta - 2 + \eps} \to +\infty,
	\end{align*}
	(refer back to the definition \eqref{eq:Hahatphiak.gral.curves.commutator.coeff}). We can similarly estimate the other terms in the right-hand side of \eqref{eq:Hhat.u1u2.quasiorth.est.1} to find as $a \to +\infty$
	\begin{equation}
		\label{eq:Hlaahat.u1u2.quasiorth.norm.est}
		\begin{aligned}
			(\| \widehat\opH_{\la_a} \phi_{a,1} u \|^2 + \| \widehat\opH_{\la_a} \phi_{a,2} u \|^2)^{\frac12} &= \| \widehat\opH_{\la_a} (\phi_{a,1} + \phi_{a,2}) u \|\\
			&\quad + \BigO(a^{-\frac12} t_a^{-1} \| \widehat\opH_{\la_a} u \| + \widetilde\Theta(a,\eps) \| u \|).
		\end{aligned}
	\end{equation}

	Finally, we combine the above results to prove the proposition. As in the proof of Theorem~\ref{thm:R}, we have as $a \to +\infty$
	\begin{equation}
		\label{eq:Hlaahat.uk.norm.est}
		\begin{aligned}
			\| \widehat\opH_{\la_a} u_0 \| &\le (1 + \BigO(a^{-\frac12} t_a^{-1})) \| \widehat\opH_{\la_a} u \| + \BigO(\widetilde\Theta(a, \eps)) \| u \|,\\
			\| \widehat\opH_{\la_a} (u_1 + u_2) \| &\le (1 + \BigO(a^{-\frac12} t_a^{-1})) \| \widehat\opH_{\la_a} u \| + \BigO(\widetilde\Theta(a, \eps)) \| u \|,
		\end{aligned}
	\end{equation}
	with small $\eps > 0$ and $\widetilde\Theta(a, \eps)$ as in \eqref{eq:Hahatphiak.gral.curves.commutator.coeff}.
	
	By \eqref{eq:resnorm.local.R.gralcurve}, \eqref{eq:Hlaahat.u1u2.quasiorth.norm.est} and \eqref{eq:Hlaahat.uk.norm.est}, we obtain for $a \to +\infty$
	\begin{equation}
		\label{eq:uk.gral.curves.local.est}
		\begin{aligned}
			V_2(t_a) \| u_1 + u_2 \| &\le \| (\opA_{\beta} - \mu_a)^{-1} \| (1 + \BigO(\Phi_a)) \| \widehat\opH_{\la_a} (u_1 + u_2) \|\\
			&\quad + \BigO(\| (\opA_{\beta} - \mu_a)^{-1} \| a^{-\frac12} t_a^{-1}) \| \widehat\opH_{\la_a} u \|\\
			&\quad + \BigO(\| (\opA_{\beta} - \mu_a)^{-1} \| \widetilde\Theta(a,\eps)) \| u \|\\
			&\le \| (\opA_{\beta} - \mu_a)^{-1} \| (1 + \BigO(\Phi_a)) \| \widehat\opH_{\la_a} u \|\\
			&\quad + \BigO(\| (\opA_{\beta} - \mu_a)^{-1} \| \widetilde\Theta(a,\eps)) \| u \|.
		\end{aligned}
	\end{equation}

	By \eqref{eq:Hlaahat.away.est.R} and \eqref{eq:Hlaahat.uk.norm.est}, we have as $a \to +\infty$
	\begin{equation}
		\label{eq:uk.gral.curves.away.est}
		V_2(t_a) \| u_0 \| \ls a ^{-\frac12} t_a^{-1} \| \widehat\opH_{\la_a} u \| + a ^{-\frac12} t_a^{-1} \widetilde\Theta(a, \eps) \| u \|.
	\end{equation}

	Combining \eqref{eq:uk.gral.curves.local.est} and \eqref{eq:uk.gral.curves.away.est}, we find that as $a \to +\infty$
	\begin{equation*}
		\begin{aligned}
			V_2(t_a) \| u \| &\le V_2(t_a) \left(\| u_0 \| + \| u_1 + u_2 \| \right)\\
			&\le \| (\opA_{\beta} - \mu_a)^{-1} \| (1 + \BigO(\Phi_a)) \| \widehat\opH_{\la_a} u \| + \BigO(\| (\opA_{\beta} - \mu_a)^{-1} \| \widetilde\Theta(a,\eps)) \| u \|
		\end{aligned}
	\end{equation*}
	and hence
	\begin{equation*}
		\begin{aligned}
			&V_2(t_a) (1 - \BigO(\| (\opA_{\beta} - \mu_a)^{-1} \| \widetilde\Theta(a,\eps) (V_2(t_a))^{-1})) \| u \|\\
			&\hspace{5cm} \le \| (\opA_{\beta} - \mu_a)^{-1} \| (1 + \BigO(\Phi_a)) \| \widehat\opH_{\la_a} u \|.
		\end{aligned}			
	\end{equation*}
	Recalling \eqref{eq:Hahatphiak.gral.curves.commutator.coeff} and $a^{-1} t_a^2 = o(1)$, we find as $a \to +\infty$
	\begin{align*}
		\beta < 2, \qquad &\| (\opA_{\beta} - \mu_a)^{-1} \| \widetilde\Theta(a,\eps) (V_2(t_a))^{-1}\\
		&\approx \| (\opA_{\beta} - \mu_a)^{-1} \| a^{-\frac32} t_a + \mu_a \| (\opA_{\beta} - \mu_a)^{-1} \| (a^{\frac12} t_a)^{-2}\\
		&\approx \| (\opA_{\beta} - \mu_a)^{-1} \| (a^{\frac12} t_a)^{-1} a^{-1} t_a^2 + \mu_a \| (\opA_{\beta} - \mu_a)^{-1} \| (a^{\frac12} t_a)^{-2}\\
		&\ls \Phi_a,\\
		\beta \ge 2, \qquad &\| (\opA_{\beta} - \mu_a)^{-1} \| \widetilde\Theta(a,\eps) (V_2(t_a))^{-1}\\
		&= \| (\opA_{\beta} - \mu_a)^{-1} \| (a^{\frac12} t_a)^{-1 + \eps} \widetilde\Theta(a,\eps) (V_2(t_a))^{-1} (a^{\frac12} t_a)^{1 - \eps}\\
		&\ls \Phi_a,
	\end{align*}
	where in the last line we have applied \eqref{eq:Hahatphiak.gral.curves.commutator.decay}. It follows
	\begin{equation*}
		\| u \| \le \| (\opA_{\beta} - \mu_a)^{-1} \| (V_2(t_a))^{-1} (1 + \BigO(\Phi_a)) \| \widehat\opH_{\la_a} u \|.
	\end{equation*}
	This result combined with the lower bound \eqref{eq:Hlaahat.low.bound.gral.curves} yields \eqref{eq:resnorm.R.gralcurve}.
\end{proof}

We conclude this sub-section with a general construction for the level curves of the resolvent similar to that in Sub-section~\ref{sssec:resnorm.adj.iR} but considering now those adjacent to the real axis. Letting $\zeta_a := \mu_a^{\frac{3 \beta + 1}{2 \beta}} \| (\opA_{\beta} - \mu_a)^{-1} \|$, we re-write \eqref{eq:Abeta.resnorm.ex.poly} as
\begin{equation*}
	\frac{2 \beta}{\beta + 1} \mu_a^{\frac{\beta + 1}{\beta}} \exp\left(\frac{2 \beta}{\beta + 1} \mu_a^{\frac{\beta + 1}{\beta}}\right) = \frac{2 \beta}{\beta + 1} \sqrt{\frac{\beta}{\pi}} \zeta_a (1 + o(1)), \quad a \to +\infty.
\end{equation*}
This is an equation in $\mu_a$ which we can solve using the Lambert function $W_0(x)$ (refer to Sub-section~\ref{sssec:resnorm.adj.iR} for further details)
\begin{equation*}
	\frac{2 \beta}{\beta + 1} \mu_a^{\frac{\beta + 1}{\beta}} = W_0\left(\frac{2 \beta}{\beta + 1} \sqrt{\frac{\beta}{\pi}} \zeta_a (1 + o(1))\right), \quad a \to +\infty,
\end{equation*}
to deduce
\begin{equation*}
	\mu_a = \left(\frac{\beta + 1}{2 \beta}\right)^{\frac{\beta}{\beta + 1}} \left(\log(\| (\opA_{\beta} - \mu_a)^{-1} \|)\right)^{\frac{\beta}{\beta + 1}} (1 + o(1)), \quad a \to +\infty.
\end{equation*}
Using \eqref{eq:resnorm.R.gralcurve} and substituting $\| (\opH - \la_a)^{-1} \| = \eps^{-1}$, with $\eps > 0$, we obtain (recall $t_a V_2(t_a) = 2 a^{\frac12}$)
\begin{equation}
	\label{eq:resnorm.R.levelcurves}
	b_a = \left(\frac{\beta + 1}{2 \beta}\right)^{\frac{\beta}{\beta + 1}} V_2(t_a) \left(\log(V_2(t_a) \eps^{-1})\right)^{\frac{\beta}{\beta + 1}} (1 + o(1)), \quad a \to +\infty.
\end{equation}
Finally, we note that, when $V_2(x) = x^2$ (\ie~$\opH$ is the Davies operator), then $t_a = 2^{\frac13} a^{\frac16}$ and the above equation becomes
\begin{equation}
	\label{eq:resnorm.R.levelcurves.davies}
	b_a = \left(\frac32\right)^{\frac23} a^{\frac13} \left(\log(a^{\frac13} \eps^{-1})\right)^{\frac23} (1 + o(1)), \quad a \to +\infty,
\end{equation}
(compare these curves with \eqref{eq:resnorm.R.levelcurves.powerlike} for $n = 1$ and with the known formulas derived in \cite[Prop.~4.6]{BordeauxMontrieux-2013}).

\subsection{Optimality of the pseudomode construction in \cite{Krejcirik-2019-276}}
\label{ssec:optimality.pseudomodes}
In this paper, the curves in $\C$ along which the norm of the resolvent diverges are found by a non-semi-classical pseudomode construction. As a corollary of \eqref{eq:res.bounded.region.iR}, using Assumption~\ref{asm:iR}~\ref{itm:nuasm}, we find that for any $\eps>0$, the norm of the resolvent is uniformly bounded inside the region determined by 
$a \ls b^{\frac{2}{3}} x_b^{\frac{2}{3} \nu} - \eps$.
This shows that the lower edge (i.e.~the left-hand side) of the condition \cite[Eq.~(5.5)]{Krejcirik-2019-276} is optimal.

Similarly using \eqref{eq:res.bounded.region.R} we obtain optimality of the upper edge of the condition \cite[Eq.~(5.5)]{Krejcirik-2019-276} (with $\nu =-1)$. Denoting the regular variation index of $V_2$ by $\beta>0$, we obtain from \eqref{eq:ta.def} and \eqref{eq:V.regvarying.decomposition} that
\begin{equation}
	\label{eq:ta.a.Lta}
	t_a = (2 a^\frac12)^\frac{1}{1+\beta} L(t_a)^{-\frac{1}{1+\beta}}.
\end{equation}
Hence, recalling that $t_a \to +\infty$ as $a \to + \infty$ and using \eqref{eq:slowvarying.estimates}, we get (with any $\gamma>0$)
\begin{equation}\label{eq:ta.est.2}
(2 a^\frac12)^{\frac{1}{1+\beta}-\gamma} \leq 	t_a \leq  (2 a^\frac12)^{\frac{1}{1+\beta}+\gamma}, \quad a \to + \infty.
\end{equation}
Similarly from $V(x_b) = b$ we arrive at (with any $\gamma>0$)
\begin{equation}\label{eq:xb.est}
b^{\frac 1 \beta - \gamma} \leq	x_b \leq b^{\frac 1 \beta + \gamma}, \quad b \to +\infty.
\end{equation}
Then, using \eqref{eq:ta.a.Lta}, inequality \eqref{eq:res.bounded.region.R} can be rewritten as (the constant $C_{\beta,\eps'}>0$ can be given explicitly)
\begin{equation}\label{eq:a.b.cond.1}
	a> C_{\beta,\eps'} (b+\eps)^{2 +\frac2 \beta} L(t_a)^{-\frac2{\beta}}.
\end{equation}
Finally, employing \eqref{eq:ta.a.Lta}, \eqref{eq:ta.est.2} and \eqref{eq:xb.est}, the condition \eqref{eq:a.b.cond.1} is satisfied if $a \gs b^2 x_b^{2 - \gamma'}$ with  some $\gamma'>0$ which complements \cite[Eq.~(5.5)]{Krejcirik-2019-276}.

\subsection{Extension of Theorem~\ref{thm:iR} to operators in $L^2(\R)$}
\label{ssec:iR.R}
We outline a procedure to extend Theorem~\ref{thm:iR} to operators defined on the real line.
\begin{asm-sec}
	\label{asm:iR.R}
	Suppose that $V := i V_2$ with $V_2:\R \rightarrow \R$, $V_2 \in \LilocR \cap C^2((-\infty, -x_0) \cup (x_0, \infty))$ for some $x_0 \geq 0$ and let $V_{2,\pm} := V_2 \chi_{\R_\pm}$, $V_{\pm} := i V_{2,\pm}$. Assume further that the following conditions are satisfied:
	\begin{enumerate} [\upshape (i)]
		\item \label{itm:incr.unbd.iR.R} $V_2$ is unbounded and eventually increasing (in $\Rplus$)/decreasing (in $\Rminus$):
		\begin{align*}
			&\lim_{x \to +\infty} V_{2,+}(x) = +\infty, \quad V_{2,+}'(x) > 0, \quad x > x_0,\\
			&\lim_{x \to -\infty} V_{2,-}(x) = +\infty, \quad V_{2,-}'(x) < 0, \quad x < -x_0;
		\end{align*}
		\item \label{itm:nuasm.iR.R} $V_2$ has controlled derivatives: there exist $\nu_{+}, \nu_{-} \in [-1,+\infty)$ such that
		\begin{align*}
			&V_{2,+}'(x) \approx V_{2,+}(x) \: x^{\nu_{+}},& |V_{2,+}''(x)| &\lesssim V_{2,+}'(x) \: x^{\nu_{+}}, & x > x_0,\\
			&|V_{2,-}'(x)| \approx V_{2,-}(x) \: |x|^{\nu_{-}},& |V_{2,-}''(x)| &\lesssim |V_{2,-}'(x)| \: |x|^{\nu_{-}}, & x < -x_0;
		\end{align*}
		\item \label{itm:imvgrowth.iR.R}$V_2$ grows sufficiently fast: we have
		\begin{equation*}
			\Upsilon(x) = o(1), \qquad |x| \rightarrow +\infty,
		\end{equation*}
		where
		\begin{equation}\label{eq:Ups.R.def}
			\Upsilon(x) := 
			\begin{cases}
				x^{\nu_{+}} (V_{2,+}'(x))^{-\frac13}, &  x > x_0, \\[1mm] 
				|x|^{\nu_{-}} |V_{2,-}'(x)|^{-\frac13}, &  x < -x_0.
			\end{cases}
		\end{equation}
	\end{enumerate}
\end{asm-sec}
With $V$ satisfying Assumption~\ref{asm:iR.R}, we consider the Schr\"{o}dinger operator 
\begin{equation}
	\label{eq:H.iR.R.def}
	H = \Dt + V,
	\quad \Dom(H) = \WttR \cap \Dom(V)
\end{equation}
in $\Lt(\R)$ (refer to Sub-section~\ref{ssec:Schr.prelim} for details). Moreover, for sufficiently large $b > 0$, we define the turning points $x_{b,\pm}$ by
\begin{equation}\label{eq:xb.pm.def}
	V_2(x_{b,\pm}) = b, \quad \text{with} \quad x_{b,+} > x_0, \ x_{b,-} < -x_0.
\end{equation}
In the following, we use the notation $\max\{a_\pm\}:=\max\{a_+,a_-\}$, $\min\{a_\pm\}:=\min\{a_+,a_-\}$.
\begin{proposition}\label{prop:whole.R}
Let $V$ satisfy Assumption~\ref{asm:iR.R}, let $\opH$ be the Schr\"odinger operator \eqref{eq:H.iR.R.def} in $\Lt(\R)$ and let $\opA_{1,\pi/2}$ be the Airy operator \eqref{eq:airyde}. Let $b$, $x_{b,\pm}$ be as in \eqref{eq:xb.pm.def} and let $\Upsilon$ be as in \eqref{eq:Ups.R.def}. Then
\begin{equation}
\|(\opH - i b)^{-1} \|  =
\| \opA_{1,\frac \pi 2}^{-1} \| \max \{ |V_{2,\pm}'(x_{b,\pm})|^{-\frac23} \}  ( 1 + \BigO (\max \{\Upsilon(x_{b,\pm})\}) ).
\end{equation}
\end{proposition}
\begin{proof}[Sketch of proof]
To justify the above claim, we introduce a partition of unity. For $\delta_{b,\pm} := \delta |x_{b,\pm}|^{-\nu_{\pm}}$, $0 < \delta < 1/4$, with $b$ large enough so that $x_{b,+} - 2\delta_{b,+} > x_0$ and $x_{b,-} + 2\delta_{b,-} < -x_0$, let $\phi_{b,0}, \: \phi_{b,\pm} \in \CiR$, $0 \le \phi_{b,0}, \: \phi_{b,\pm} \le 1$, satisfying
\begin{align*}
	\phi_{b,+}(x) &:= 1, \; x \in [x_{b,+} - \delta_{b,+}, \infty), & \supp \phi_{b,+} &\subset [x_{b,+} - 2\delta_{b,+}, \infty),\\
	\phi_{b,-}(x) &:= 1, \; x \in (-\infty, x_{b,-} + \delta_{b,-}],& \supp \phi_{b,-} &\subset (-\infty, x_{b,-} + 2 \delta_{b,-}],\\
	\phi_{b,0} &:= 1 - (\phi_{b,+} + \phi_{b,-}), &\|\phi_{b,\pm}^{(j)}\|_{\infty}& \lesssim \delta_{b,\pm}^{-j}, \; j \in \{ 1, 2 \}.
\end{align*}		

For convenience, we shall denote $\alpha_{b,\pm} = |V_{2,\pm}'(x_{b,\pm})|^{-\frac23}$, $\gamma_{b,\pm} = \BigO ( \Upsilon(x_{b,\pm}) )$, $\Lambda_{b,\pm} = \alpha_{b,\pm} (1 + \gamma_{b,\pm})$ and $H_b=H-ib$. For $u \in \Dom(\opH)$, we write $u = u_0 + u_{+} + u_{-}$, with $u_0 := \phi_{b,0} u$, $u_{\pm} := \phi_{b,\pm} u$, and introduce the operators in $\Lt(\R_{\pm})$
\begin{equation*}
	\opH_{\pm} = \Dt + V_{\pm},
	\quad \Dom(\opH_{\pm}) = W^{2,2}(\R_{\pm}) \cap W_0^{1,2}(\R_\pm) \cap \Dom(V_{\pm}).
\end{equation*}
Since $V_{+}$ satisfies Assumption~\ref{asm:iR} and $u_{+} \in \Dom(\opH_{+})$, we have by \eqref{eq:Hb.est.iR}
\begin{equation}
	\label{eq:Hla.up.below.iR.R}
	\| \opH_b u_{+} \| = \| (\opH_{+} - ib) u_{+} \| \ge \| \opA_{1,\frac \pi 2}^{-1} \|^{-1} \alpha_{b,+}^{-1} (1 - \gamma_{b,+}) \| u_{+} \|, \quad b \to + \infty.
\end{equation}
Moreover, with the isometry $(Uu)(x):=u(-x)$ in $L^2(\R)$, it is easy to see
\begin{equation}
	\label{eq:Hla.um.below.iR.R}
	\| \opH_b  u_{-} \| \ge \| \opA_{1,\frac \pi 2}^{-1} \|^{-1} \alpha_{b,-}^{-1} (1 - \gamma_{b,-}) \| u_{-} \|, \quad b \to + \infty.
\end{equation}
Since $u_{+} \perp u_{-}$ and $\opH_b u_{+} \perp \opH_b u_{-}$ in $L^2$, by combining \eqref{eq:Hla.up.below.iR.R} and \eqref{eq:Hla.um.below.iR.R} we find
\begin{align*}
	\| u_{+} + u_{-}\|^2 &= \| u_{+} \|^2 + \| u_{-} \|^2 \le \| \opA_{1,\frac \pi 2}^{-1} \|^2 (\Lambda_{b,+}^2 \| \opH_b u_{+} \|^2 + \Lambda_{b,-}^2 \| \opH_b u_{-} \|^2)\\
	&= \| \opA_{1,\frac \pi 2}^{-1} \|^2 \| \opH_b (\Lambda_{b,+} u_{+} + \Lambda_{b,-} u_{-}) \|^2
\end{align*}
and therefore
\begin{equation}
	\label{eq:Hla.upum.below.iR.R}
	\| \opA_{1,\frac \pi 2}^{-1} \| \| \opH_b (\Lambda_{b,+} u_{+} + \Lambda_{b,-} u_{-}) \| \ge \| u_{+} + u_{-}\|.
\end{equation}

Since $\supp u_{0} \subset [x_{b,-} + \delta_{b,-}, x_{b,+} - \delta_{b,+}]$, then arguing as in Proposition~\ref{prop:away.iR} we deduce that for large enough $b$
\begin{equation*}
	\langle |V_2 - b| u_0, u_0 \rangle \le \| \opH_b u_0 \| \| u_0 \|.
\end{equation*}
It follows that for $x \in [x_{b,-} + \delta_{b,-}, x_{b,+} - \delta_{b,+}]$ and sufficiently large $b$
\begin{equation*}
	|V_2(x) - b| \gs \min\{|V_{2,\pm}'(x_{b,\pm})| \delta_{b,\pm}\} \approx b,
\end{equation*}
reasoning as in the proof of Proposition~\ref{prop:away.iR} and applying Assumption~\ref{asm:iR.R}~\ref{itm:nuasm.iR.R}. Hence
\begin{equation}
	\label{eq:Hla.u0.below.iR.R}
	b^{-1} \| \opH_b u_0 \| \gs \| u_0 \|, \quad b \to +\infty.
\end{equation}

Furthermore, arguing as in the proof of \eqref{eq:Hbphikb.commutator.norm.est}, we are able to derive upper estimates as  $b \to +\infty$
\begin{equation}
	\label{eq:Hla.upum.above.iR.R}
	\begin{aligned}
		&\| \opH_b (\Lambda_{b,+} u_{+} + \Lambda_{b,-} u_{-}) \|\\
		&\quad = \| \opH_b (\Lambda_{b,+} \phi_{b,+} + \Lambda_{b,-} \phi_{b,-}) u \|\\
		&\quad \le \| (\Lambda_{b,+} \phi_{b,+} + \Lambda_{b,-} \phi_{b,-}) \opH_b u \| + 2 \| (\Lambda_{b,+} \phi_{b,+}' + \Lambda_{b,-} \phi_{b,-}') u' \|\\
		&\quad\quad + \| (\Lambda_{b,+} \phi_{b,+}'' + \Lambda_{b,-} \phi_{b,-}'') u \|\\
		&\quad \le \max\{\Lambda_{b,\pm}\} \| \opH_b u \| + (\Lambda_{b,+} \gamma_{b,+} + \Lambda_{b,-} \gamma_{b,-}) \| \opH_b u \|\\
		&\quad\quad + (\Lambda_{b,+} x_{b,+}^{2\nu_{+}} \gamma_{b,+}^{-1} + \Lambda_{b,-} |x_{b,-}|^{2\nu_{-}} \gamma_{b,-}^{-1}) \| u \|\\	&\quad\quad + (\Lambda_{b,+} x_{b,+}^{2\nu_{+}} + \Lambda_{b,-} |x_{b,-}|^{2\nu_{-}}) \| u \|\\
		&\quad \le \max\{\Lambda_{b,\pm}\} (1 + \gamma_{b,+} + \gamma_{b,-}) \| \opH_b u \| + (\gamma_{b,+} + \gamma_{b,-}) \| u \|
	\end{aligned}
\end{equation}
and
\begin{equation*}
	\begin{aligned}
		\| \opH_b u_0 \| &= \| \opH_b \phi_{b,0} u \| \le \| \phi_{b,0} \opH_b \| + 2 \| \phi_{b,0}' u' \| + \| \phi_{b,0}'' u \|\\
		&\ls \| \opH_b u \| + (x_{b,+}^{2\nu_{+}} + |x_{b,-}|^{2\nu_{-}}) \| u \|.
	\end{aligned}
\end{equation*}
By Assumption~\ref{asm:iR.R}~\ref{itm:nuasm.iR.R}, we have $b^{-1} \approx \alpha_{b,\pm} \gamma_{b,\pm}$, and therefore
\begin{equation}
	\label{eq:Hla.u0.above.iR.R}
	b^{-1} \| \opH_b u_0 \| \ls b^{-1} \| \opH_b u \| + (\gamma_{b,+}^3 + \gamma_{b,-}^3) \| u \|.
\end{equation}	

Combining the lower and upper estimates \eqref{eq:Hla.upum.below.iR.R}, \eqref{eq:Hla.u0.below.iR.R}, \eqref{eq:Hla.upum.above.iR.R} and \eqref{eq:Hla.u0.above.iR.R} and noting as above $b^{-1} \approx \alpha_{b,\pm} \gamma_{b,\pm}$, we have as $b \to +\infty$
\begin{align*}
	\| u \| &\le \| u_0 \| + \| u_{+} + u_{-} \|\\
	&\le (\| \opA_{1,\frac \pi 2}^{-1} \| \max \{\Lambda_{b,\pm}\} (1 + \gamma_{b,+} + \gamma_{b,-}) + \BigO(b^{-1})) \| \opH_b u \|
	+ (\gamma_{b,+} + \gamma_{b,-}) \| u \|\\
	&\le \| \opA_{1,\frac \pi 2}^{-1} \| \max \{\alpha_{b,\pm}\} (1 + \gamma_{b,+} + \gamma_{b,-}) \| \opH_b u \| + (\gamma_{b,+} + \gamma_{b,-}) \| u \|.
\end{align*}
Hence, by Assumption~\ref{asm:iR.R}~\ref{itm:imvgrowth.iR.R}, for any $u \in \Dom(\opH)$ we obtain 
\begin{equation*}
	\| \opH_b^{-1} \| \le \| \opA_{1,\frac \pi 2}^{-1} \| \max \{\alpha_{b,\pm}\} (1 + \BigO(\max \{\Upsilon(x_{b,\pm})\})), \quad b \to +\infty.
\end{equation*}
If $\max \{\alpha_{b,\pm}\} = \alpha_{b,+}$, using Proposition~\ref{prop:lbound.iR} we can find a family of functions $u_b \in \Dom(\opH)$ such that as $b \to +\infty$
\begin{equation*}
	\| u_b \| =  \| \opA_{1,\frac \pi 2}^{-1} \| \alpha_{b,+} (1 + \gamma_{b,+}) \| \opH_b u_b \|
\end{equation*}
and it therefore follows as $b \to +\infty$
\begin{equation*}
	\| \opH_b^{-1} \| \ge \| \opA_{1,\frac \pi 2}^{-1} \| \max \{\alpha_{b,\pm}\} (1 - \BigO(\max \{\Upsilon(x_{b,\pm})\})).
\end{equation*}
We can similarly argue when $\max \{\alpha_{b,\pm}\} = \alpha_{b,-}$.
\end{proof}
Our strategy to prove Theorem~\ref{thm:iR} can be re-deployed, with minimal and obvious changes, when Assumption~\ref{asm:iR.R}~\ref{itm:incr.unbd.iR.R} is replaced with 
\begin{align*}
	&\lim_{x \to +\infty} V_{2,+}(x) = +\infty, \quad V_{2,+}'(x) > 0, \quad x > x_0,\\
	&\lim_{x \to -\infty} V_{2,-}(x) = -\infty, \quad V_{2,-}'(x) > 0, \quad x < -x_0
\end{align*}
and $V_{2,+}(x_{b,+}) = b$, $V_{2,-}(x_{b,-}) = -b$, to prove that as $b \to +\infty$
\begin{equation}
	\label{eq:res.norm.iR.R.posneg}
	\|(\opH - i (\pm b))^{-1} \| = \| \opA_{1,\frac \pi 2}^{-1} \| \left( V_{2,\pm}'(x_{b,\pm}) \right)^{-\frac{2}{3}} \left( 1 + \BigO( \Upsilon(x_{b,\pm})) \right),
\end{equation}
where we have used the fact that $\opA_{1,-\pi/2} = \opA_{1,\pi/2}^*$ and therefore $\| \opA_{1,-\pi/2}^{-1} \| = \| \opA_{1,\pi/2}^{-1} \|$. One can analogously treat the case where the potential is bounded on one of the half-lines and unbounded on the other one.

Finally, without going into details, we remark that our analysis for general curves in the numerical range (see Subsection~\ref{ssec:inside.Num}) can be extended, using the above methodology, to the whole line. For example, with $V$ satisfying Assumption~\ref{asm:iR.R}, and $\rho_{\pm} = |V_2'(x_{b,\pm})|^{-1/3}$, $\mu_{b,\pm} = \rho_{\pm}^2 a$, $\Phi_{b,\pm} = \langle\mu_{b,\pm}\rangle^2 \| (\opA_{1,\pi/2} - \mu_{b,\pm})^{-1} \| \Upsilon(x_{b,\pm})$, and assuming
\begin{equation*}
	\Phi_{b,\pm} = o(1), \quad b \to +\infty,
\end{equation*}
we find as $b \to +\infty$
\begin{equation}
	\label{eq:resnorm.iR.R.gralcurve}
	\|(\opH - \la_b)^{-1} \| = \max\{\| (\opA_{1,\frac{\pi}{2}} - \mu_{b,\pm})^{-1} \| |V_2'(x_{b,\pm})|^{-\frac{2}{3}} \} (1 + \BigO(\Phi_{b,\pm})).
\end{equation}

\subsection{Extension of Theorem~\ref{thm:R} to operators in $L^2(\R_+)$}
\label{ssec:R.Rp}
We briefly indicate how Theorem~\ref{thm:R} can be adapted for operators $H_+=-\partial_x^2+ V_+$ in $L^2(\Rplus)$ subject to a Dirichlet boundary condition at $0$ and with $V_+ := i V_{2,+}$ satisfying the conditions in Assumption~\ref{asm:R} for $x>0$. The even extension $V_2$ of $V_{2,+}$ to $\R$, and the corresponding complex potential $V := i V_2$, satisfy Assumption~\ref{asm:R} up to a possible lack of smoothness at $0$, which can however be removed by a compactly supported perturbation $W$, similarly as in Subsection~\ref{ssec:inside.Num}. Then Theorem~\ref{thm:R} can be applied to $H=-\partial_x^2 + V$ in $L^2(\R)$. Since the odd extension of each $u_{+} \in \Dom(H_+)$ belongs to $\Dom(H)$ and for each odd $u \in \Dom(H)$, we have $(H u)_{\restriction \Rplus} = H_+ (u)_{\restriction \Rplus}$, we arrive at the desired inequality for any $u_+ \in \Dom(H_+)$ (see \eqref{eq:Hb.est.R} in the proof of Theorem~\ref{thm:R})
\begin{equation*}
 \| \opA_{\beta}^{-1} \| (V_{2,+}(t_a))^{-1} ( 1 + \BigO ( \iota(t_a) + (a^{\frac12} t_a)^{-1 + \eps} ) ) \| (\opH_+-a) u_+ \| \geq \| u_+ \|.
\end{equation*}

\subsection{Extension of Theorem~\ref{thm:iR} to radial operators}
\label{ssec:rad}
Let $v: \Rplus \to \Rplus$ and consider the Schr\"odinger operator in $L^2(\Rd)$ with  $d \ge 2$
\begin{equation}\label{eq:H.rad.def}
	\opH = -\Delta + i v(|x|), \quad \Dom(\opH) = W^{2,2}(\Rd) \cap \Dom(v(|\cdot|)).
\end{equation}
We justify below that the claim of Theorem~\ref{thm:iR} remains valid in this case;  a relatively small real part of the potential (in the sense of Assumption~\ref{asm:iR}) can be added in a straightforward way but we omit the details. 
\begin{proposition}\label{prop:rad}
Let $H$ be the radial Schr\"odinger operator in $L^2(\Rd)$ as in \eqref{eq:H.rad.def} with  $d \ge 2$ and with $v$ such that $V:=i v$ satisfies Assumption~\ref{asm:iR}. Then
\begin{equation}\label{eq:H.rad.claim}
	\|(\opH - i b)^{-1} \| = \| \opA_{1,\frac{\pi}{2}}^{-1} \| v'(x_b)^{-\frac{2}{3}} \left( 1 + \BigO \left(\Upsilon(x_b)\right) \right), \quad b \to +\infty,
\end{equation}
where $x_b>0$ is defined by the equation $v(x_b)=b$ for all sufficiently large $b$.
\end{proposition}
\begin{proof}[Sketch of proof]
The first step of the proof (see Sub-section~\ref{sssec:step.1.iR}) can be performed in the same way using the multidimensional operator $\opH$, \ie~we split $\Rd$ into $\Omega'_b = \{x \in \Rd \, : \, ||x|-x_b| \leq \delta_b\}$, with $\delta_b = \delta x_b^{-\nu}$, and the rest.

In the second step (see Sub-section~\ref{sssec:step.2.iR}), we decompose $\opH-ib$ in a standard way into a countable family (labelled by $l \in \N_0$) of one-dimensional operators in $L^2(\Rplus)$ 
\begin{equation}
	H_{b,l}	= \Dtr + \frac{c_{l,d}}{r^2} + i (v(r)-v(x_b)), \quad c_{l,d} = l(l+d-2) + \frac14(d-1)(d-3)
\end{equation}
with appropriate domains (see \eg~\cite[Chap.~18]{Weidmann-2003} for details)). The challenge is to obtain suitable estimates for all $l \in \N_0$ and all $b>b_0$ with $b_0$ independent of $l$. 

Following the same procedure (in particular shift and scaling and using the fact that $\supp u \subset \Omega_b$) as in Sub-section~\ref{sssec:step.2.iR}, we arrive at operators in $L^2(\R)$
\begin{equation}
	S_{b,l} = A + \la_{b,l} + (T_{b,l}-\la_{b,l}) \chi_{\Omega_{b,\rho}} + R_b, \quad b>0, \ l \in \N_0,
\end{equation}
where $\rho := v'(x_b)^{-\frac{1}{3}}$,  $\Omega_{b,\rho} := ( - 2 \delta_b \rho^{-1}, 2 \delta_b \rho^{-1})$, 
\begin{equation}
	A:= \Dt + ix, \quad T_{b,l}:=\frac{c_{l,d} \rho^2 }{(\rho x + x_b)^2}, \quad 
	R_b(x) := i \frac12 \frac{v''(\tilde{s} \rho x + x_b)}{ v'(x_b)} \rho x^2 \chi_{\Omega_{b,\rho}}(x)
\end{equation}
with $0 \leq \tilde s \leq 1$ (see \eqref{eq:Rb.def}) and
\begin{equation}
	\la_{b,l}:= \frac{c_{l,d} \rho^2}{x_b^{2}}  = c_{l,d} \frac{\Upsilon^{2}(x_b)}{x_b^{2+2\nu}}.
\end{equation}
Note that for a fixed $l \in \N_0$, $\la_{b,l} \to 0$ as $b \to + \infty$ and that $\la_{b,l} \geq 0$ for all $l \ge l_d \in \N_0$ ($l_d$ can be set to 0 for $d > 2$ and to 1 for $d = 2$) and all large $b$.  

An important observation is that the graph norm of $A$ satisfies (uniformly for all $l \ge l_d$ and all large $b$)
\begin{equation}
	\label{eq:airy.labl.inequality}
	\| (A + \la_{b,l})u \| + \|u\| \gs \|u''\| + \|x u\| + \la_{b,l} \|u\| + \|u\|, \quad u \in \Dom(A).
\end{equation}
To see this, it is enough to note
\begin{equation*}
	\| (A + \la_{b,l})u \|^2 = \| A u \|^2 + \la_{b,l}^2 \| u \|^2 + 2 \la_{b,l} \| u' \|^2
\end{equation*}
and to apply \eqref{eq:airyinequality}. This equation also shows that $\| (A + \la_{b,l})u \| \ge \| (A + \la_{b,l'})u \|$ for $l \ge l' \ge l_d$ and hence
\begin{equation}
	\label{eq:airy.labl.monotonicity}
	\| (A + \la_{b,l})^{-1} \| \le \| (A + \la_{b,l'})^{-1} \|, \quad l \ge l' \ge l_d, \quad b > 0.
\end{equation}
Furthermore, since $\la_{b,l_d} \to 0$ as $b \to +\infty$ and $(\opA - z)^{-1}$ is bounded on bounded sets in $\C$, we can find $b_0 > 0$ (independent of $l$) such that for all $b \ge b_0$ we have $\|(\opA + \la_{b,l_d})^{-1}\| \ls 1$. It follows from \eqref{eq:airy.labl.monotonicity} that $\|(\opA + \la_{b,l})^{-1}\| \ls 1$ for all $l \ge l_d$ and all $b \ge b_0$. Note that this last estimate combined with \eqref{eq:airy.labl.inequality} implies that $\|x (\opA + \la_{b,l})^{-1}\| \ls 1$ for all $l \ge l_d$ and all $b \ge b_0$.

The estimates of $R_b$ (see \eqref{eq:Rbdecay}) remain valid and we have (uniformly in $l$)
\begin{equation}
	\left\|  \frac{T_{b,l}-\la_{b,l}}{\la_{b,l}} \chi_{\Omega_{b,\rho}}\right\|_\infty \ls \frac{\delta}{x_b^{1+\nu}}, 
	\qquad 
	\left\|  \frac{T_{b,l}-\la_{b,l}}{\la_{b,l} x} \chi_{\Omega_{b,\rho}}\right\|_\infty \ls \frac{\Upsilon (x_b)}{x_b^{1+\nu}},
\end{equation}
as $b \to +\infty$. Thus
\begin{equation}
	S_{b,l} = \left( I + \frac{T_{b,l}-\la_{b,l}}{\la_{b,l}}  \chi_{\Omega_{b,\rho}} \la_{b,l} (A+\la_{b,l})^{-1} + \frac{R_b}{x} x (A+\la_{b,l})^{-1} \right) (A+\la_{b,l}).
\end{equation}
is invertible and its graph norm is equivalent to that of $A+\la_{b,l}$ (uniformly in $l$). Moreover, by the second resolvent identity, the previous estimates and \eqref{eq:Rbdecay}, we obtain (uniformly in $l$)
\begin{equation*}
	\begin{aligned}
		\| S_{b,l}^{-1} - (A+\la_{b,l})^{-1} \| &
		\leq \| S_{b,l}^{-1} x\| \| (\la_{b,l} x)^{-1} (T_{b,l}-\la_{b,l})  \chi_{\Omega_{b,\rho}}\|_\infty \|\la_{b,l} (A+\la_{b,l})^{-1}\|
		\\
		& \quad + \| S_{b,l}^{-1} x\| \| x^{-2 }R_b\| \|x (A+\la_{b,l})^{-1}\|
		\\
		& \ls \Upsilon(x_b) x_b^{-(1+\nu)} + \Upsilon(x_b), \quad b \to + \infty.
	\end{aligned}
\end{equation*}
Since $\la_{b,l} \geq 0$ for all $l \ge l_d$ and all large $b$ and $A$ is m-accretive, we get 
\begin{equation}
	\| S_{b,l}^{-1}\| = \|A^{-1} \| \left( 1 + \BigO(\Upsilon(x_b)) \right), \quad b \to + \infty;
\end{equation}
for finitely many $l \in \N_0$, $l<l_d$, the same claim follows by treating $T_{b,l}$ as a perturbation. The rest of the proof of this step is the same as the one in Sub-section~\ref{sssec:step.2.iR} and can be reformulated as an estimate for the full operator $\opH$.

The third step (see Sub-section~\ref{sssec:step.3.iR}) can be performed for $S_{b,l}$ with a fixed $l$ and so it requires only minor and straightforward adjustments. 

The last step (see Sub-section~\ref{sssec:step.4.iR}) is completely analogous.
\end{proof}

\subsection{Remarks on semi-classical operators}
\label{ssec:s-c}
We indicate how the strategy of Theorem~\ref{thm:iR} applies in the semi-classical case for the operator $H_h = - h^2 \partial_x^2 + V$ in $L^2(\Rplus)$ subject to Dirichlet boundary condition at $0$ with $h > 0, \; h \to 0$ and $V := i V_2$, noting that resolvent norm bounds for semi-classical Schr\"odinger operators with purely imaginary potential have been derived in \eg~\cite{Pravda-Starov-2006-73,Dencker-2004-57,BordeauxMontrieux-2013,Sjoestrand-2009,Bellis-2018-9,Bellis-2019-277,Henry-2014,Almog-2016-48}. We assume that $0 \leq V_2 \in C^2(\ov \Rplus)$ satisfies the conditions in Sub-section~\ref{ssec:Schr.prelim} so that $H_h$ is m-accretive. Suppose in addition that $x_0 \in \Rplus$ is such that $V_2'(x_0) \neq 0$ and there is $\delta_0 >0$ such that
\begin{equation}\label{eq:semi-preim}
	\begin{aligned}
	&\inf_{|x-x_0|\ge\delta_0} |V_2(x)-V_2(x_0)| \\
	& \qquad \gs \min\{|V_2(x_0-\delta_0) - V_2(x_0)|, |V_2(x_0+\delta_0) - V_2(x_0)|\}.			
	\end{aligned}
\end{equation}
We focus on the resolvent estimate for the spectral point $\la = V(x_0)$ from the range of the potential.

In Step 1 (see Sub-section~\ref{sssec:step.1.iR}), one considers functions $u \in \Dom(H_h)$ with $\supp u \cap (x_0-\delta_h, x_0+\delta_h) = \emptyset$ with a suitably selected $\delta_h \to 0+$ as $h \to 0$. Then the quadratic form estimate (see Proposition~\ref{prop:away.iR}), Taylor's theorem and \eqref{eq:semi-preim} yield (for the considered functions $u$)
\begin{equation}\label{eq:Hh.1}
	\|(H_h - \la) u \| \gs \delta_h \|u\|, \quad h \to 0.
\end{equation}

In Step 2 (see Sub-section~\ref{sssec:step.2.iR}), for functions $u$ supported in $\cI:=(x_0-2\delta_h,x_0+2\delta_h)$, taking out the factor $h^2$, the shift $x \mapsto x+x_0$, rescaling $x \mapsto \sigma x$ and Taylor's theorem lead to operators in $L^2(\R)$
\begin{equation}
	T_h = \sigma^{-2} \left( -\partial_x^2 + i h^{-2} \sigma^3 V_2'(x_0) x + i h^{-2} \frac{V_2''(\xi)}{2} \sigma^4 x^2 \chi_{\cI}(\sigma x + x_0) \right).
\end{equation}
Selecting $\sigma$ so that $\sigma^3 h^{-2}=1$, we obtain
\begin{equation}
T_h = h^{-\frac 43} \left( -\partial_x^2 + i V_2'(x_0) x + W_h(x) \right),
\end{equation}
where $\|W_h\| = \BigO(h^{-\frac23}\delta_h^2)$ as $h \to 0$. Hence choosing $\delta_h = h^{\frac13 + \eps}$ with $\eps>0$, we readily obtain the norm resolvent convergence to the Airy operator $A_{r,\theta}$, with $r = |V_2'(x_0)|$ and $\theta = \sgn(V_2'(x_0)) \pi/2$, see Sub-section~\ref{ssec:Airy.prelim},
\begin{equation}\label{eq:Th.est}
h^{\frac 43} T_h \to -\partial_x^2 + i V_2'(x_0) x, \quad h \to 0.
\end{equation}
Thus, rewritting \eqref{eq:Th.est} for $H_h$, we arrive at (for the considered functions $u$)
\begin{equation}\label{eq:Hh.2}
\|(H_h - \la) u \| \geq h^\frac23 \|A_{r,\theta}^{-1}\|^{-1} (1-\BigO(h^{-\frac 23} \delta_h^2)) \|u\|, \quad h \to 0.
\end{equation}

Following the strategy in Step 4 (see Sub-section~\ref{sssec:step.4.iR}), we combine the estimates \eqref{eq:Hh.1}, \eqref{eq:Hh.2} above. To this end we employ a cut-off $\phi$ satisfying $\phi(x) = 1$ for $x\in [x_0-\delta_h, x_0+\delta_h]$, $\phi(x) = 0$ for $x \notin (x_0-2 \delta_h, x_0+2\delta_h)$ and $\|\phi^{(j)}\|_\infty \ls \delta_h^{-j}$, $j=1,2$. Moreover, a simple estimate of the quadratic form yields $\|u'\|^2 \leq h^{-2} \| (H_h - \la) u \| \| u \|$, $u \in \Dom(H_h)$. By following the steps in Step 4, we obtain
\begin{equation}
	\label{eq:s-c.ubound.iR}
	\|(H_h - \la) u \| \geq h^\frac23 \|A_{r,\theta}^{-1}\|^{-1} (1-\BigO(h^{-\frac23} \delta_h^2)) \|u\|, \quad u \in \Dom(H_h)
\end{equation}
as $h \to 0$.

Finally, it is straightforward to adapt the reasoning in Proposition~\ref{prop:lbound.iR} (see Sub-section~\ref{sssec:step.3.iR}) to prove that the bound \eqref{eq:s-c.ubound.iR} is optimal and we omit the details.

\section{An inverse problem}
\label{sec:bbt.iR}
In \cite[Thm.~1.5]{Batty-2016-270}, the authors relate the rate of time-decay of the norm of a one-parameter semigroup to the rate of growth of the norm of the resolvent of its generator along the positive part of the imaginary axis. Inspired by the presentation on this topic in \cite{CB-CIRM}, we consider the following problem. Which assumptions must a non-negative, unbounded function $r:\Rplus \rightarrow \Rplus$ satisfy for there to exist a potential $i V_2$ such that the Schr\"{o}dinger operator $\opH = \Dt + i V_2$ verifies $\|(\opH - i b)^{-1} \| = r(b)$ as $b \to +\infty$? Theorem~\ref{thm:iR} enables us to answer this question as follows.

\begin{proposition}
	\label{prop:batty}
	Let $r \in C^1(\ov{\Rplus};\Rplus)$ and $r(y) \to +\infty$ as $y \to +\infty$. Assume furthermore that $r$ satisfies the following conditions as $y \to +\infty$:
	\begin{align}
		\label{eq:batty.prop.asm1}&\int_{0}^{y} r^{\frac32}(u) \dd u \ls y \; r^{\frac32}(y),
		\\
		\label{eq:batty.prop.asm2}&\frac{|r'(y)|}{r^{\frac52}(y)} \int_{0}^{y} r^{\frac32}(u) \dd u \ls 1,
		\\
		\label{eq:batty.prop.asm3}&\frac{r^{\frac12}(y)}{\int_{0}^{y} r^{\frac32}(u) \dd u} = o(1).
	\end{align}
	Then the potential $V := i V_2$, where $V_2$ is a real function determined by the equation
	\begin{equation}
		\label{eq:r.int.cond}
		\| \opA_{1,\frac \pi 2}^{-1}\|^{-\frac 32} \int_{0}^{V_2(x)} r^{\frac32}(u) \, \dd u = x, \quad x \geq 0,
	\end{equation}
	with $\opA_{1,\pi/2}$ as in \eqref{eq:airyde},
	satisfies Assumption~\ref{asm:iR} with $\nu = -1$ and
	\begin{equation}
		\label{eq:r.diff.cond}
		\| \opA_{1,\frac \pi 2}^{-1} \| (V_2'(x_b))^{-\frac23} = r(b), \quad b \to +\infty,
	\end{equation}
	with $x_b$ as in \eqref{eq:xb.def}.
	
	If $r \in C^1(\ov{\Rplus};\Rplus)$ is regularly varying with positive index, it is eventually increasing and it satisfies
	\begin{equation}\label{eq:r.nu.asm}
		|r'(y)| \ls r(y) y^{-1}, \quad y \to +\infty,
	\end{equation}	
	then the conditions \eqref{eq:batty.prop.asm1}--\eqref{eq:batty.prop.asm3} hold.
\end{proposition}
\begin{proof}
	Note that \eqref{eq:r.int.cond} can be indeed solved as the left-hand side is an increasing function in $y:=V_2(x)$. It is immediate that $V_2$ determined by \eqref{eq:r.int.cond} satisfies \eqref{eq:r.diff.cond}. Moreover such $V_2$ is unbounded and increasing. It remains to verify Assumptions~\ref{asm:iR}~\ref{itm:nuasm} and \ref{itm:imvgrowth}. Firstly, by differentiating \eqref{eq:r.int.cond} and employing \eqref{eq:batty.prop.asm1}, we have
	\begin{equation*}
		\frac{V_2'(x) x}{V_2(x)} \approx \frac{x}{y  r^{\frac32}(y)} \approx \frac{\int_{0}^{y} r^{\frac32}(u) \dd u}{y r^{\frac32}(y)} \ls 1, \quad x \to +\infty.
	\end{equation*}
	Similarly using \eqref{eq:batty.prop.asm1} and \eqref{eq:batty.prop.asm2}
	\begin{equation*}
		\frac{|V_2''(x)| x}{V_2'(x)} \approx  \frac{|r'(y)| x}{ r^{\frac52}(y)} \approx \frac{|r'(y)| \int_{0}^{y} r^{\frac32}(u) \dd u}{r^{\frac52}(y)} \ls 1, \quad x \to +\infty.
	\end{equation*}
	Lastly, by \eqref{eq:batty.prop.asm3}
	\begin{equation*}
		\frac{(V_2'(x))^{-\frac13}}{x} \approx \frac{r^{\frac12}(y)}{\int_{0}^{y} r^{\frac32}(u) \dd u}  \to 0, \quad x \to +\infty.
	\end{equation*}

	As for the second statement in the Proposition, let $r$ be regularly varying with index $\beta > 0$ (see Sub-section~\ref{ssec:reg.var}) and eventually increasing and assume furthermore that it satisfies \eqref{eq:r.nu.asm}. From the facts that $r$ is bounded on compact subsets of $\ov\Rplus$ and that it is eventually increasing, we have as $y \to +\infty$
	\begin{equation*}
		\frac{\int_{0}^{y} r^{\frac32}(u) \dd u}{y \: r^{\frac32}(y)} = \frac{\int_{0}^{y_0} r^{\frac32}(u) \dd u + \int_{y_0}^{y} r^{\frac32}(u) \dd u}{y \: r^{\frac32}(y)} \ls \frac1{y \: r^{\frac32}(y)} + \frac{y - y_0}{y}\ls 1.
	\end{equation*}
	Moreover, using \eqref{eq:r.nu.asm} and the previous estimate,
	\begin{equation*}
		\frac{|r'(y)| \int_{0}^{y} r^{\frac32}(u) \dd u}{r^{\frac52}(y)} 
		\ls 
		\frac{\int_{0}^{y} r^{\frac32}(u) \dd u}{y \: r^{\frac32}(y)} \ls 1, \quad y \to +\infty.
	\end{equation*}
	Finally, calling $W_y(t) = r(y t) / r(y), \; \omega_\beta = t^\beta, \; t \ge 0,$ and arguing as in Lemma~\ref{lem:reg.var.comp}, it is possible to show that $\|( W_y - \omega_\beta ) \, \chi_{[0, 1]} \|_{\infty} \to 0$ as $y \to +\infty$, and we have
	\begin{equation*}
		\frac{r^{\frac12}(y)}{\int_{0}^{y} r^{\frac32}(u) \dd u} = \frac{\left( \int_{0}^{y} \left(\frac{r(u)}{r(y)}\right)^{\frac32} \dd u \right)^{-1}}{r(y)} = \frac{\left( \int_{0}^{1} \left(W_y(t)\right)^{\frac32} \dd t \right)^{-1}}{r(y) \: y} \ls \frac{1}{r(y) \: y} \to 0,
	\end{equation*}
	for $y \to +\infty$, as required.
\end{proof}

\begin{example}
A basic example of a function satisfying the conditions of Proposition~\ref{prop:batty} is $r(y) = \langle y\rangle^{\alpha}$ with $\alpha > 0$, which is regularly varying and increasing and for which \eqref{eq:r.nu.asm} holds. The sought-after potential satisfies $V_2(x) \approx x^{2/(2+3\alpha)}$, \ie~it is, as expected, sub-linear (see also the examples in Section~\ref{sec:examples}).
\end{example}

\begin{example}
We remark that conditions \eqref{eq:batty.prop.asm1}--\eqref{eq:batty.prop.asm3} include many others rates, growing both faster (\eg~$r(y) = \exp(y^{\alpha})$ with $\alpha > 0$) and more slowly (\eg~$r(y) = \log(e + y)$ or $r(y) = \log\log(e + y)$). For instance, consider $r(y) = \exp(y^{\alpha})$ with $\alpha > 0$. The condition \eqref{eq:batty.prop.asm1} is satisfied for any increasing $r$. To verify \eqref{eq:batty.prop.asm2}, observe that integration by parts yields, as $y \to + \infty$
\begin{equation}
	\begin{aligned}
		\int_1^y \exp(\tfrac 32 u^{\alpha}) \, \dd u &= \frac 2 {3\alpha} \left[ \frac{\exp( \tfrac 32 u^{\alpha})}{u^{\alpha-1}} \right]_1^y - \frac{2(1-\alpha)}{3 \alpha} \int_1^y \frac{\exp(\tfrac 32 u^{\alpha})}{u^{\alpha}} \, \dd u
\ls 		\frac{\exp(\frac 32 y^{\alpha})}{y^{\alpha-1}}.
	\end{aligned}
\end{equation}
Hence
\begin{equation}
	\frac{|r'(y)| \int_{0}^{y} r^{\frac32}(u) \dd u}{r^{\frac52}(y)} 
	\ls 
	\frac{y^{\alpha-1} \int_0^y \exp(\frac 32 u^{\alpha}) \, \dd u}{\exp(\frac 32 y^{\alpha})} 
	\ls 1, \quad y \to +\infty.
\end{equation}
Finally, since
\begin{equation}
	\int_1^y \exp(\tfrac 32 u^{\alpha}) \, \dd u 
	\gs 
	\frac{\int_1^y u^{\alpha-1} \exp(\tfrac 32 u^{\alpha}) \, \dd u}{\max\{1,y^{\alpha-1}\}} 
	\gs 
	\frac{\exp(\tfrac 32 y^{\alpha})}{\max\{1,y^{\alpha-1}\}}, \quad y \to +\infty, 
\end{equation}
for the condition \eqref{eq:batty.prop.asm3} we arrive at
\begin{equation}
	\frac{r^{\frac12}(y)}{\int_{0}^{y} r^{\frac32}(u) \dd u} \ls \frac{\max\{1,y^{\alpha-1}\} \exp(\tfrac 12 y^{\alpha})}{\exp(\tfrac 32 y^{\alpha})} = o(1), \quad y \to + \infty.
\end{equation}
\end{example}

\section{Examples}
\label{sec:examples}
We illustrate the behaviour of the norm of the resolvent in several examples where the numerical range, $\Num(\opH)$, and the spectrum, if any, lie in the first quadrant of the complex plane. 
In the sequel we denote
\begin{equation*}
	\Psi(\la) := \| (\opH - \la)^{-1} \|.
\end{equation*}
Recall that we have $\Psi(\la) \le 1 / \dist(\la, \overline{\Num(\opH)})$, $\la \notin \overline{\Num(\opH)}$, thus we focus on the behavior of $\Psi(\la)$ for $\la$ in the first quadrant only.

\subsection{Power-like potentials}
\label{ssec:example.power.V}
Let $\opH = \Dt + i \langle x \rangle^p$, $p > 0$, with $\Dom(\opH) = \WttR \cap \Dom(\langle x \rangle^p)$. It is routine to verify that the assumptions of Theorems~\ref{thm:iR} and \ref{thm:R} (see also the extensions in Sub-sections~\ref{ssec:iR.R}, \ref{ssec:R.Rp}) are satisfied and we thus have
\begin{equation}
	\label{eq:resnorm.japbracket.power}
	\begin{aligned}
		\Psi(i b) &= p^{-\frac23} \| \opA_{1,\frac \pi 2}^{-1} \| b^{-\frac23 (1 - \frac1{p})} (1 + \BigO(b^{-\frac2{p}})) (1 + \BigO(b^{-\frac13 (1 + \frac2{p})}))\\
		&= p^{-\frac23} \| \opA_{1,\frac \pi 2}^{-1} \| b^{-\frac23 (1 - \frac1{p})} (1 + \BigO(b^{-l_p})), \quad b \to +\infty,\\
		\Psi(a) &= 2^{-\frac{p}{p+1}} \| \opA_p^{-1} \| a^{-\frac12 \frac{p}{p+1}} (1 + \BigO(a^{-\frac1{p + 1}})) (1 + \BigO(a^{-m_p}))\\
		&= 2^{-\frac{p}{p+1}} \| \opA_p^{-1} \| a^{-\frac12 \frac{p}{p+1}} (1 + \BigO(a^{-m_p})), \quad a \to +\infty,
	\end{aligned}
\end{equation}
with the Airy operators $\opA_{1,\pi/2} = \Dt + i x$ and $\opA_p = \Nt + |x|^p$ (see \eqref{eq:airyde} and \eqref{eq:A.beta.def}, respectively) and
\begin{equation*}
	l_p := 
	\begin{cases}
		2 / p, &  p \ge 4, \\[1mm] 
		(1 + 2 / p) / 3, &  p \in (0, 4),
	\end{cases} \quad
	m_p := 
	\begin{cases}
		1 / (p + 1), &  p \ge 2, \\[1mm] 
		p / (2p + 2), &  p \in (0, 2);
	\end{cases}
\end{equation*}
note that, in this example, the remainder for $\Psi(a)$ is dominated by $\iota(t_a)$ which is independent of $\eps$.

For $V(x) = i x^{2n}$, $n \in \N$, we find similar formulas with improved remainder term for the real axis (in this case, $\iota(t_a) = 0$ and moreover we can take $\eps = 0$ in \eqref{eq:resnorm.R})
\begin{equation}
	\label{eq:resnorm.even.power}
	\begin{aligned}
		\Psi(i b) &= (2n)^{-\frac23} \| \opA_{1,\frac \pi 2}^{-1} \| b^{-\frac23 (1 - \frac1{2n})} (1 + \BigO(b^{-\frac13 (1 + \frac1{n})})), \quad b \to +\infty,\\
		\Psi(a) &= 2^{-\frac{2n}{2n+1}} \| \opA_{2n}^{-1} \| a^{-\frac{n}{2n+1}} (1 + \BigO(a^{-\frac{n+1}{2n+1}})), \quad a \to +\infty.
	\end{aligned}
\end{equation}

We can also derive estimates for odd potentials $V(x) = i x^{2n-1}$, $n \in \N$, along both the positive and negative parts of the imaginary axis (see \eqref{eq:res.norm.iR.R.posneg} in our closing remarks in Sub-section~\ref{ssec:iR.R}), namely as $b \to +\infty$,
\begin{equation}
	\label{eq:resnorm.odd.power}
	\begin{aligned}
		\Psi(i b) = (2n-1)^{-\frac23} \| \opA_{1,\frac \pi 2}^{-1} \| b^{-\frac23 \frac{2n-2}{2n-1}} (1 + \BigO(b^{-\frac13 \frac{2n+1}{2n-1}})), \quad \Psi(-ib) = \Psi(ib).
	\end{aligned}
\end{equation}

From \eqref{eq:resnorm.japbracket.power}, we note that, for power-like potentials with degree $p > 1$, $\Psi(ib)$ decays progressively faster as $p \to +\infty$ with limit $\Psi(ib) \approx b^{-\frac23}$, the decay rate for $V(x) = i e^{\langle x \rangle}$. As we consider potentials that grow super-exponentially, the asymptotic behaviour of $\Psi(ib)$ changes, and an additional factor (a negative power of $\log b$) comes into play (see Example~\ref{ssec:example.fast.V}). At the other end of the range for $p$, as $p \to 0+$, $p < 1$, we observe the growth rate of $\Psi(ib)$ along the imaginary axis increasing ever faster. The transition from power-like potentials to (more slowly growing) logarithmic ones also determines a change in asymptotics for $\Psi(ib)$, with growth along the imaginary axis becoming exponential (see Example~\ref{ssec:example.slow.V}).

Arguing as in the closing remarks in Sub-section~\ref{sssec:resnorm.adj.iR} (see \eqref{eq:resnorm.iR.levelcurves}), we find the level curves for the resolvent of $\opH$ with potential $V(x) = i x^n$, $n \in \N$. Note that $\rho = n^{-\frac13} b^{-\frac13\frac{n-1}{n}}$ and hence
\begin{equation}
	\label{eq:resnorm.iR.levelcurves.powerlike}
	a_b = \left(\frac{3n}{4}\right)^{\frac23} b^{\frac23(1 - \frac1{n})} \left(\log\left(b^{\frac23(1 - \frac1{n})} \eps^{-1}\right)\right)^{\frac23} (1 + o(1)), \quad b \to +\infty.
\end{equation}
Since we require $\rho = o(1)$, we need $n > 1$, and, for $\Phi_b \approx b^{\frac13\frac{n-4}{n}} = o(1)$, we must have $n < 4$.

We can similarly find the level curves, adjacent to the real axis, for the resolvent of $\opH$ with even potential $V(x) = i x^{2 n}$, $n \in \N$ (refer to Sub-section~\ref{sssec:resnorm.adj.R}). Note that $V_2(t_a) = 2^{\frac{2 n}{2 n + 1}} a^{\frac{n}{2 n + 1}}$ and hence an application of \eqref{eq:resnorm.R.levelcurves} yields
\begin{equation}
	\label{eq:resnorm.R.levelcurves.powerlike}
	b_a = \left(\frac{2 n + 1}{2 n}\right)^{\frac{2 n}{2 n + 1}} a^{\frac{n}{2 n + 1}} \left(\log\left(a^{\frac{n}{2 n + 1}} \eps^{-1}\right)\right)^{\frac{2 n}{2 n + 1}} (1 + o(1)), \quad a \to +\infty.
\end{equation}
Noting that $\Phi_a \approx V_2(t_a) (a^{\frac12} t_a)^{-1} \approx a^{-\frac1{2 n + 1}} = o(1)$, as $a \to +\infty$, and therefore any even monomial is admissible as a potential in this instance.

Two cases of particular interest are the operators with potentials $V(x) = i x^2$ (the Davies operator) and $V(x) = i x^3$ (the imaginary cubic oscillator). They have been studied in detail in the literature using both semi-classical and non-semi-classical methods: see e.g. \cite{Davies-1999-200,Boulton-2002-47,Dencker-2004-57,BordeauxMontrieux-2013,Krejcirik-2019-276} for the Davies example and \cite{Bender-1998-80,BordeauxMontrieux-2013,Sjoestrand-2009,Dondl-2016} for the cubic case. The behaviour of the norm of the resolvent for each of them is illustrated in Fig.~\ref{fig:examples.V} which shows the regions of uniform boundedness of $\Psi(\la)$ described in Sub-section~\ref{ssec:inside.Num} (see \eqref{eq:res.bounded.region.iR} and \eqref{eq:res.bounded.region.R}). Furthermore we observe that the level curves determined by \eqref{eq:resnorm.iR.levelcurves.powerlike} with $n = 2$ and $n = 3$ match those found using semi-classical methods in \cite[Prop.~4.6, Prop.~4.2]{BordeauxMontrieux-2013} and that, as expected, the level curves determined by \eqref{eq:resnorm.iR.levelcurves.powerlike} with $n = 2$ and those determined by \eqref{eq:resnorm.R.levelcurves.powerlike} with $n = 1$ (see \eqref{eq:resnorm.R.levelcurves.davies}) are symmetrical with respect to the bi-sector $y = x$.
\begin{figure}[h]
	\includegraphics[scale=0.65]{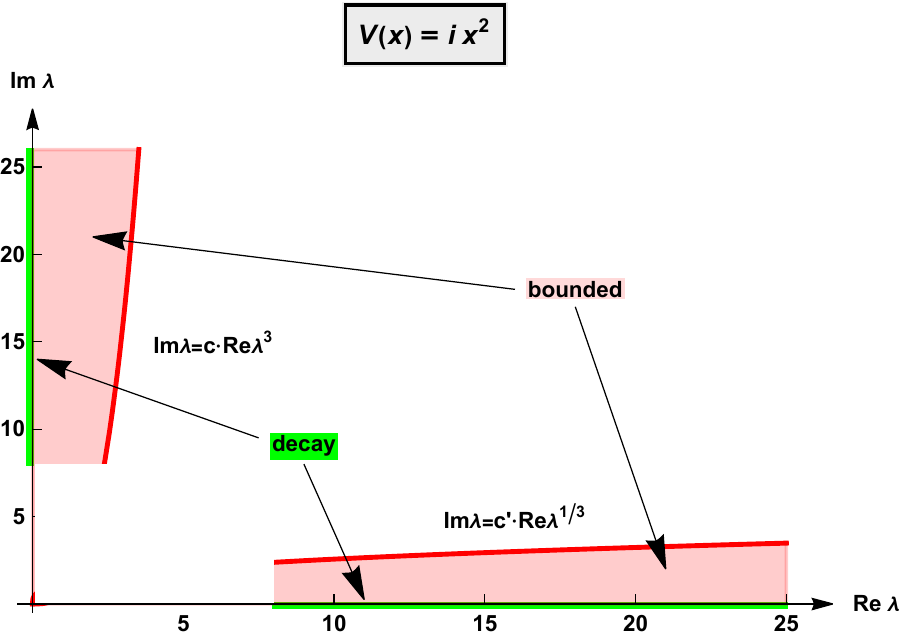}\hfill
	\includegraphics[scale=0.65]{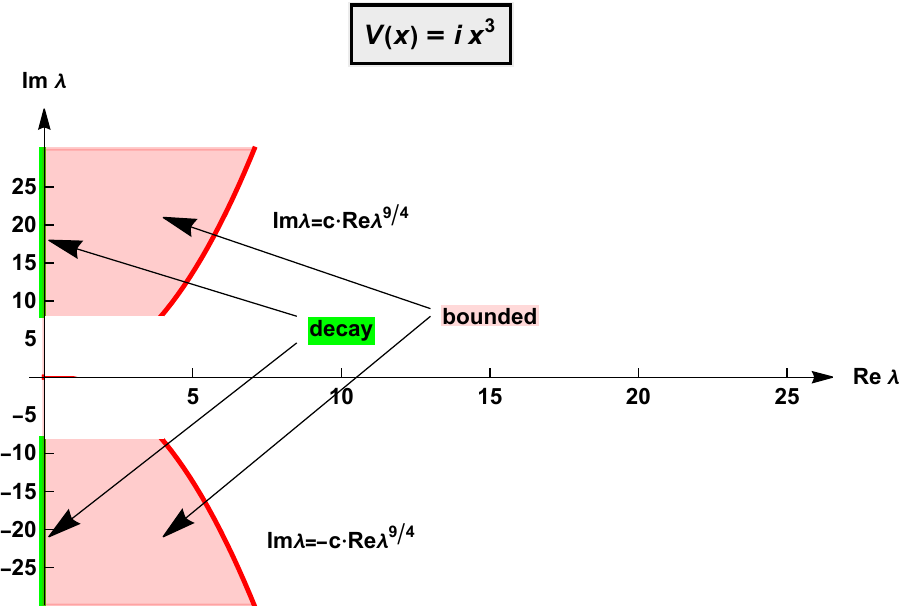}\\[\smallskipamount]
	\includegraphics[scale=0.65]{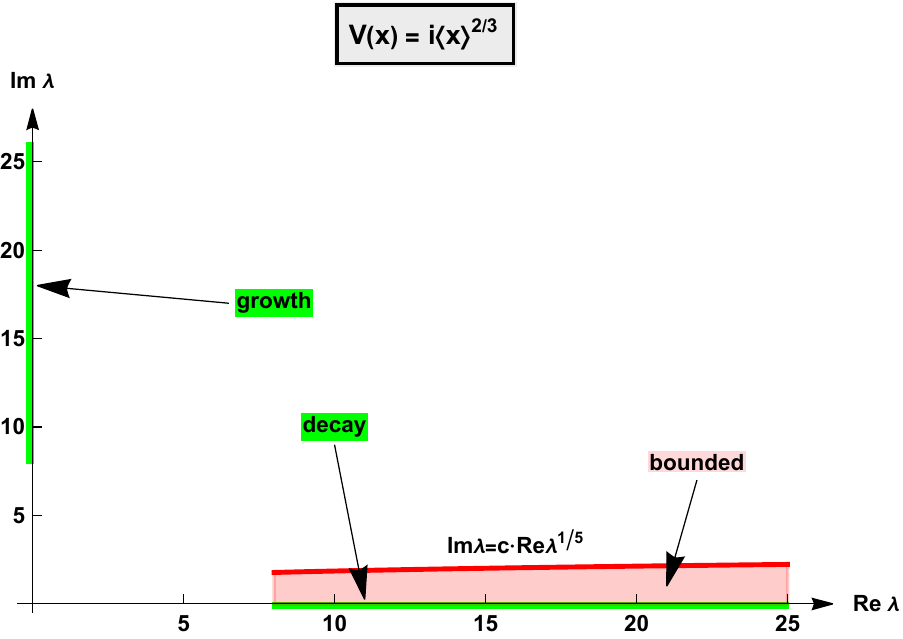}\hfill
	\includegraphics[scale=0.65]{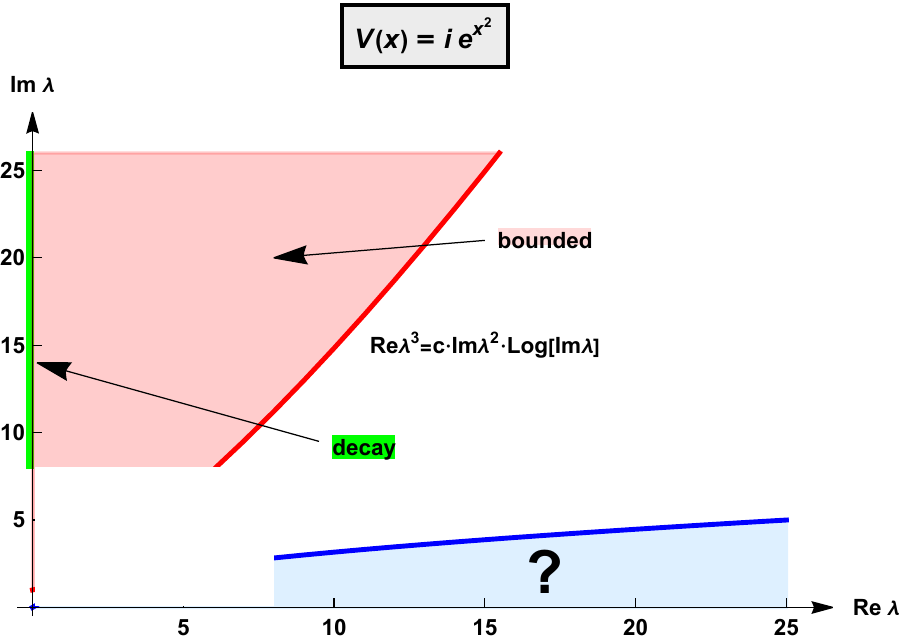}
	\caption{Schematic behaviour of $\Psi(\la)$ for operators with potentials growing at different rates. Corresponding asymptotic estimates are provided in \eqref{eq:resnorm.even.power} with $n = 1$ (top left), \eqref{eq:resnorm.odd.power} with $n = 2$ (top right), \eqref{eq:resnorm.japbracket.power} with $p = 2/3$ (bottom left) and \eqref{eq:fast.V.resnorm.iR} (bottom right). To produce the plots, we have used $\| \opA_{1,\pi/2}^{-1} \| = \| \opA_2^{-1} \| \approx 1.33377$ and $\| \opA_{2/3}^{-1} \| \approx 1.12648$, calculated using NDEigenvalues in Mathematica.}
	\label{fig:examples.V}
\end{figure}

We also show the behaviour of $\Psi(\la)$ for the operator with sub-linear potential $V(x) = i \langle x \rangle^{\frac23}$ in Fig.~\ref{fig:examples.V}, remarking that the completeness of the eigensystem for this operator with Dirichlet boundary conditions in $L^2(\Rplus)$ was proved in \cite{Tumanov-2021-280}.

\subsection{Slowly growing potential}
\label{ssec:example.slow.V}
Let $\opH = \Dt + i \log \langle x \rangle$ with domain $\Dom(\opH) = \WttR \cap \Dom(\log \langle x \rangle)$. Then
\begin{equation*}
	\Psi(i b) = \| \opA_{1,\frac \pi 2}^{-1} \| e^{\frac23 b} (1 + \BigO(e^{-\frac23 b})), \quad b \to +\infty.
\end{equation*}
As in the sub-linear potential case, the fact that $\Psi(\la)$ grows along the imaginary axis leads to an $\eps$-shifted critical curve that intersects it at some $b > 0$.

\subsection{Fast growing potential}
\label{ssec:example.fast.V}
Let $\opH = \Dt + i e^{x^2}$ with $\Dom(\opH) = \WttR \cap \Dom(e^{x^2})$. Then
\begin{equation}
	\label{eq:fast.V.resnorm.iR}\Psi(i b) = 2^{-\frac23} \| \opA_{1,\frac \pi 2}^{-1} \| b^{-\frac23} (\log b)^{-\frac13} (1 + \BigO(b^{-\frac13} (\log b)^{\frac13})), \quad b \to +\infty,
\end{equation}
which is as before illustrated in Fig.~\ref{fig:examples.V}. Since the decay of $\Psi(\la)$ on the imaginary axis is faster than for any polynomial potential, the region for uniform boundedness of $\Psi(\la)$ adjacent to the imaginary axis is correspondingly wider. Note that Theorem~\ref{thm:R} on the behaviour of $\Psi(\la)$ for $\la \in \Rplus$ is not applicable in this case, see also Figure~\ref{fig:examples.V}, and therefore the description of the critical region next to the real axis is currently an open question although  \cite[Eq.~(5.5)]{Krejcirik-2019-276} provides a clue as to what it may look like.

\appendix 
\section{Generalised Airy operator}
\label{sec:genAiry}

We analyse the following first order operator in $\Lt(\R)$ which we refer to as a generalised Airy operator
\begin{equation}
	\label{eq:genairyde}
	\opA = \Nt + W, \quad
	\Dom(\opA) = \{ u \in \Lt(\R) \, : \, -u' + W u \in \Lt(\R) \}.
\end{equation}
\begin{proposition}	\label{prop:A.gen}
	Let $W \in L^{\infty}_{\loc}(\R)$ with $\Re W \ge 0$ a.e.~and let $A$ be as in \eqref{eq:genairyde}. Then
	\begin{enumerate}[\upshape i), wide]
		\item \label{itm:A.maccr} $A$ is densely defined and m-accretive;
		\item $A$ has a compact resolvent if
		\begin{equation}\label{eq:ReW.essinf}
			\lim_{N \to +\infty} \essinf_{|x| \geq N}{\Re W(x)} = +\infty;
		\end{equation}
		\item \label{itm:A.star} the adjoint operator reads
		\begin{equation}
			\opA^* = \Ntp + \ov {W}, \quad
			\Dom(\opA^*) = \big\{ u \in \Lt(\R) \, : \, u' + \ov{W} u \in \Lt(\R) \big\};
		\end{equation}
		\item we have
		\begin{equation}\label{eq:airy.evs}
			\la \in \sigma_p(\opA) \quad \iff \quad \exp \left(\int_{0}^{x} \Re W(t) \, \dd t - \Re \la\, x \right) \in L^2(\R);
		\end{equation}
		hence $\sigma_p(\opA) = \emptyset$ if 
		\begin{equation}\label{eq:ReW.essinf+}
			\lim_{N \to +\infty} \essinf_{x \geq N}{\Re W(x)} = +\infty.
		\end{equation}
	\end{enumerate}
\end{proposition}
\begin{proof}
	\begin{enumerate}[\upshape i), wide]
		\item It is clear that $\CcR \subset \Dom(\opA)$ and therefore that $\opA$ is densely defined. Moreover, a standard cut-off argument, using a sequence $u_n:= \phi(x/n) u$ for $0 \neq \phi \in \CcR$ and any $u \in \Dom(\opA)$, see \eg~\cite[Lem.~3.6]{Krejcirik-2017-221}, shows that 
		\begin{equation}\label{eq:core.A}
			\Core_A:=\{ u \in W^{1,2}(\R) \, : \, \supp u \text{ is bounded}\}
		\end{equation}
		is a core of $A$. Thus for all $u \in \Core_A$, we have $\langle \opA u, u \rangle = -\langle u', u \rangle +  \langle W u, u \rangle$, hence
		\begin{equation*}
			\Re \langle \opA u, u \rangle = \langle \Re W u, u \rangle = \| ( \Re W )^{\frac{1}{2}} u \|^2 >0,
		\end{equation*}
		i.e.~$\opA$ is accretive; moreover, 
		\begin{equation}
			\label{eq:reW.upperbound}
			2 \| \left( \Re W \right)^{\frac{1}{2}} u \|^2 \le \| \opA u \|^2 + \| u \|^2.
		\end{equation}

		For $\la > 0$ and $u \in \Dom(\opA)$, we have 
		\begin{equation}
			\| \left( \opA + \la \right) u \|^2 = \| \opA u \|^2 + \la^2 \| u \|^2 + 2 \la \| \left( \Re W \right)^{\frac{1}{2}} u \|^2,	
		\end{equation}
		thus
		\begin{equation}
			\| u \| \le \frac{1}{\la} \| \left( \opA + \la \right) u \|.
		\end{equation}
		This shows that $\opA + \la$ is injective, that $( \opA + \la )^{-1} : \Ran( \opA + \la ) \rightarrow \Dom(\opA)$ is bounded and that $\|( \opA + \la )^{-1} \| \le 1/\la$. Moreover, $\Ran ( \opA + \la )$ is closed. 
		
		Next we show that $\Ran(\opA+\la)$ is dense in $\Lt(\R)$. Let $f \in \CcR$ and assume that $\supp f \subset [a,b]$ for some $a, b \in \R$, $a < b$. Elementary calculations show that
		\begin{align*}
			u(x) = e^{\int_0^x W(t) \dd t + \la x} \int_x^b f(y) e^{-\int_0^y W(t) \dd t - \la y} \dd y \, \chi_{(-\infty, b)}(x)
		\end{align*}
		solves $-u' + ( W + \la )  u = f$. Furthermore, since $\Re W \ge 0$ and $\la > 0$
		\begin{align*}
			|u(x)| &\le e^{\int_0^x \Re W(t) \dd t + \la x} \, \int_a^b |f(y)| e^{-\int_0^y \Re W(t) \dd t - \la y}\,  \dd y \, \chi_{(-\infty, b)}(x) \\
			&\leq e^{\int_0^x \Re W(t) \dd t + \la x} \, \|f\|_{L^1} \, \chi_{(-\infty, b)}(x),
		\end{align*}
		hence $u \in L^2(\R)$. We have thus shown $\CcR \subset \Ran\left(\opA + \la\right)$, consequently $\Ran\left(\opA + \la\right) = \Lt(\R)$, $-\la \in \rho(\opA)$ and therefore $\opA$ is m-accretive.
		\item The compactness of $(\opA+1)^{-1}$ follows from \eqref{eq:ReW.essinf}, $\Dom(\opA) \subset W^{1,2}_{\rm loc}(\R)$ and \eqref{eq:reW.upperbound} (see \eg~\cite[Sub-secs.~14.2, 5.2]{Helffer-2013-book}).
		\item By simple adjustments of the arguments to prove \ref{itm:A.maccr}, we can show that $\opB := \dd / \dd x  + \ov W$ with the maximal domain $\Dom(\opB) := \{ u \in \Lt(\R) \, : \, u' + \ov{W} u \in \Lt(\R) \}$ is m-accretive. Moreover, for all $u \in \Core_A$ and $v \in \Dom(\opB)$, we have
		\begin{equation*}
			\langle \opA u, v \rangle = \langle -u', v \rangle + \langle W u, v \rangle = \langle u, v' \rangle + \langle u, \overline{W} v \rangle =  \langle u, \opB v \rangle,
		\end{equation*}
		which shows that $ \opB \subset \opA^*$. However, the fact that $\opA$ is m-accretive implies that $\opA^*$ is also m-accretive (see \eg~\cite[Thm.~6.6]{EE}) and therefore it must be the case that $\opB = \opA^*$, as claimed. 
		
		\item If $\la \in \sigma_p(\opA)$, there is $0 \neq u_{\la} \in \Dom(\opA)$ such that
		$
		-u_{\la}' + W u_{\la} -\la u_{\la} = 0.
		$
		Then $u_\la$ must have the form
		$
		u_{\la}(x) = C \exp( \int_{0}^{x} W(t)\, \dd t - \la x)$, $x \in \R$, 
		for some $C \in \C \setminus \{0\}$. Therefore
		\begin{equation*}
			|u_{\la}(x)| = |C| e^{\int_{0}^{x} \Re W(t) \, \dd t - (\Re \la) x},\quad x \in \R,
		\end{equation*}
		from which \eqref{eq:airy.evs} follows. Finally, using \eqref{eq:ReW.essinf+}, we obtain
		\begin{equation}
			\lim_{x \to +\infty} \frac{\int_{0}^{x} \Re W(t) \, \dd t}{x} =  +\infty,
		\end{equation}
		thus no $u_\la$ can be in $L^2(\R)$.
		\qedhere
	\end{enumerate}
	
\end{proof}

%

\subsection{Separation property}

Under more restrictive assumptions on $W$, analogous to \eqref{eq:vgrowth}, the graph norm of $A$ separates.  
\begin{proposition}
	\label{prop:A.sep}
	Let $W \in \LilocR \cap C^1\left( \R\setminus [-x_0,x_0] \right)$, with some $x_0 > 0$, satisfying $\Re W \ge 0$ a.e., and suppose that
	\begin{enumerate}[\upshape (i), wide]
		\item \label{itm:genAiry.separation.prop} there exist $\eps \in (0,1)$ and $M> 0$ such that
		\begin{equation}\label{eq:A.W.sep.asm}
			|\Re W'(x)| \le \eps |\Re W(x)|^2 + M, \quad |x| > x_0;
		\end{equation}
		\item \label{itm:genAiry.im.part} $\Im V$ is relatively bounded w.r.t.~$\Re W$, \ie~there is $C_W \geq 0$ such that
		\begin{equation}\label{eq:ImW.rb}
			|\Im W| \leq C_W (\Re W + 1) \quad \text{a.e.~in } \R.
		\end{equation}
	\end{enumerate}
	Then 
	\begin{equation}\label{eq:A.dom.sep}
		\Dom(\opA) = \Dom(A^*) = W^{1,2}(\R) \cap \Dom(\Re W)
	\end{equation}
	and we have
	\begin{equation}\label{eq:genAiry.graph.norm}	
		\begin{aligned}
			\| \opA u \|^2 + \| u \|^2 &\geq C_A \left(\| u' \|^2 + \| \Re W u \|^2 + \| u \|^2 \right), \quad u \in \Dom(A),
			\\
			\| \opA^* u \|^2 + \| u \|^2 &\geq C_{A^*} \left(\| u' \|^2 + \| \Re W u \|^2 + \| u \|^2 \right), \quad u \in \Dom(A^*);
		\end{aligned}
	\end{equation}
	the constants $C_{A}, C_{A^*}>0$ depend only on $\eps$, $M$, $C_W$ and $\|W \chi_{[-x_0,x_0]}\|_\infty$.
\end{proposition}
\begin{proof}
	Consider $\phi \in C^\infty_c((-2x_0,2 x_0))$ such that $0 \le \phi \le 1$ and $\phi = 1$ on $[-x_0, x_0]$.
	We split $W = W_1 + W_2 :=(1 - \phi) W + \phi W$, where $W_2 \in \LiR$, $W_1 \in C^1(\R)$ and  $\supp W_1 \subset (-\infty, -x_0] \cup [x_0, +\infty)$. Since $W \in \LilocR$ and $W_1' = (1 - \phi)  W' - \phi' W$, the assumption \eqref{eq:A.W.sep.asm} is satisfied also for $W_1$, possibly with a different constant $M'$.
	
	Let $\opA_1$ be the operator determined by \eqref{eq:genairyde} with potential $W_1$. We show that the separation \eqref{eq:A.dom.sep} and \eqref{eq:genAiry.graph.norm} holds for $A_1$. The latter remain valid for $\opA = \opA_1 + W_2$ since $W_2$ is bounded. 

	For $u \in \Core_{A_1}$, see \eqref{eq:core.A} and \eqref{eq:A.dom.sep}, integration by parts yields
	\begin{align*}
		\| \opA_1 u \|^2 &= \| u' \|^2 + \| W_1  u \|^2 - 2 \Re \langle u',  W_1  u\rangle 
		\\
		&= \| u' \|^2 + \| W_1  u \|^2 - 2 ( \Re \langle u', \Re W_1  u\rangle + \Im \langle u',  \Im W_1  u\rangle) \\
		&= \| u' \|^2 + \| W_1 \, u \|^2 + \langle u, \Re W_1'  u\rangle - 2 \Im \langle u',  \Im W_1 u\rangle 
		\\
		&\ge \| u' \|^2 + \| \Re W_1  u \|^2 + \| \Im W_1  u \|^2 - \langle u, \, |\Re W_1'|  u\rangle - 2  \| u' \|  \| \Im W_1  u \|.
	\end{align*}
	Using \eqref{eq:A.W.sep.asm} for $W_1$ (see remarks above), Young inequality with $\delta \in (0,1)$ and the assumption~\eqref{eq:ImW.rb} in the second step, we arrive at 
	\begin{align*}
		\| \opA_1  u \|^2 &\ge (1 - \delta) \| u' \|^2 + (1 - \eps) \| \Re W_1 u \|^2 - (\delta^{-1} - 1) \| \Im W_1  u \|^2 - M' \| u \|^2 
		\\
		&\ge 
		(1 - \delta) \| u' \|^2 + ( 1 - \eps - C_W' (\delta^{-1} - 1) )  \| \Re W_1  u \|^2 
		\\
		&\quad 
		- ( M' + C_W' ( \delta^{-1} - 1 ))  \| u \|^2.
	\end{align*}
	We select $\delta$ so that $C_W'/\left(1 - \eps + C_W'\right) < \delta < 1$, thus $1 - \eps - C_W' \left(\delta^{-1} - 1\right) > 0$. Therefore for all $u \in \Core_{A_1}$ (and hence for all $u \in \Dom(A_1)$)
	\begin{equation}
		\label{eq:genAiry.A1.graph.norm}	
		\| \opA_1  u \|^2 + \| u \|^2 \gs \| u' \|^2 + \| \Re W_1 \, u \|^2 + \| u \|^2.
	\end{equation}
	Since the opposite inequality is immediate, we conclude with \eqref{eq:A.dom.sep} for $A_1$ and hence for $A$ since $W_2$ is bounded.
	The reasoning for $A^*$ is completely analogous.
\end{proof}

\bibliography{references}
\bibliographystyle{acm}

\end{document}

%% file: commands.tex
\usepackage{amsmath,amsthm, amssymb}
\usepackage{mathrsfs}
\usepackage[shortlabels]{enumitem}
\usepackage{graphicx}
\usepackage[font=small]{caption}
\usepackage{xcolor}
\usepackage{url}

\newcommand{\N}{\mathbb{N}}
\newcommand{\R}{{\mathbb{R}}}
\newcommand{\C}{{\mathbb{C}}}

\newcommand{\dd}{{{\rm d}}}

\newcommand{\odd}{\overline{\dd}}

\newcommand{\cf}{\emph{cf.}}
\newcommand{\ie}{{\emph{i.e.}}}
\newcommand{\eg}{{\emph{e.g.}}}

\newcommand{\ov}{\overline}

\newcommand{\la}{\lambda}

\newcommand{\eps}{\varepsilon}

\newcommand{\essinf}{\operatorname*{ess \,inf}}

\newcommand{\Dom}{{\operatorname{Dom}}}

\newcommand{\Ran}{{\operatorname{Ran}}}

\renewcommand{\Re}{\operatorname{Re}}
\renewcommand{\Im}{\operatorname{Im}}
\newcommand{\dist}{\operatorname{dist}}
\newcommand{\sgn}{\operatorname{sgn}}
\newcommand{\supp}{\operatorname{supp}}

\newcommand{\Num}{\operatorname{Num}}

\newcommand{\loc}{\mathrm{loc}}
\newcommand{\BigO}{\mathcal{O}}

\newcommand{\Os}{{\operatorname{Os-}}}
\newcommand{\rad}{{\operatorname{rad}}}
\newcommand{\opH}{H}

\newcommand{\opwV}{\widehat V}

\newcommand{\opA}{A}
\newcommand{\opB}{B}

\newcommand{\opI}{I}

\newcommand{\opK}{K}
\newcommand{\opP}{P}
\newcommand{\opQ}{Q}
\newcommand{\opR}{R}
\newcommand{\opS}{S}
\newcommand{\opT}{T}
\newcommand{\opU}{U}

\newcommand{\opOP}{\operatorname{OP}}
\newcommand{\Core}{{\mathcal{C}}}

\newcommand{\Rplus}{\R_+}
\newcommand{\Rminus}{\R_-}

\newcommand{\CiR}{{C^{\infty}(\R)}}
\newcommand{\SchwR}{{\sS(\R)}}

\newcommand{\intR}{\int_{\R}}

\newcommand{\Rd}{\mathbb{R}^d}

\newcommand{\Lt}{{L^2}}

\newcommand{\LiR}{{L^{\infty}(\R)}}

\newcommand{\LiOm}{{L^{\infty}(\Omega)}}

\newcommand{\LiRplus}{{L^{\infty}(\Rplus})}
\newcommand{\LilocR}{{L^{\infty}_{\rm loc}(\R})}

\newcommand{\LilocRplusclosed}{{L^{\infty}_{\rm loc}\left(\overline{\Rplus}\right)}}

\newcommand{\CcR}{{C_c^{\infty}(\R)}}
\newcommand{\CcOm}{{C_c^{\infty}(\Omega)}}

\newcommand{\WttR}{{W^{2,2}(\R)}}

\newcommand{\Dt}{-\partial_x^2}
\newcommand{\Dtr}{-\partial_r^2}
\newcommand{\Dtp}{\partial_x^2}
\newcommand{\Nt}{-\partial_x}

\newcommand{\Ntp}{\partial_x}

\newcommand{\DD}{\Delta_{\rm D}}

\newcommand{\ls}{\lesssim 	}
\newcommand{\gs}{\gtrsim}

\theoremstyle{plain}

\newtheorem{theorem}{Theorem}[section]
\newtheorem{lemma}[theorem]{Lemma}

\newtheorem{proposition}[theorem]{Proposition}

\theoremstyle{definition}
\newtheorem{example}[theorem]{Example}

\newtheorem{asm-sec}[theorem]{Assumption}

\newcommand\cA{\mathcal A}
\newcommand\cB{\mathcal B}

\newcommand\cH{\mathcal H}

\newcommand\cI{\mathcal I}

\newcommand\cS{\mathcal S}

\newcommand\cX{\mathcal X}

\newcommand\sF{\mathscr F}
\newcommand\sS{\mathscr S}